\documentclass[11pt]{amsart}

\usepackage{amsmath}
\usepackage{amssymb}
\usepackage{amsfonts}
\usepackage{amsthm}
\usepackage{mathrsfs}
\usepackage{fullpage}
\usepackage[all]{xy}
\usepackage{rotating}
\usepackage{enumerate}

\newcommand{\sd}{d_\mathrm{Sen}}
\newcommand{\htd}{d_\mathrm{HT}}
\newcommand{\drd}{d_\mathrm{dR}}

\newcommand{\bdr}{B_{\mathrm{dR}}}

\newcommand{\bht}{B_{\mathrm{HT}}}
\newcommand{\bdrk}[1]{\bdr^+/t^{#1}\bdr^+}

\newcommand{\br}{\operatorname{br}}

\newcommand{\Frac}{\operatorname{Frac}}

\newcommand{\End}{\operatorname{End}}

\newcommand{\supp}{\operatorname{Supp}}
\newcommand{\spec}{\operatorname{Spec}}
\newcommand{\spm}{\operatorname{Spm}}
\newcommand{\gal}{\operatorname{Gal}}

\newcommand{\rank}{\operatorname{rank}}

\newcommand{\ind}{\operatorname{Ind}}

\newcommand{\tr}{\operatorname{Tr}}
\newcommand{\im}{\operatorname{im}}
\newcommand{\coim}{\operatorname{coim}}

\newcommand{\coker}{\operatorname{coker}}

\newcommand{\GL}{\mathrm{GL}}

\newcommand{\surj}{\twoheadrightarrow}
\newcommand{\inj}{\hookrightarrow}
\newcommand{\set}[1]{\left\lbrace#1\right\rbrace}

\newcommand{\ol}[1]{\overline{#1}}

\newcommand{\ses}[3]{0\rightarrow #1 \rightarrow #2 \rightarrow #3 \rightarrow 0}
\newcommand{\iso}{\stackrel{\sim}{\rightarrow}}
\renewcommand{\hat}[1]{\widehat{#1}}
\renewcommand{\(}{\left(}
\renewcommand{\)}{\right)}

\newtheorem{thm}{Theorem}[section]
\newtheorem{thmn}{Main Theorem}

\newtheorem{lem}[thm]{Lemma}
\newtheorem{prop}[thm]{Proposition}
\newtheorem{coro}[thm]{Corollary}
\newtheorem{claim}[thm]{Claim}

\theoremstyle{remark}
\newtheorem{remark}[thm]{Remark}
\theoremstyle{definition}
\newtheorem{defin}[thm]{Definition}

\title{Interpolating Hodge-Tate and de Rham Periods}
\author{Shrenik Shah}
\address{Department of Mathematics, Columbia University, New York, NY 10027}
\email{snshah@math.columbia.edu}
\thanks{S.S.\ has been supported through NSF grants DGE-1148900, DMS-1401967, and the Department of Defense (DoD) National Defense Science \& Engineering Graduate Fellowship (NDSEG) Program.}

\begin{document}

\begin{abstract}

We study the interpolation of Hodge-Tate and de Rham periods over rigid analytic families of Galois representations.  Given a Galois representation on a coherent locally free sheaf over a reduced rigid space and a bounded range of weights, we obtain a stratification of this space by locally closed subvarieties where the Hodge-Tate and bounded de Rham periods (within this range) as well as 1-cocycles form locally free sheaves.  We also prove strong vanishing results for higher cohomology.  Together, these results give a simultaneous generalization of results of Sen, Kisin, and Berger-Colmez.  The main result has been applied by Varma in her proof of geometricity of Harris-Lan-Taylor-Thorne Galois representations as well as in several works of Ding.

\end{abstract}

\maketitle

\section{Introduction} \label{sec:ipintro}

\subsection{Motivation}

Fontaine introduced period rings including $\bht$, the ring of Hodge-Tate periods, and $\bdr$, the ring of de Rham periods, in order to algebraically detect and study geometric properties of $p$-adic representations.  Let $B$ be a ring of periods.  If $K$ is a finite extension of $\mathbf{Q}_p$ and $V$ is a finite dimensional vector space over $\mathbf{Q}_p$ with a continuous linear action of $G_K$, one defines $D_B(V) = H^0(G_K,B \otimes_{\mathbf{Q}_p} V)$ to be the $B^{G_K}$-vector space of $B$-periods of $V$.  For the ring $B=\bdr$, we have $B^{G_K} = K$.

The dimension $\dim_K D_B(V)$ is conjecturally involved in detecting whether $V$ may be a restriction of a global Galois representation ``arising from geometry,'' i.e.\ appearing inside the \'etale cohomology of a variety over a number field.  It is therefore important to be able to calculate $\dim_K D_B(V)$.  A result of Tsuji achieves this for $V$ arising from geometry, but in many cases, the representation $V$ is constructed via the use of congruences.  In such situations, $V$ can often be realized as the specialization of a $p$-adic family.  Here, we study the variation of $\dim_K D_B(V)$ as $V$ moves in such a family.

We work with a $\mathbf{Q}_p$-Banach algebra $\mathfrak{R}$ and a finite free module $\mathfrak{M}$ over it, equipped with a continuous $\mathfrak{R}$-linear action of $G_K$.  We study $D_B(\mathfrak{M}) = H^0(G_K,B \widehat{\otimes}_{\mathbf{Q}_p} \mathfrak{M}),$ which has the structure of a module over $B^{G_K} \otimes_{\mathbf{Q}_p} \mathfrak{R}$.  (Here $\hat{\otimes}$ denotes completed tensor product.)

The first understanding of $D_B(\mathfrak{M})$ was achieved for $B = \bht$ in the seminal work of Sen \cite{sen,sen2}.  Sen attaches $p$-adically varying Hodge-Tate-Sen weights to the family $\mathfrak{M}$.  These are generalized eigenvectors of an endomorphism $\phi$, called the Sen operator, which acts on a finite $\mathfrak{R}$-module $\mathfrak{E}$.  Since $\phi$ is not necessarily semi-simple, a continuous $G_K$-representation on a $\mathbf{Q}_p$-vector space $V$ with a multiple integral Hodge-Tate-Sen weight $k$ need not have a multiple Hodge-Tate weight $k$.

We work instead with $B = t^k\bdrk{\ell}$ for $k < \ell \in \mathbf{Z}$, a ``bounded'' de Rham period ring.  If the Hodge-Tate-Sen weights do not vary or if considering a specialization to a maximal point of $\mathfrak{R}$, we can deduce statements for $B= \bdr$ as well.

The primary impetus for this study is the growing need for a refined understanding of the variation of periods in families.  Our main results are suited to, for example, studying the Galois representations constructed by Harris, Lan, Taylor, and Thorne \cite{hltt} -- this application has already been carried out by Varma \cite{varma}, who puts their representations into $p$-adic families.  Ding \cite{ding1,ding2,ding3} has studied partially de Rham families using our results in order to prove Breuil's locally analytic socle conjecture for $\GL_2$ as well as formulas and identifications for $\mathcal{L}$-invariants of Fontaine-Mazur and Breuil.

Possible future applications could include generalizations to higher rank groups of the proof of properness of the eigencurve by Diao-Liu \cite{dl}, which compared variation of a de Rham period and a crystalline period over the eigencurve.  Another is to extensions constructed by specializing a $p$-adic family; these arise, for example, in the work of Skinner and Urban \cite{su}.  We develop one method for proving potential semi-stability of an extension of a representation $V$ by a twist of its dual when this extension arises as the specialization of a $p$-adic family (see Theorem \ref{thm:vhtdr}).

We are also interested in the geometry of families, including the global eigenvariety of a reductive group $G$ over $\mathbf{Q}$.   One interpretation of the Hodge-Tate-Sen theory is that the weight space for $\mathfrak{X}$ is intrinsic to the $G_K$-representation on $\mathfrak{M}$.  Our results refine this statement, implying that for each choice of some fixed Hodge-Tate-Sen weights, there is a stratification of the corresponding subspace of $\mathfrak{X}$ defined entirely in terms of existence of de Rham periods.

\subsection{Related work}

Interpolation results of this sort have already been essential to work on the Langlands program.  Berger and Colmez \cite{bc} interpolate periods when \emph{all} of the Hodge-Tate-Sen weights are fixed.  Their work is used by Chenevier and Harris \cite{ch} to establish local-global compatibility of the Langlands correspondence at $p$ for $p$-adic Galois representations that are constructed via $p$-adic interpolation rather than in the cohomology of a Shimura variety.  Kisin \cite{kisin} proved an important interpolation result for a single Hodge-Tate, de Rham, or crystalline period of fixed Hodge-Tate-Sen weight, where the remaining weights are required to vary.  His work has been applied to prove local-global compatibility statements in works of Skinner, Jorza, Luu, and the author \cite{skinner,jorza,luu,shahdiss}.

Liu proved an interpolation theorem for multiple semistable periods using different methods, which implies the interpolation of the corresponding de Rham periods in that setting \cite{rliu}.  However, this result does not allow one to interpolate de Rham periods that do not come from semistable periods, except in the aforementioned bounded weight situation of Berger and Colmez \cite{bc}.

Bellovin \cite{bell} independently proved strong results regarding $p$-adic families of $(\varphi,\Gamma)$-modules in her dissertation, which we learned about after posting a draft of this article.  She proves base change results that include a statement analogous to Proposition \ref{prop:dr1cocycle} below, and obtains the stratification given in Theorem \ref{thm:introstratfam}.  Neither our intermediate nor our final results are exactly equivalent, as the passage to $(\varphi,\Gamma)$-modules is not an equivalence of categories in the family setting (c.f.\ \cite{kliu}).  Although some aspects of our strategies are similar, the methodologies for many crucial parts of the argument are completely different.  Moreover, the results on higher cohomology in Section \ref{sec:highercoh} are completely distinct.

In fact, there is a long line of works on families of $(\varphi,\Gamma)$-modules \cite{kliu,pottharst,liu2,hellmann,kpx,rliu}.  Bellovin's results in \cite{bell} extend earlier works of Kedlaya-Liu \cite{kliu} and Pottharst \cite{pottharst}.  Another main emphasis was constructing a global triangulation \cite{liu2,hellmann,kpx}, which primarily involves globalizing a crystalline period -- this is substantially more difficult than the analogous results for Hodge-Tate and de Rham periods given in Section \ref{sec:stratfam} below.  The aforementioned work of Liu \cite{rliu} used the results of \cite{kpx} to obtain the interpolation of any number of crystalline (and even semistable) periods in families, which together with the results here and by Bellovin \cite{bell} allow one to interpolate any of Fontaine's periods in $p$-adic families.

\subsection{Statement and discussion of main results}

For the finite flat interpolation of Hodge-Tate and bounded de Rham periods, our work generalizes both Kisin's work and Berger-Colmez; we allow any number of fixed Hodge-Tate-Sen weights and any number of varying weights.  More precisely, in Theorem \ref{thm:stratfam2}, we provide for any bounded interval $[i,j]$ a decomposition $\mathfrak{X} = \coprod \mathfrak{L}$ of the reduced space $\mathfrak{X}$ into locally closed strata $\pi_\mathfrak{L}:\mathfrak{L} \inj \mathfrak{X}$ so that $D_B(\pi_\mathfrak{L}^*\mathfrak{M})$ is a coherent locally free sheaf if $B$ is bounded in this range.  If $\mathfrak{X}$ is a non-reduced affinoid and we interpolate a de Rham period for every fixed Hodge-Tate-Sen weight, we obtain finite flatness after localizing to fix multiplicities (as Kisin does).

We also study specializations of the family.  This presents some serious difficulties in general. For example, Mazur and Wiles \cite{mw} have studied cases where a family of de Rham representations specialize to one that does not even have a semisimple Sen operator; it is this kind of behavior that we avoid by localization in the result just mentioned.  However, in Theorem \ref{thm:dr1cocycle}, we construct Hodge-Tate and de Rham periods at every specialization that was originally removed by this localization.  Mazur and Wiles' example demonstrates that our construction gives an optimal bound on the number of periods of the specialization in both the Hodge-Tate and de Rham cases.  The technical heart of this paper is in Section \ref{subsec:special}, and our results on specialization hinge on the arguments there.

The first main result of this paper is Theorem \ref{thm:drcase}.  We provide here a simplified statement that includes this result together with Propositions \ref{prop:nicespecial}, \ref{prop:anyspecial}, and \ref{prop:drspecial}.  See Definition \ref{defin:stabdens} for the notion of stable density, which generalizes Zariski density to the setting of non-reduced rings.  See Remark \ref{remark:topology} for the definitions of the topology on $\bdr$ and of $\otimes_{\mathbf{Q}_p}$ for non-Fr\'echet spaces.  Note that much of the following holds even if $\mathfrak{R}$ is only a Noetherian $\mathbf{Q}_p$-Banach algebra; see Section \ref{sec:nonreducedbase} for precise statements.

\begin{thmn} \label{thm:mainthm}

Fix positive integers $n$ and $k$ and a multiset $\set{w_1, \dots, w_n}$ of integers in the interval $[0,k-1]$.  Let $S(T) = \prod_{j=1}^n (T+w_j)$.  Let $K$ be a finite extension of $\mathbf{Q}_p$, let $\mathfrak{R}$ be an affinoid $\mathbf{Q}_p$-algebra, and let $\mathfrak{M}$ be a finitely generated free $\mathfrak{R}$-module equipped with a continuous $G_K$-action.  Suppose that the Sen polynomial $P(T)$ of $\mathfrak{M}$ factors as $P(T) = Q(T)S(T)$, and define $Q_k=\prod_{j=0}^{k-1}Q(-j)$.  Assume that $\set{\xi_i}_{i \in I}$ is a set of maps $\xi_i: \mathfrak{R}\rightarrow \mathfrak{R}_i$ of $\mathbf{Q}_p$-Banach algebras, where $\mathfrak{R}_i$ is local Artinian of finite dimension over $\mathbf{Q}_p$, with the following properties.
\begin{enumerate}[{\normalfont (i)}]
  \item For $i \in I$, $\dim_K H^0(G_K,\bdrk{k} \otimes_{\mathbf{Q}_p} (\mathfrak{M} \otimes_{\mathfrak{R}} \mathfrak{R}_i))=n\dim_{\mathbf{Q}_p} \mathfrak{R}_i$.
  \item For $i \in I$, the image of $Q_k$ in $\mathfrak{R}_i$ is a unit.
  \item The collection $\set{\xi_i}$ is stably dense in $\mathfrak{R}$.
\end{enumerate}
We then have the following.
\begin{enumerate}[{\normalfont (a)}]
	\item For $r \in \set{0,1}$, the $K \otimes_{\mathbf{Q}_p} \mathfrak{R}_{Q_k}$-module $H^r(G_K,\bdrk{k} \hat{\otimes}_{\mathbf{Q}_p} \mathfrak{M})_{Q_k}$ is finite flat of rank $n$.
	\item For a map $\xi:\mathfrak{R} \rightarrow \mathfrak{R}'$ of affinoid $\mathbf{Q}_p$-algebras, the kernel and cokernel of the map
\[H^r(G_K,\bdrk{k} \hat{\otimes}_{\mathbf{Q}_p} \mathfrak{M}) \otimes_{\mathfrak{R}} \mathfrak{R}' \rightarrow H^r(G_K,\bdrk{k} \hat{\otimes}_{\mathbf{Q}_p} (\mathfrak{M} \otimes_{\mathfrak{R}} \mathfrak{R}'))\]
are killed by a power of $\xi(Q_k)$ for $r \in \set{0,1}$.  If $\xi(Q(-j))$ is a unit for $j\notin [0,k-1]$, then $H^r(G_K,\bdr \hat{\otimes}_{\mathbf{Q}_p} (\mathfrak{M} \otimes_{\mathfrak{R}} \mathfrak{R}'))$ and $H^r(G_K,\bdr^+ \hat{\otimes}_{\mathbf{Q}_p} (\mathfrak{M} \otimes_{\mathfrak{R}} \mathfrak{R}'))$ are finite flat of rank $n$ over $K \otimes_{\mathbf{Q}_p} \mathfrak{R}'$ for $r \in \set{0,1}$.
\end{enumerate}

Let $\xi':\mathfrak{R} \rightarrow \mathfrak{R}'$ be a map of $\mathbf{Q}_p$-Banach algebras with $\mathfrak{R}'$ local Artinian of finite $\mathbf{Q}_p$-dimension.  If either $\mathfrak{R}'$ is a field or $\xi'(Q_k)$ is a unit, then for $B \in \set{\bdr,\bdr^+,\bdrk{k}}$ we have
\begin{align*}
	\dim_K H^0(G_K,\bdr \otimes_{\mathbf{Q}_p} (\mathfrak{M}\otimes_\mathfrak{R} \mathfrak{R}')) &\ge \dim_K H^0(G_K,\bdr^+ \otimes_{\mathbf{Q}_p} (\mathfrak{M}\otimes_\mathfrak{R} \mathfrak{R}'))\\
	\ge \dim_K H^0(G_K,\bdrk{k} \otimes_{\mathbf{Q}_p} (\mathfrak{M}\otimes_\mathfrak{R} \mathfrak{R}')) &\ge n\dim_{\mathbf{Q}_p} \mathfrak{R}'\end{align*}
\[\textrm{and}\quad \dim_K H^0(G_K,B \otimes_{\mathbf{Q}_p} (\mathfrak{M}\otimes_\mathfrak{R} \mathfrak{R}')) = \dim_K H^1(G_K,B \otimes_{\mathbf{Q}_p} (\mathfrak{M}\otimes_\mathfrak{R} \mathfrak{R}')).\]

\end{thmn}

To illustrate this result, we explain Varma's study \cite{varma} of the Galois representations $\rho_\pi: G_F \rightarrow \GL_n(E)$ associated to regular cuspidal algebraic automorphic representations $\pi$ on $\GL_{n/F}$ by Harris, Lan, Taylor, and Thorne \cite{hltt}, where $F$ is a CM field and $E$ is a finite extension of $\mathbf{Q}_p$.  Varma has constructs a family $\mathfrak{M}$ with a dense set of specializations to de Rham representations as well as a specialization to the Galois representation $\rho_\pi\oplus\rho_\pi^{c\vee}(k)$, where $(\cdot)(k)$ denotes a Tate twist and $c$ denotes complex conjugation.  Moreover, in $\mathfrak{M}$, she is able to fix the Hodge-Tate-Sen weights associated to the factor $\rho_\pi$ and have $k$ be such that the weights of $\rho_\pi^{c\vee}(k)$ are negative while those of $\rho_\pi$ are nonnegative.  (Our convention is that $\mathbf{Q}_p(1)$ has the Hodge-Tate weight $-1$.)  In this setting, Theorem \ref{thm:drcase} applies, and using Proposition \ref{prop:nicespecial}, one shows that $\dim_{F} D_{\bdr^+}(\rho_\pi \oplus \rho_\pi^{c\vee}(k))=n[E:\mathbf{Q}_p]$ has exactly the expected dimension.

If $\mathfrak{R}$ is reduced, one can think of Main Theorem \ref{thm:mainthm} as showing that if there are $n$ fixed Hodge-Tate-Sen weights, having $n$ de Rham periods defines a closed condition on $\spm\mathfrak{R}$.  Requiring that these are the same number turns out to be unnecessarily restrictive; one could imagine trying to interpolate $m$ Hodge-Tate periods of weight 0 even though $n > m$ Hodge-Tate weights 0 are fixed, or a similar statement in the de Rham setting.

To achieve this, we define a \emph{de Rham datum} $\mathbf{D} = (\Omega, \Delta)$ in Section \ref{subsec:drdatumdef}, where $\Omega:\mathbf{Z} \rightarrow \mathbf{Z}_{\ge 0}$ and $\Delta:\mathbf{Z} \times \mathbf{Z} \rightarrow \mathbf{Z}_{\ge 0}$ keep track of the number of fixed Hodge-Tate-Sen weights and the generic dimension of $H^0(G_K,t^k\bdrk{\ell} \otimes_{\mathbf{Q}_p} \ol{\mathfrak{M}}_x)$, respectively.  (We write $\ol{\mathfrak{M}}_x = \mathfrak{M}_x \otimes_{\mathfrak{O}_{\mathfrak{X},x}} \kappa(x)$, where $\mathfrak{O}_\mathfrak{X}$ is the structure sheaf, the subscript $x$ on a sheaf denotes localization at $x$, and $\kappa(x)$ is the residue field.)   The definition encodes the natural conditions that are satisfied by these quantities. We say $\mathbf{D} \ge \mathbf{D}'$ if this inequality holds pointwise for $\Omega$ and $\Delta$.  We show that on an irreducible rigid space $\mathfrak{X}$ with a coherent locally free module $\mathfrak{M}$ equipped with a continuous $G_K$ action, there is a naturally associated de Rham datum that can be detected locally.

Associated to $\mathbf{D}$ are the Hodge-Tate dimension in the interval $(k,\ell)$, given by $\htd^{(k,\ell)}(\mathbf{D})=\sum_{i=k}^{\ell-1} \Delta(i,i+1)$, and the de Rham dimension, given by $\drd(\mathbf{D})=\max_{k,\ell \in \mathbf{Z}}\Delta(k,\ell)$.  We also write $\supp(\mathbf{D})$ for the smallest interval containing the support of $\Omega$.  A summary of Theorems \ref{thm:stratfam1}, \ref{thm:stratfam2}, and \ref{thm:stratfam3} appears below.  We write $\mathbf{D}_{\mathfrak{M},x}$ for the de Rham datum of the representation $\ol{\mathfrak{M}}_x$ over the base $\spm\kappa(x)$.  By ``thickened geometric point'', we mean a local Artinian $\mathbf{Q}_p$-Banach algebra of finite $\mathbf{Q}_p$-dimension.

\begin{thmn} \label{thm:introstratfam}

Suppose that $\mathfrak{X}$ is a reduced rigid space over $\mathbf{Q}_p$ and $\mathfrak{M}$ is a coherent locally free sheaf of $\mathfrak{O}_\mathfrak{X}$-modules equipped with a homomorphism $G_K \rightarrow \End_{\mathfrak{O}_\mathfrak{X}} \mathfrak{M}$ that is continuous when restricted to any affinoid.

For any de Rham datum $\mathbf{D}$, the points $x \in \mathfrak{X}$ such that $\mathbf{D}_{\mathfrak{M},x} \ge \mathbf{D}$ form a closed analytic subvariety $\mathfrak{S}_\mathbf{D}$.  If $\mathbf{D} \ge \mathbf{D}'$, then $\mathfrak{S}_\mathbf{D} \subseteq \mathfrak{S}_{\mathbf{D}'}$.

For any interval $[i,j] \subseteq \mathbf{Z}$, there is a decomposition of $\mathfrak{X}$ into a finite disjoint union of locally closed subspaces $\pi_{\mathbf{D}}^{[i,j]}: \mathfrak{L}_{\mathbf{D}}^{[i,j]} \inj \mathfrak{X}$ indexed by de Rham data $\mathbf{D}=(\Omega,\Delta)$ with $\supp(\mathbf{D})\subseteq [i,j]$.  Let $\mathfrak{M}_\mathbf{D}^{[i,j]} = \pi_{\mathbf{D}}^{[i,j]*}\mathfrak{M}$.  If $i \le k < \ell \le j+1$, there is a coherent locally free sheaf $\mathfrak{H}_{(k,\ell)}^r(\mathfrak{M}_\mathbf{D}^{[i,j]})$ such that for any affinoid subdomain $\mathfrak{U}\subseteq \mathfrak{L}_\mathbf{D}^{[i,j]}$, there is a canonical isomorphism $\mathfrak{H}_{(k,\ell)}^r(\mathfrak{M}_\mathbf{D}^{[i,j]})(\mathfrak{U}) \cong H^r(G_K,t^k\bdrk{\ell} \hat{\otimes}_{K,\sigma} \mathfrak{M}_\mathbf{D}^{[i,j]} (\mathfrak{U})).$  Moreover, the formation of $\mathfrak{H}_{(k,\ell)}^r(\mathfrak{M}_\mathbf{D}^{[i,j]})$ is compatible with pullback along any morphism $\pi:\mathfrak{Y} \rightarrow \mathfrak{L}_\mathbf{D}^{[i,j]}$ of reduced rigid spaces.

For any thickened geometric point $\xi:\mathfrak{x} \rightarrow \mathfrak{L}_{\mathbf{D}}^{[i,j]}$, where $\mathfrak{x} = \spm \mathfrak{R}_\mathfrak{x}$, we have
\begin{enumerate}[{\normalfont (a)}]
	\item $H^r(G_K,t^k\bht^+/t^\ell\bht^+ \otimes_{\mathbf{Q}_p}\xi^*\mathfrak{M}_{\mathbf{D}}^{[i,j]})$ is flat over $K \otimes_{\mathbf{Q}_p} \mathfrak{R}_\mathfrak{x}$ of rank $\htd^{(k,\ell)}(\mathbf{D})$,
  \item $H^r(G_K,t^k\bdrk{\ell} \otimes_{\mathbf{Q}_p} \xi^*\mathfrak{M}_{\mathbf{D}}^{[i,j]})$ is flat over $K \otimes_{\mathbf{Q}_p} \mathfrak{R}_\mathfrak{x}$ of rank $\Delta(k,\ell)$, and
  \item $\dim_K H^r(G_K,\bdr \otimes_{\mathbf{Q}_p} \xi^*\mathfrak{M}_{\mathbf{D}}^{[i,j]}) \ge \dim_K H^{r^+}(G_K,t^i\bdr^+ \otimes_{\mathbf{Q}_p} \xi^*\mathfrak{M}_{\mathbf{D}}^{[i,j]}) \ge \drd(\mathbf{D})\dim_{\mathbf{Q}_p} \mathfrak{R}_\mathfrak{x}$
\end{enumerate}
for all $i \le k < \ell \le j + 1$ and $r,r^+ \in \set{0,1}$.

\end{thmn}

We can recover the reduced case of Main Theorem \ref{thm:mainthm} by setting $\mathbf{D}$ to be a \emph{full} de Rham datum, which is one for which $\Omega(i) = \Delta(i,i+1)$ and $\Delta(i,k) = \Delta(i,j) + \Delta(j,k)$ for all $i \le j \le k \in \mathbf{Z}$.

In Section \ref{sec:highercoh}, we prove strong vanishing results for higher continuous Galois cohomology of modules of Hodge-Tate or de Rham periods in a family.  The key ingredient, Theorem \ref{thm:hkvanish}, generalizes a result of Sen \cite[Proposition 2]{sen}.  The main result, Theorem \ref{thm:gkvanishdr}, is the following.

\begin{thmn} \label{thm:mainthm2}

Suppose that $K$ is a finite extension of $\mathbf{Q}_p$, $\mathfrak{R}$ is a Noetherian $\mathbf{Q}_p$-Banach algebra, and the finitely generated $\mathfrak{R}$-Banach module $\mathfrak{M}$ is equipped with a continuous $\mathfrak{R}$-linear action of $G_K$.  Then for $n \ge 2$, $k \le \ell \in \mathbf{Z}$, and $B \in \set{\bht,t^k\bdrk{\ell}, t^k\bdr^+,\bdr}$, we have $H^n(G_K,B\hat{\otimes}_{\mathbf{Q}_p} \mathfrak{M})=0$.

\end{thmn}

Appendix \ref{sec:appendix} proves a simple result on specialization of a morphism of free modules over regular local rings.  This provides an alternate approach to the main theorems above; we include this due to its possible application to situations where our approach of studying $1$-cocycles is unavailable.

\subsection{Method of proof}

The aforementioned results of Berger-Colmez \cite{bc} and Bellovin \cite{bell} use an approach very different from ours; they use the theory of $(\varphi,\Gamma)$-modules while we work directly with the $G_K$-module itself.

Our approach to Theorem \ref{thm:drcase} is similar to Kisin's work in \cite[\S2]{kisin}.  Kisin shows that in \emph{any} family of representations with a single fixed Hodge-Tate-Sen weight 0, the $\mathbf{C}_p$-periods form a finite flat module of rank 1 after localizing to remove possible multiplicity of this weight.  However, if a Hodge-Tate-Sen weight occurs with multiplicity or if one has distinct Hodge-Tate-Sen weights and would like to study de Rham periods, this property is no longer automatic.  One needs to use the existence of Hodge-Tate or de Rham periods at a dense set of points.  To interpolate these periods, we devise an algebraic situation in which finite flatness of $D_B(\mathfrak{M})$ will follow from the density of the given specializations.  Since the formation of $D_B(\mathfrak{M})$ is not always compatible with specialization, we use Sen's theory to study a simpler module $\mathfrak{E}$ associated to $\mathfrak{M}$ instead \cite{sen,sen2}.

Another consideration is that in attempting to pass to unbounded de Rham periods, the direct approach requires one to localize at all integral weights above the fixed range.  For thickenings of geometric points, we control the dimension of the space of periods as one passes to the limit.  This requires some functional analysis.

We also need to retain information when multiplicities change.  For this, we study the behavior of a morphism under base change in Section \ref{subsec:special}.  The series of lemmas proved there use Sen's theory to show that certain tensor products of Banach spaces are already complete, allowing us to apply homological algebra.  By various arguments that allow us to switch from 1-cocycles to 0-cocycles, we are able to show existence of periods at specializations anywhere on the base.

For the stratification result, we apply the same techniques, but since we no longer have fine control over the localization needed to make the module of periods finite flat, we need a slightly different approach.  In order to globalize in this setting, we first study the affinoid case, and then apply those interpolation results to analytically continue across intersections.

\subsection{Acknowledgments}

This work formed part of the author's Ph.D. dissertation.  I would like to thank my advisor Christopher Skinner for his guidance over the course of the research.  I also thank Kiran Kedlaya, Mark Kisin, and Aaron Pollack for helpful discussions, and thank Brian Conrad both for helpful discussions and his comments on a draft of this paper (including the argument for Lemma \ref{lem:galdes} used below, which is more efficient than my original proof).  I thank the anonymous referees for their careful reading of the manuscript and for numerous suggestions that greatly improved the exposition.

\section{Interpolating de Rham and Hodge-Tate periods and 1-cocycles} \label{sec:nonreducedbase}

We study the interpolation of Hodge-Tate and de Rham periods over a Noetherian Banach algebra.  We prove some useful lemmas in commutative algebra, and then summarize the work of Sen \cite{sen,sen2} and Kisin \cite[\S 2]{kisin} on the construction and properties of the Sen operator.  We proceed to the proofs of Theorem \ref{thm:htcase} in the Hodge-Tate setting and Theorem \ref{thm:drcase} in the de Rham setting for periods over an open set $\spec \mathfrak{R}_{Q_{\sigma,k}}$.  We then study specializations on all of $\spec\mathfrak{R}$ under the additional assumption that $\mathfrak{R}$ is affinoid, pass from bounded periods to the full ring $\bdr$, and give an application to essentially self-dual specializations.

\subsection{Commutative algebra}

All rings are commutative.  When we say a module is finite, we mean that it is finitely generated.

\begin{lem}\label{lem:galdes}

Suppose that $G$ is a finite group, $K$ is a field of characteristic 0, and $M$ is a module over a $K$-algebra $R$ equipped with an $R$-linear $G$-action.
\begin{enumerate}[{\normalfont (a)}]
	\item If $M$ is finite (resp.\ flat, projective), then $M^G$ is finite (resp.\ flat, projective).
	\item For any $R$-module $N$, the natural map $\varphi:N \otimes_R M^G \rightarrow (N \otimes_R M)^G$ is an isomorphism.
\end{enumerate}

\end{lem}

\begin{proof}
There is a functorial decomposition $M = M^G \oplus M^{G,\perp}$ as the sum of the image and kernel of the projection $\frac{1}{|G|}\sum_{g \in G}g: M \rightarrow M$.  Part (a) follows, as well as that $\varphi$ is injective.  For any $K[G]$-module $P$, the image of $\frac{1}{|G|}\sum_{g \in G}g:P \rightarrow P$ is $P^G$, and the image of this map for the module $N \otimes_R M$ lands in the submodule $N \otimes_R M^G$.  Thus $\varphi$ is also surjective.
\end{proof}

We will use the next lemma to decompose a module into finite flat summands.  We state it in a slightly more general form in case it may be helpful in other contexts.

\begin{lem} \label{lem:splitht}

Suppose $M$ is a module over the ring $R$.  Let $r_1,\dots,r_n \in R$, let $k_1,\dots,k_n \in \mathbf{Z}_{\ge 1}$, let $S(x)=\prod_{i=1}^n (x-r_i)^{k_i}$, let $Q(x) \in R[x]$, let $P(x) = Q(x)S(x)$, and assume that $r = \prod_i Q(r_i)\prod_{i<j}(r_i-r_j)$ is a unit in $R$.  Let $\varphi: M \rightarrow M$ be a morphism of $R$-modules such that $P(\varphi)M=0$.

Define $S_i(x) = \prod_{j \ne i} (x-r_j)^{k_j}$ and $P_i(x)=Q(x)S_i(x)$ for $1 \le i \le n$. We define the following submodules of $M$: $M_0 = \ker Q(\varphi), M^0 = S(\varphi)M, M_i = \ker (\varphi-r_i)^{k_i},$ and $M^i = P_i(\varphi)M.$  Then we have $M_i=M^i$ for $0 \le i \le n$ and a $\varphi$-compatible splitting $M = \oplus_{i=0}^n M_i = \oplus_{i=0}^n M^i.$

\end{lem}

\begin{proof}
Given a factorization $U(x)=(x-r)^kT(x)$, where $U,T\in R[x]$ and $T(r)$ is a unit, $T(x)$ is a unit in $R[x]/(x-r)$ and thus in $R[x]/((x-r)^k)$.  So $T(x)$ and $((x-r)^k)$ are coprime, and
\[R[x]/U(x) = R[x]/T(x) \times R[x]/(x-r)^k\]
by the Chinese remainder theorem.  Setting $U(x) = Q(x)\prod_{j=i}^n (x-r_j)^{k_j}$, $r=r_i$, $k=k_i$, and $T(x) = Q(x)\prod_{j=i+1}^n (x-r_j)^{k_j}$ for each $i$, we deduce
\[R[x]/P(x) = \underbrace{R[x]/Q(x)}_{R_0} \times \prod_{i=1}^n \underbrace{R[x]/(x-r_i)^{k_i}}_{R_i}.\]
Since $P(\varphi)M=0$, $M$ can be regarded as an $R[x]/P(x)$-module with $x$ acting as $\varphi$.  All claims above follow from the observations that each $P_i(x)$ (resp.\ each $(x-r_i)^{k_i}$) is a unit multiple of the idempotent projecting onto $R_i$ (resp.\ $R_0\times\prod_{j\ne i}R_j$), and $S(x)$ (resp.\ $Q(x)$) is a unit multiple of the the idempotent projecting onto $R_0$ (resp.\ $\prod_{i=1}^n R_i$).
\end{proof}

We will also use the following easy lemma.

\begin{lem} \label{lem:niceend}

Let $M$ be a flat module over a ring $R$ equipped with an endomorphism $\psi:M \rightarrow M$.  If $\psi$ has flat cokernel, then it has flat kernel and image.  Moreover, the formation of $\ker \psi$, $\im \psi$, and $\coker \psi$ commutes with taking the tensor product with an $R$-module $N$.  If $M$ is also finitely generated, and $R$ is Noetherian, then the kernel, image, and cokernel of $\psi$ are finitely generated, with the rank of the kernel equal to the rank of the cokernel on each connected component of $\spec R$.

If we drop the hypotheses that $M$ is flat and $\psi$ has flat cokernel, but instead ask that the $R$-module $N$ is flat, then tensoring with $N$ again commutes with formation of $\ker \psi$, $\im \psi$, and $\coker \psi$.

\end{lem}

\begin{proof}

We have exact sequences
\begin{equation} \label{eqn:niceendexact} 0 \rightarrow \ker \psi \rightarrow M \rightarrow \im \psi \rightarrow 0\quad\textrm{and}\quad 0 \rightarrow \im \psi \rightarrow M \rightarrow \coker \psi \rightarrow 0.\end{equation}
It follows that if $\coker \psi$ is flat, so are $\im \psi$ and $\ker \psi$.  The finite generation claim is clear, and equality of ranks follows from examining the localization of (\ref{eqn:niceendexact}) at a prime of $R$ in each connected component.

Let $N$ be an $R$-module, and assume that either $M$ and $\coker \psi$ or $N$ are flat.  By tensoring (\ref{eqn:niceendexact}) with $N$ and using the flatness of either $\coker \psi$ and $\im \psi$ or $N$, we obtain a commutative diagram
\[\xymatrix@R=15pt{ 0 \ar[r] & (\ker \psi) \otimes_R N \ar[r] \ar[d] & M \otimes_R N \ar[r] \ar[d]^{\begin{sideways}=\end{sideways}} & M \otimes_R N \ar[r] \ar[d]^{\begin{sideways}=\end{sideways}} & (\coker \psi) \otimes_R N \ar[r] \ar[d] & 0\\
0 \ar[r] & \ker (\psi \otimes_R N) \ar[r] & M \otimes_R N \ar[r] & M \otimes_R N \ar[r] & \coker (\psi \otimes_R N) \ar[r] & 0
}\]
with both rows exact.  By the 5-lemma, we obtain isomorphisms
\[(\ker \psi) \otimes_R N \iso \ker (\psi \otimes_R N)\textrm{ and }(\coker \psi) \otimes_R N \iso \coker (\psi \otimes_R N)\]
via the natural maps.  The last part of the claim follows similarly from the exact sequence
\[0 \rightarrow (\im \psi) \otimes_R N \rightarrow M \otimes_R N \rightarrow (\coker \psi) \otimes_R N \rightarrow 0\]
and the isomorphism $(\coker \psi) \otimes_R N \iso \coker (\psi \otimes_R N)$ above.
\end{proof}

We can say more in the case of the cokernel.

\begin{lem} \label{lem:cokerspec}

Suppose $\psi: M \rightarrow N$ is a map of $R$-modules and $P$ is an $R$-module.  Then the natural map $(\coker \psi) \otimes_R P \rightarrow \coker(\psi \otimes_R P)$ is an isomorphism.

\end{lem}

\begin{proof}
Tensor products commute with colimits.
\end{proof}

We will need an explicit description of continuous $\mathbf{Z}_p$-cohomology for certain nice modules.

\begin{lem} \label{lem:zpcoh}

Suppose that $\Gamma$ is the additive topological group $\mathbf{Z}_p$, and write $\gamma$ for the topological generator $1 \in \mathbf{Z}_p$.  Let $M$ be a continuous $\Gamma$-module with the property that there exist $\Gamma$-invariant submodules $M_i$ for $i \in \mathbf{Z}_{\ge 0}$ such that $M_i \subseteq M_{i-1}$, $M \cong \varprojlim_i M/M_i$ with the inverse limit topology, and each $M/M_i$ is a discrete $p^\infty$-torsion module, meaning that each $m \in M/M_i$ is killed by a finite power of $p$.  Then there is an isomorphism $H^1(\Gamma,M) \cong \coker(\gamma-1:M \rightarrow M)$.

As a special case, this identification applies to any $\mathbf{Q}_p$-Banach space with a continuous $\Gamma$-action.

\end{lem}

\begin{proof}
It suffices to show that for any $m \in M$, there is a continuous 1-cocycle $\varphi:\Gamma \rightarrow M$ with $\varphi(\gamma)=m$.  The cocycle relation forces $\varphi(\gamma^k) = \sum_{j=0}^{k-1} \gamma^j m$ for $k \in \mathbf{Z}_{> 0}$.  Since a continuous $\varphi:\Gamma \rightarrow M$ is determined by $\varphi|_{\mathbf{Z}_{> 0}}$, we need only check that this has a continuous extension to $\Gamma$.  In fact, we need only check that $\varphi|_{\mathbf{Z}_{>0}}: \Gamma \rightarrow M_i$ has a continuous extension for each $i$; these are automatically compatible and give $\varphi:\Gamma \rightarrow M$ by the universal property.  This is proved in \cite[Proposition 1.7.7]{nsw}.

To see the final claim, note that for any $\mathbf{Q}_p$-Banach space $M$ with a continuous $\Gamma$-action, $\Gamma$ preserves a norm on $M$ inducing its topology by \cite[Lemma 6.5.5]{emerton}, so $M$ is the projective limit of the discrete $p^{\infty}$-torsion spaces $M/M_i$ for $M_i$ the open ball of radius $p^{-i}$ centered at 0 with respect to the invariant norm.  So $M$ and these $M_i$'s meet the conditions above.
\end{proof}

\subsection{Some results of Sen and Kisin} \label{subsec:senkisin}

Fix a separable algebraic closure $\ol{\mathbf{Q}}_p$ of $\mathbf{Q}_p$ and denote by $\mathbf{C}_p$ its completion.  For any finite extension $K \subseteq \ol{\mathbf{Q}}_p$ of $\mathbf{Q}_p$, let $G_K = \gal(\ol{\mathbf{Q}}_p/K)$ be the absolute Galois group of $K$, and let $\chi_K: G_K \rightarrow \mathbf{Z}_p^\times$ denote the cyclotomic character.  Denote the kernel of $\chi_K^{p-1}$ (or $\chi_K^2$ if $p=2$) by $H_K$, and write $\Gamma_K = G_K/H_K$.  Note that $\Gamma_K$ is isomorphic to $\mathbf{Z}_p$.  We define $K_\infty = \ol{\mathbf{Q}}_p^{H_K}$ and $\widehat{K}_\infty = \mathbf{C}_p^{H_K}$.

Let $K/\mathbf{Q}_p$ and $E/\mathbf{Q}_p$ be finite extensions such that $E$ contains the normal closure of $K$.  Define $\Sigma = \set{\sigma:K \rightarrow E}$ to be the set of embeddings of $K$ in $E$.  Let $\mathfrak{R}$ be a Noetherian Banach algebra over $E$ and let $\mathfrak{M}$ be a free $\mathfrak{R}$-module of finite rank equipped with a continuous $G_K$-action.

We denote the completed tensor product of Banach spaces or modules by $\hat{\otimes}$, and we always require any map between topological algebras or modules to be continuous.  In this paper, rings are always commutative (except for endomorphism rings of modules) and cohomology is always continuous cohomology.  We \emph{never} regard a cohomology group $H^r(G,M)$ for $r \ge 1$ as having a topology, even if it is possible to do so by virtue of it being a finitely generated module over a Banach algebra.  (But we do equip $H^0(G,M)$ with the subspace topology when $M$ is topological.)

Using the work of Sen \cite{sen,sen2}, after a base change in $K$, Kisin relates the study of $G_K$-cohomology of $\mathbf{C}_p \hat{\otimes}_{K,\sigma} \mathfrak{M}$ to the action of a linear operator on a finite free module $\mathfrak{E}$ over $\mathfrak{R}$.  Kisin works under the hypothesis $E \subseteq K$, but we have instead reformulated his results for $E$ containing the normal closure of $K$, and a choice of embedding $\sigma: K \rightarrow E$.  To deduce the forms of his results given below, one sets Kisin's $E$ to be $K$ and regards the $E$-algebra $\mathfrak{R}$ as a $K$-algebra via the choice of $\sigma$.  In this paper, we will always fix $\sigma$ and regard $\mathfrak{R}$ as a $K$-algebra in this way.  We also fix a choice of topological generator $\gamma \in \Gamma_K$.

\begin{prop}[{\cite[\S2]{kisin}}] \label{prop:wstarprop}

We have the following facts regarding the $H_K$-invariants of $\mathbf{C}_p \hat{\otimes}_{K,\sigma} \mathfrak{M}$.
\begin{enumerate}[{\normalfont (a)}]
	\item After replacing $E$ and $K$ with sufficiently large finite Galois extensions, the $\widehat{K}_\infty \hat{\otimes}_{K,\sigma} \mathfrak{R}$-module $(\mathbf{C}_p \hat{\otimes}_{K,\sigma} \mathfrak{M})^{H_K}$
is free with a basis $\mathbf{e} = \set{e_1,\dots,e_n}$ such that the free $\mathfrak{R}$-module $\mathfrak{E}$ generated by $\mathbf{e}$ is stable by $\Gamma_K$.
	\item We may define a map $\phi \in \End_\mathfrak{R} \mathfrak{E}$ by $\phi = (\log\chi(\gamma^{p^r}))^{-1} \cdot \log(\gamma^{p^r}|_\mathfrak{E})$ for any sufficiently large integer $r$, where $\log(\gamma^{p^r}|_\mathfrak{E})$ is defined using the standard series expansion in $\gamma^{p^r}-1$.
	\item Denote also by $\phi$ the extension of scalars to $(\mathbf{C}_p \hat{\otimes}_{K,\sigma} \mathfrak{M})^{H_K}$, and call the characteristic polynomial of $\phi$ on this module $P_\sigma(T)\in \widehat{K}_\infty \hat{\otimes}_{K,\sigma} \mathfrak{R}[T]$.  Then in fact, we have $P_\sigma(T) \in \mathfrak{R}[T]$, even before increasing $E$ and $K$.  Moreover, this polynomial is canonically attached to $\mathfrak{M}$.
	\item The preceding constructions are all compatible with base change along a map of Noetherian $E$-Banach algebras $\mathfrak{R} \rightarrow \mathfrak{R}'$, without requiring a further extension of $K$.  In particular, we have
\[\mathfrak{E}_{\mathfrak{R}'} = \mathfrak{E} \otimes_{\mathfrak{R}} \mathfrak{R}'\textrm{ and }(\mathbf{C}_p \hat{\otimes}_{K,\sigma} \mathfrak{M})^{H_K} \hat{\otimes}_\mathfrak{R} \mathfrak{R}' \iso (\mathbf{C}_p \hat{\otimes}_{K,\sigma} (\mathfrak{M} \otimes_\mathfrak{R} \mathfrak{R}'))^{H_K},\]
and these are equipped with the Sen operator $\phi \hat{\otimes}_{\mathfrak{R}} \mathfrak{R}'$.
	\item If $\mathfrak{R}' = E'$ is a finite extension field of $E$, then $\phi \hat{\otimes}_{\mathfrak{R}} E'$ coincides with the Sen operator of \cite{sen3}.
\end{enumerate}
\end{prop}

We check that the kernel and cokernel of the Sen operator calculate $G_K$-cohomology of $\mathbf{C}_p \hat{\otimes}_{K,\sigma} \mathfrak{M}$.
\begin{prop}[{\cite[Proposition 2.3]{kisin}}] \label{prop:sencoh}

Write
\[\psi=\gamma-1:(\mathbf{C}_p \hat{\otimes}_{K,\sigma} \mathfrak{M})^{H_K} \rightarrow (\mathbf{C}_p \hat{\otimes}_{K,\sigma} \mathfrak{M})^{H_K}.\]
After replacing $E$ and $K$ with finite extensions, we have
\[\ker(\psi) = \ker(\phi|_\mathfrak{E}) \textrm{ and } \coker(\psi) = \coker(\phi|_\mathfrak{E}).\]
Moreover, if $\mathfrak{R} \rightarrow \mathfrak{R}'$ is a map of Noetherian $E$-Banach algebras, the same is true for $\mathfrak{R}'$ without a further extension of $K$.

\end{prop}

\begin{proof}
In the proof of \cite[Proposition 2.3]{kisin}, Kisin shows that after increasing $K$ beyond that used in Proposition \ref{prop:wstarprop} if necessary, the kernel and cokernel of $\phi|_\mathfrak{E}$ are equal, respectively, to the kernel and cokernel of $\gamma-1$ on all of $(\mathbf{C}_p \hat{\otimes}_{K,\sigma} \mathfrak{M})^{H_K}$, so we are left to check the change of ring.

Kisin uses a construction by Tate of a topological isomorphism $\widehat{K}_\infty = K \oplus X_0$ such that $X_0$ is $\Gamma_K$-stable and $1-\gamma:X_0 \rightarrow X_0$ is a $\Gamma_K$-equivariant isomorphism \cite{tate}.  We then find that
\[(\mathbf{C}_p \hat{\otimes}_{K,\sigma} \mathfrak{M})^{H_K} \cong \widehat{K}_\infty \hat{\otimes}_{K,\sigma} \mathfrak{E} \cong \mathfrak{E} \oplus X_0 \hat{\otimes}_{K,\sigma} \mathfrak{E}.\]

Kisin shows that on the first factor, $-\psi = 1-\exp(\log(\chi(\gamma))\phi) = \log(\chi(\gamma))\phi G(\phi)$, where $G(\phi) \in \GL_n(\mathfrak{R})$ for sufficiently large $K$ (which replaces $\gamma$ by $\gamma^{p^r}$).  A base change will not alter this property.

On the second factor, $X_0 \hat{\otimes}_{K,\sigma} \mathfrak{E}$, Kisin rewrites the action of $-\psi=1-\gamma$ as
\[1-\gamma \otimes \gamma = (\gamma(1-\gamma)^{-1}\otimes (1-\gamma)+1)((1-\gamma)\otimes 1),\]
and shows that if we replace $K$ by a sufficiently large extension, the action of $1-\gamma$ is topologically nilpotent on $\mathfrak{E}$.  This calculation then implies that $1-\gamma$ acts as an isomorphism on $X_0 \hat{\otimes}_{K,\sigma} \mathfrak{E}$.  For a ring extension of Noetherian $E$-Banach algebras, $1-\gamma$ will remain an isomorphism, so there is no need to extend $K$ further.
\end{proof}

\begin{coro} \label{coro:sencoh}
For $E$ and $K$ sufficiently large, we have
\[H^0(G_K,\mathbf{C}_p \hat{\otimes}_{K,\sigma} \mathfrak{M}) \cong \ker(\phi|_\mathfrak{E}) \textrm{ and } H^1(G_K,\mathbf{C}_p \hat{\otimes}_{K,\sigma} \mathfrak{M}) \cong \coker(\phi|_\mathfrak{E}).\]
\end{coro}

\begin{proof}
The first follows from Proposition \ref{prop:sencoh}.  For the second, by Proposition \ref{prop:sencoh}, the inflation-restriction exact sequence, and the vanishing $H^1(H_K,\mathbf{C}_p \hat{\otimes}_{K,\sigma} \mathfrak{M})=0$ \cite[Proposition 2]{sen}, we are reduced to checking that $H^1(\Gamma_K,(\mathbf{C}_p \hat{\otimes}_{K,\sigma} \mathfrak{M})^{H_K}) \cong \coker (\gamma-1)$.  The result now follows from Lemma \ref{lem:zpcoh}.
\end{proof}

\begin{remark}
For the remainder of this paper, we always use $\phi$ for the map on $\mathfrak{E}$.  We also use $\phi$ only when $K$ is large enough so that the conclusions of Proposition \ref{prop:wstarprop}, Proposition \ref{prop:sencoh}, and Corollary \ref{coro:sencoh} hold.  However, we make use of $P_\sigma(T)$ more generally, since it is defined as an element of $\mathfrak{R}[T]$.
\end{remark}

Let $P_\sigma(T) \in \mathfrak{R}[T]$ denote the $\sigma$-factor of the Sen polynomial of $\mathfrak{M}$.  (The usual Sen polynomial is the product of the $P_\sigma(T)$ for all $\sigma \in \Sigma$.)

\begin{prop}[{\cite[Proposition 2.3]{kisin}}] \label{prop:htgap}

Fix $\sigma \in \Sigma$.  Then the $\mathfrak{R}$-modules
\[H^0(G_K,\mathbf{C}_p \hat{\otimes}_{K,\sigma} \mathfrak{M})\textrm{ and }H^1(G_K,\mathbf{C}_p \hat{\otimes}_{K,\sigma} \mathfrak{M})\]
are finitely generated and killed by $\det(\phi) = P_\sigma(0) \in \mathfrak{R}$.

\end{prop}

\subsection{$\mathbf{C}_p$-periods in families} \label{subsec:htcase}

Kisin \cite[Proposition 2.4]{kisin} considers a family with a single fixed Hodge-Tate-Sen weight 0 of multiplicity 1, localizes to fix this multiplicity, and shows that the $\mathbf{C}_p$-periods form a finite flat module of rank 1 that is compatible with base change.  We generalize this result, though we require an extra hypothesis on the family if the multiplicity is greater than 1: there exist a dense set of specializations with $m$ Hodge-Tate weights equal to 0.  On the open set where the weight 0 has fixed multiplicity $m$, we show that both the $\mathbf{C}_p$-periods and 1-cocycles interpolate in a finite flat module and are compatible with base change.

In order to interpolate periods over $E$-Banach algebras that are possibly non-reduced, we will need to define a robust version of density that generalizes Zariski density.  For specializations $\xi_i: \mathfrak{R} \rightarrow \mathfrak{R}_i$, we want the map $\mathfrak{R} \rightarrow \prod_i \mathfrak{R}_i$ to be injective.  When each $\mathfrak{R}_i$ is a finite field extension of $E$, this is Zariski density of $\set{\ker \xi_i}$.  Limiting ourselves to this case would force $\mathfrak{R}$ to be reduced.  However, when the $\mathfrak{R}_i$ are permitted to be local Artinian rings of finite $E$-dimension, the condition that $\mathfrak{R} \inj \prod_i \mathfrak{R}_i$ is no longer stable under localization.  This motivates the following definition.

\begin{defin} \label{defin:stabdens}

For a local Artinian algebra $R$ with maximal ideal $\mathfrak{m}$, we define the \emph{breadth} of $R$, $\br(R) \in \mathbf{Z}_{\ge 1}$, by $\br(R) = \min\set{n:\mathfrak{m}^n =0}$.  We say that a set $\set{\xi_i}$ of specializations $\xi_i:R \rightarrow R_i$ of a ring $R$ to local Artinian rings $R_i$ is \emph{stably dense} if $R \inj \prod_i R_i$ is injective and $\set{\br(R_i)}$ is bounded above.

\end{defin}

The term ``stable'' is justified by the following observation.

\begin{lem} \label{lem:stabdens}

Suppose that $\set{\xi_i}$ is a stably dense set of specializations of the ring $R$ to local Artinian rings $R_i$.  Then for any $f \in R$, the localizations $\set{\xi_{i,f}}$ are stably dense in $R_f$.

\end{lem}

\begin{proof}
Let $\mathfrak{m}_i$ be the maximal ideal of $R_i$.  Let $f$ be an element of $R$, and suppose that $\frac{a}{f^k} \in \ker (R_f \rightarrow \prod_i R_{i,f})$.  Then $a$ is also in the kernel.  We have $R_{i,f}=0$ if $\xi_i(f) \in \mathfrak{m}_i$ and $R_{i,f}= R_i$ if $\xi_i(f) \notin \mathfrak{m}_i$.  Write $\prod_i R_i = \prod_{\set{i: \xi_i(f) \in \mathfrak{m}_i}} R_i \times \prod_{\set{i: \xi_i(f) \notin \mathfrak{m}_i}} R_i.$  Then the image of $a$ is 0 in the second factor, while the image of $f$ is nilpotent in the first by the bounded breadth condition.  Therefore $f^ba=0$ for $b=\max\set{\br(R_i)}$, so $a=0\in R_f$.  The bounded breadth condition remains satisfied, so we conclude that $\set{\xi_{i,f}}$ is stably dense.
\end{proof}

We proceed to the proof of the main theorem of this section.

\begin{thm} \label{thm:htcase}

Maintain the notation of Section \ref{subsec:senkisin}.  Fix $\sigma \in \Sigma$.  Suppose that $P_\sigma(T) = T^mQ_{\sigma}(T)$.  If $m \ge 2$, assume that there exist maps $\xi_i: \mathfrak{R} \rightarrow \mathfrak{R}_i$ of Banach $E$-algebras for $i \in I$, where $\mathfrak{R}_i$ is local Artinian of finite $E$-dimension, such that the following hold.
\begin{enumerate}[{\normalfont (i)}]
	\item For each $i$, $\dim_E H^0(G_K, \mathbf{C}_p \otimes_{K,\sigma} (\mathfrak{M} \otimes_\mathfrak{R} \mathfrak{R}_i)) = m\dim_E \mathfrak{R}_i$.
	\item For each $i$, the image of $Q_\sigma(0)$ in $\mathfrak{R}_i$ is a unit.
	\item The set $\set{\xi_i}$ is stably dense in $\mathfrak{R}$.
\end{enumerate}
Then we have the following.  We use subscripts to denote localization.
\begin{enumerate}[{\normalfont (a)}]
	\item For $r \in \set{0,1}$, the $\mathfrak{R}_{Q_\sigma(0)}$ module $H^r(G_K,\mathbf{C}_p \hat{\otimes}_{K,\sigma} \mathfrak{M})_{Q_\sigma(0)}$ is finite flat of rank $m$.
	\item For any map $\xi:\mathfrak{R} \rightarrow \mathfrak{R}'$ of Noetherian $E$-Banach algebras, the cokernel and kernel of the map $H^r(G_K,\mathbf{C}_p \hat{\otimes}_{K,\sigma} \mathfrak{M}) \otimes_{\mathfrak{R}} \mathfrak{R}' \rightarrow H^r(G_K,\mathbf{C}_p \hat{\otimes}_{K,\sigma} (\mathfrak{M} \otimes_{\mathfrak{R}} \mathfrak{R}'))$ are killed by a power of $\xi(Q_\sigma(0))$ if $r=0$, and vanish if $r=1$.
\end{enumerate}

\end{thm}

\begin{remark}

Whenever we localize a ring $\mathfrak{R}$ or $\mathfrak{R}$-module, we do so in the categories of rings and their modules, with no topology or norm structure.  When we write $(\cdot)^G_a$, we always mean the localization of $(\cdot)^G$ at $a$.

\end{remark}

\begin{remark}

The rings $\mathfrak{R}_i$ will usually be geometric points, or possibly thickened geometric points of the form $\mathfrak{R}/\mathfrak{m}^n$.

\end{remark}

To apply the results of Section \ref{subsec:senkisin}, we will need the following.

\begin{lem} \label{lem:changeek}
For a finite Galois extension $E'$ of $E$, a profinite group $G$, an integer $r \ge 0$, and a continuous $E[G]$-module $M$, we have a natural isomorphism
\begin{equation}\label{eqn:edescalc} H^r(G,M) \otimes_E E' \iso H^r(G,M \otimes_E E').\end{equation}

For a finite Galois extension $K'$ of $K$ whose normal closure is contained in $E$, a fixed choice of extension $\sigma': K'\rightarrow E$ of $\sigma$, and $r \in \set{0,1}$, we have a natural isomorphism
\begin{equation}\label{eqn:kdescalc} H^r(G_K,\mathbf{C}_p\hat{\otimes}_{K,\sigma} \mathfrak{M}) \iso H^r(G_{K'}, \mathbf{C}_p \hat{\otimes}_{K',\sigma'} \mathfrak{M}).\end{equation}
\end{lem}

\begin{proof}
Let $C^\cdot(G,M)$ denote the complex of continuous cochains.  The natural map
\[C^\cdot(G,M) \otimes_E E' \iso C^\cdot(G,M \otimes_E E')\]
is an isomorphism since there is a topological isomorphism $E' \cong E^k$ for some $k$.  The isomorphism (\ref{eqn:edescalc}) follows from this and Lemma \ref{lem:niceend}, using the flat $E$-module $E'$ for $N$.

By inflation-restriction and the vanishing of finite group cohomology in characteristic 0, we have for any finite Galois extension $K'/K$ contained in $E$ and $r \in \set{0,1}$ a natural isomorphism
\[H^r(G_K,\mathbf{C}_p\hat{\otimes}_{K,\sigma} \mathfrak{M}) \iso H^r(G_{K'}, \mathbf{C}_p \hat{\otimes}_{K,\sigma} \mathfrak{M})^{\gal(K'/K)},\]
where $\gal(K'/K)$ acts on both $\mathbf{C}_p$ and $\mathfrak{M}$.

For clarity, we write $K'$ for the field with its usual $\gal(K'/K)$ action and $K_\mathrm{triv}'$ for $K'$ with the trivial action.  We claim that for any $\sigma': K_\mathrm{triv}' \rightarrow E$ extending $\sigma$, we have
\begin{equation}\label{eqn:switchfield}\mathbf{C}_p \hat{\otimes}_{K,\sigma} \mathfrak{R} = \mathbf{C}_p \hat{\otimes}_K (K_\mathrm{triv}' \hat{\otimes}_{K_\mathrm{triv}',\sigma'} \mathfrak{R}) = (\mathbf{C}_p \hat{\otimes}_K K_\mathrm{triv}') \hat{\otimes}_{K_\mathrm{triv}',\sigma'} \mathfrak{R} \cong \mathbf{C}_p^{\set{\tau| \tau \in \gal(K'/K)}} \hat{\otimes}_{K_\mathrm{triv}',\sigma'} \mathfrak{R},\end{equation}
where the action of $G_K$ on the right acts on each copy of $\mathbf{C}_p$ as well as by permutation via its quotient $\gal(K'/K)$ on the factors, and acts trivially on $\mathfrak{R}$.  In particular, $\sigma$ sends $(c_\tau)_\tau$ to $(\sigma(c_{\sigma^{-1}\tau}))_\tau$.  The last map is obtained from the $G_K$-equivariant map $\mathbf{C}_p \hat{\otimes}_K K_\mathrm{triv}' \rightarrow \mathbf{C}_p^{\set{\tau| \tau \in \gal(K'/K)}}$ defined by $c \otimes k\mapsto (c\cdot \tau(k))_\tau$ (where $\tau(k)$ denotes the usual action of $\tau$ on $k$).  It follows from independence of characters that if $\set{k_j}$ are a $K$-basis of $K'$ and $\set{\tau_i} = \gal(K'/K)$, the matrix $(\tau_i(k_j))_{i,j}$ is invertible, so this is an isomorphism.  We deduce
\begin{align*}
	H^r(G_{K'}, \mathbf{C}_p \hat{\otimes}_{K,\sigma} \mathfrak{M})^{\gal(K'/K)} &\cong H^r(G_{K'}, \mathbf{C}_p^{\set{\tau| \tau \in \gal(K'/K)}} \hat{\otimes}_{K_\mathrm{triv}',\sigma'} \mathfrak{M})^{\gal(K'/K)}\\
	&\cong [H^r(G_{K'}, \mathbf{C}_p \hat{\otimes}_{K_\mathrm{triv}',\sigma'} \mathfrak{M})^{\set{\tau| \tau \in \gal(K'/K)}}]^{\gal(K'/K)}.
\end{align*}
The second map is induced by the natural $G_K$-equivariant Banach space isomorphism
\begin{equation} \label{eqn:kdes1} \mathbf{C}_p^{\set{\tau \in \gal(K'/K)}} \hat{\otimes}_{K_\mathrm{triv}',\sigma'} \mathfrak{M} \iso (\mathbf{C}_p \hat{\otimes}_{K_\mathrm{triv}',\sigma'}\mathfrak{M})^{\set{\tau \in \gal(K'/K)}}\end{equation}
defined by sending $(c_\tau)_\tau \otimes m$ to $(c_\tau \otimes m)_\tau$ and extending by continuity.

We construct an isomorphism
\[ H^r(G_{K'}, \mathbf{C}_p \hat{\otimes}_{K_\mathrm{triv}',\sigma'} \mathfrak{M}) \iso [H^r(G_{K'}, \mathbf{C}_p \hat{\otimes}_{K_\mathrm{triv}',\sigma'} \mathfrak{M})^{\set{\tau| \tau \in \gal(K'/K)}}]^{\gal(K'/K)}\]
by mapping the cocycle $\psi$ to $(\tau\psi)_{\tau \in \gal(K'/K)}$.  Combining the above, we obtain (\ref{eqn:kdescalc}); we write $K'$ on the right in that expression since there is no $\gal(K'/K)$-action.
\end{proof}

\begin{proof}[Proof of Theorem \ref{thm:htcase}]

We claim that may replace the pair $K,E$ with any Galois extensions $K',E'$ such that $E'$ contains the normal closure of $K'$.  More precisely, if we replace $\mathfrak{R}$ with $\mathfrak{R} \otimes_E E'$, $\mathfrak{M}$ with $\mathfrak{M} \otimes_E E'$, $E$ with $E'$, and $K$ with $K'$ in Theorem \ref{thm:htcase}, the old conditions (i)-(iii) imply the new conditions and the new conclusions (a) and (b) imply the original (a) and (b).  In fact, for all these conditions and conclusions, the ability to swap $K$ for $K'$ and vice-versa follows immediately from (\ref{eqn:kdescalc}), so we need only check the change in $E$.

If we replace $\mathfrak{R}_i \otimes_E E'$ with the set of its local Artinian factors, the claims regarding $E$ and conditions (ii) and (iii) are clear.  For (i), we apply Lemma \ref{lem:changeek} to calculate the new $E$-dimension for the base change to $\mathfrak{R}_i \otimes_E E'$.  We will obtain (i) for each local Artinian factor by showing in Claim \ref{claim:rankcalc} below that no Artinian local $\mathfrak{R} \otimes_E E'$-algebra $\mathfrak{R}'$ of finite $E'$-dimension with unit image of $Q_\sigma(0)$ can have $\dim_{E'} H^0(G_K, \mathbf{C}_p \hat{\otimes}_{K,\sigma} (\mathfrak{M} \otimes_\mathfrak{R} \mathfrak{R}')) > m\dim_{E'} \mathfrak{R}'$.  Note that will not make use of condition (i) before proving this claim.

Suppose that $E'/E$ is a finite Galois extension and $M$ is a $\mathfrak{R}$-module.  By Lemma \ref{lem:galdes}.(b), $(M \otimes_E E')^{\gal(E'/E)} =  M$, where $\gal(E'/E)$ acts only on $E'$.  If $M \otimes_E E'$ is finite flat (when viewed as a $\mathfrak{R} \otimes_E E'$-module), it is finite flat over $\mathfrak{R}$, and so by Lemma \ref{lem:galdes}.(a), $M$ is finite flat over $\mathfrak{R}$ as well.  Moreover, we have $\rank_{\mathfrak{R}} M = \rank_{\mathfrak{R}\otimes_E E'} M\otimes_E E'$.  We set $M = H^r(G_K,\mathbf{C}_p \hat{\otimes}_{K,\sigma} \mathfrak{M})_{Q_\sigma(0)}.$  By (\ref{eqn:edescalc}),
\begin{align*}
	H^r(G_K,\mathbf{C}_p \hat{\otimes}_{K,\sigma} \mathfrak{M})_{Q_\sigma(0)} \otimes_E E' &\cong (H^r(G_K,\mathbf{C}_p \hat{\otimes}_{K,\sigma} \mathfrak{M}) \otimes_E E')_{Q_\sigma(0)}\\
	&\cong H^r(G_K,\mathbf{C}_p \hat{\otimes}_{K,\sigma} (\mathfrak{M} \otimes_E E'))_{Q_\sigma(0)}.
\end{align*}
The reduction of Theorem \ref{thm:htcase}.(a) from $E$ to $E'$ follows.

For the reduction of Theorem \ref{thm:htcase}.(b), we suppose a change of ring $\mathfrak{R} \rightarrow \mathfrak{R}'$ is given, and consider the exact sequence
\begin{align}
\label{eqn:redbexact} 0\rightarrow \ker \psi &\rightarrow H^r(G_K,\mathbf{C}_p \hat{\otimes}_{K,\sigma} (\mathfrak{M} \otimes_E E')) \otimes_{\mathfrak{R} \otimes_E E'} (\mathfrak{R}' \otimes_E E')\\
\nonumber &\stackrel{\psi}{\rightarrow} H^r(G_K,\mathbf{C}_p \hat{\otimes}_{K,\sigma} ((\mathfrak{M} \otimes_E E') \otimes_{\mathfrak{R}\otimes_E E'} \mathfrak{R}'\otimes_E E')) \rightarrow \coker \psi\rightarrow 0.
\end{align}
For any $\mathfrak{R}$-module $P$, there is a natural identification
\begin{equation} \label{eqn:eswitch} (P \otimes_E E') \otimes_{\mathfrak{R} \otimes_E E'} (\mathfrak{R}' \otimes_E E') = (P \otimes_\mathfrak{R} \mathfrak{R}') \otimes_E E',\end{equation}
so by using this and (\ref{eqn:edescalc}) we may rewrite $\psi$ as
\[(H^r(G_K,\mathbf{C}_p \hat{\otimes}_{K,\sigma} \mathfrak{M}) \otimes_\mathfrak{R} \mathfrak{R}') \otimes_E E' \rightarrow H^r(G_K,\mathbf{C}_p \hat{\otimes}_{K,\sigma} (\mathfrak{M} \otimes_{\mathfrak{R}} \mathfrak{R}'))\otimes_E E'.\]
Take $\gal(E'/E)$-invariants of (\ref{eqn:redbexact}) and apply Lemma \ref{lem:galdes}.(b) to obtain
\begin{align}
\label{eqn:redbexact2} 0\rightarrow (\ker \psi)^{\gal(E'/E)} \rightarrow H^r(G_K,\mathbf{C}_p \hat{\otimes}_{K,\sigma} \mathfrak{M}) \otimes_\mathfrak{R} \mathfrak{R}'&\stackrel{\psi^{\gal(E'/E)}}{\rightarrow} H^r(G_K,\mathbf{C}_p \hat{\otimes}_{K,\sigma} (\mathfrak{M} \otimes_{\mathfrak{R}} \mathfrak{R}'))\\
\nonumber &\rightarrow (\coker \psi)^{\gal(E'/E)}\rightarrow 0.
\end{align}
By vanishing of cohomology in characteristic 0, the sequence remains exact.  Finally, if a finite power of $\xi(Q_\sigma(0))$ kills $\ker \psi$ and $\coker\psi$, it kills their $\gal(E'/E)$-invariants.

Replace $K$ and $E$ by extensions so that we may use Proposition \ref{prop:wstarprop} and Corollary \ref{coro:sencoh}.  We have a map $\phi: \mathfrak{E} \rightarrow \mathfrak{E}$ of finite free $\mathfrak{R}$-modules such that $\ker \phi \cong H^0(G_K,\mathbf{C}_p \hat{\otimes}_{K,\sigma} \mathfrak{M})$ and $\coker \phi \cong H^1(G_K,\mathbf{C}_p \hat{\otimes}_{K,\sigma} \mathfrak{M}).$  Moreover, the operator $P_\sigma(\phi)$ annihilates $\mathfrak{E}$.

By Lemma \ref{lem:splitht}, there is a $\phi_{Q_\sigma(0)}$-compatible splitting
\begin{equation} \label{eqn:esplit} \mathfrak{E}_{Q_\sigma(0)} = \ker(\phi_{Q_\sigma(0)}^m) \oplus \ker(Q_\sigma(\phi_{Q_\sigma(0)})),\end{equation}
so both of these modules are finite flat over the (non-topological) ring $\mathfrak{R}_{Q_\sigma(0)}$.

For any $\mathfrak{R}$-module $P$ and $\mathfrak{R}$-algebra $\mathfrak{R}'$, we have natural identifications
\begin{equation} \label{eqn:bclocal}
	P_{Q_\sigma(0)} \otimes_{\mathfrak{R}_{Q_\sigma(0)}} \mathfrak{R}'_{Q_\sigma(0)} = P \otimes_{\mathfrak{R}} \mathfrak{R}_{Q_\sigma(0)} \otimes_{\mathfrak{R}_{Q_\sigma(0)}} \mathfrak{R}'_{Q_\sigma(0)} = P \otimes_{\mathfrak{R}} \mathfrak{R}'_{Q_\sigma(0)}.
\end{equation}
In particular, for $\mathfrak{R}'$ in which $Q_\sigma(0)$ is a unit, we have
\begin{equation}\label{eqn:localswitch} \mathfrak{E} \otimes_{\mathfrak{R}} \mathfrak{R}' = \mathfrak{E}_{Q_\sigma(0)} \otimes_{\mathfrak{R}_{Q_\sigma(0)}} \mathfrak{R}' \textrm{ and } \phi \otimes_{\mathfrak{R}} \mathfrak{R}' = \phi_{Q_\sigma(0)} \otimes_{\mathfrak{R}_{Q_\sigma(0)}} \mathfrak{R}'.\end{equation}
We will use this to switch between the specialization of the localized module (which allows us to use its flatness) and the specialization of the original module (for which the properties in Proposition \ref{prop:wstarprop} hold).

\begin{claim}\label{claim:rankcalc}
We have $\rank_{\mathfrak{R}'} \ker((\phi \otimes_{\mathfrak{R}} \mathfrak{R}')^m)=m$ for any $\mathfrak{R}$-algebra $\mathfrak{R}'$ such that $Q_\sigma(0)$ is a unit in $\mathfrak{R}'$.
\end{claim}

\begin{proof}
For each connected component of $\spec \mathfrak{R}_{Q_\sigma(0)}$, fix one of the specializations $\mathfrak{R}_i$ whose spectrum maps to it and consider the further specialization $\xi_i':\mathfrak{R}\rightarrow E_i$ to its residue field, which is a finite extension of $E$.  By Proposition \ref{prop:wstarprop}.(e), the specialization $\phi \otimes_{\mathfrak{R}} E_i$ is the classical Sen operator, and the $\sigma$-factor of its Sen polynomial is $\xi_i'(P_\sigma(T)) \in E_i[T]$.  By condition (ii), $\xi'_i(P_\sigma(T))$ has a root of order exactly $m$ at 0.  Therefore the generalized eigenspace of $\phi \otimes_{\mathfrak{R}} E_i$ with eigenvalue 0 has dimension exactly equal to $m$, so $\dim_{E_i} \ker((\phi \otimes_{\mathfrak{R}} E_i)^m) = m$.  By Lemma \ref{lem:splitht}, $\coker(\phi_{Q_\sigma(0)}^m)$ is flat.  Applying Lemma \ref{lem:niceend} to the base change $\phi_{Q_\sigma(0)}^m \otimes_{\mathfrak{R}_{Q_\sigma(0)}} E_i = \phi^m \otimes_\mathfrak{R} E_i$ of $\phi_{Q_\sigma(0)}^m$ for the chosen $E_i$ in each connected component of $\spec \mathfrak{R}_{Q_\sigma(0)}$, we find $\rank_{\mathfrak{R}_{Q_\sigma(0)}} \ker(\phi_{Q_\sigma(0)}^m)=\dim_{E_i} \ker((\phi \otimes_\mathfrak{R} E_i)^m)=m.$
We apply Lemma \ref{lem:niceend} again with $M = \mathfrak{E}_{Q_\sigma(0)}$, $\varphi = \phi^m_{Q_\sigma(0)}$, and $N=\mathfrak{R}'$ to obtain the claim.
\end{proof}

\begin{claim} \label{claim:inclusioneq}
The inclusion $\ker(\phi_{Q_\sigma(0)}) \subseteq \ker(\phi^m_{Q_\sigma(0)})$ is an equality.
\end{claim}

\begin{proof}

This is tautological if $m=1$, so assume $m \ge 2$.  Let $N = \ker(\phi_{Q_\sigma(0)}^m)$ and let $N' = \ker(\phi_{Q_\sigma(0)})$.  By $\phi_{Q_\sigma(0)}$-equivariance of (\ref{eqn:esplit}), we obtain a map $\phi_N:N \rightarrow N$ by restricting $\phi_{Q_\sigma(0)}$ to $N$.  Moreover, we have $N' = \ker \phi_N$, since $\phi_{Q_\sigma(0)}^m|_{\ker(Q_\sigma(\phi_{Q_\sigma(0)}))}$ is an isomorphism by Lemma \ref{lem:splitht} and therefore $\phi_{Q_\sigma(0)}|_{\ker(Q_\sigma(\phi_{Q_\sigma(0)}))}$ is one as well.

Since $N$ is a locally free $\mathfrak{R}_{Q_\sigma(0)}$-module, we can pick $f_1,\dots,f_n \in \mathfrak{R}_{Q_\sigma(0)}$ such that the Zariski opens $\spec\mathfrak{R}_{Q_\sigma(0)f_j}$ cover $\spec\mathfrak{R}_{Q_\sigma(0)}$ and $N_{f_j}$ is free for each $j$.  Write $\phi_{N_{f_j}} = (\phi_N)_{f_j}$.  Then $\phi_{N_{f_j}}$ can be written as an $n\times n$ matrix over $\mathfrak{R}_{Q_\sigma(0)f_j}$.  Let $\mathfrak{m}_i$ be the maximal ideal of $\mathfrak{R}_i$ and let $I_j=\set{i: \xi_i(f_j) \notin\mathfrak{m}_i}$.  Fix $j$ and $i \in I_j$.  We have identifications
\begin{align}
	\nonumber \ker(\phi \otimes_\mathfrak{R} \mathfrak{R}_i) &= \ker(\phi_{Q_\sigma(0)} \otimes_{\mathfrak{R}_{Q_\sigma(0)}} \mathfrak{R}_i) = \ker(\phi_{Q_\sigma(0)} \otimes_{\mathfrak{R}_{Q_\sigma(0)}} \mathfrak{R}_i|_{\ker (\phi_{Q_\sigma(0)}^m \otimes_{\mathfrak{R}_{Q_\sigma(0)}} \mathfrak{R}_i)})\\
\label{eqn:bcri} 	&= \ker(\phi_{Q_\sigma(0)} \otimes_{\mathfrak{R}_{Q_\sigma(0)}} \mathfrak{R}_i|_{(\ker \phi_{Q_\sigma(0)}^m) \otimes_{\mathfrak{R}_{Q_\sigma(0)}} \mathfrak{R}_i}) = \ker(\phi_{Q_\sigma(0)}|_{\ker(\phi_{Q_\sigma(0)}^m)} \otimes_{\mathfrak{R}_{Q_\sigma(0)}} \mathfrak{R}_i)\\
\nonumber	&= \ker(\phi_N \otimes_{\mathfrak{R}_{Q_\sigma(0)}} \mathfrak{R}_i) = \ker(\phi_{N_{f_j}} \otimes_{\mathfrak{R}_{Q_\sigma(0)f_j}} \mathfrak{R}_i),
\end{align}
where we use (\ref{eqn:localswitch}) for the first identification, Lemma \ref{lem:niceend} with $M = \mathfrak{E}_{Q_\sigma(0)}$, $\varphi = \phi_{Q_\sigma(0)}$, and $N=\mathfrak{R}_i$ for the third, the fact that $\phi_{Q_\sigma(0)}$ is an isomorphism on the second factor of the $\phi_{Q_\sigma(0)}$-compatible splitting (\ref{eqn:esplit}) for the fourth, and (\ref{eqn:bclocal}) for the sixth.

By Proposition \ref{prop:wstarprop} and Corollary \ref{coro:sencoh}, $\ker(\phi \otimes_{\mathfrak{R}} \mathfrak{R}_i) = H^0(G_K,\mathbf{C}_p \hat{\otimes}_{K,\sigma} (\mathfrak{M} \otimes_\mathfrak{R} \mathfrak{R}_i)).$  We have $\ker(\phi \otimes_{\mathfrak{R}} \mathfrak{R}_i) \subseteq \ker((\phi \otimes_{\mathfrak{R}} \mathfrak{R}_i)^m)$, and the latter has dimension exactly $m\dim_E\mathfrak{R}_i$ by Claim \ref{claim:rankcalc}.  By condition (i) above, we must have $\ker(\phi\otimes_\mathfrak{R} \mathfrak{R}_i) = \ker((\phi\otimes_\mathfrak{R} \mathfrak{R}_i)^m)$, so by (\ref{eqn:bcri}) the specialization of $\phi_{N_{f_j}}$ to $\mathfrak{R}_i$ vanishes.

By Lemma \ref{lem:stabdens}, there is an injection $\mathfrak{R}_{Q_\sigma(0)f_j} \rightarrow \prod_{i \in I_j} \mathfrak{R}_i$.  We deduce that $\phi_{N_{f_j}}=0$ for each $j$ and thus that $\phi_N = 0$.
\end{proof}

We therefore have $\mathfrak{E}_{Q_\sigma(0)} = \ker(\phi_{Q_\sigma(0)}) \oplus \ker(Q_\sigma(\phi_{Q_\sigma(0)}))$ and, using Lemma \ref{lem:splitht},
\[\im(\phi_{Q_\sigma(0)})= \im(\phi_{Q_\sigma(0)}^m) = \ker(Q_\sigma(\phi_{Q_\sigma(0)})).\]
We deduce that $\ker(\phi_{Q_\sigma(0)}) \cong \coker(\phi_{Q_\sigma(0)})$ and obtain Theorem \ref{thm:htcase}.(a) from Claims \ref{claim:rankcalc} and \ref{claim:inclusioneq}.

For Theorem \ref{thm:htcase}.(b), we observe that if $\mathfrak{R}\rightarrow \mathfrak{R}'$ is a map of Noetherian $E$-Banach algebras, Proposition \ref{prop:wstarprop}.(d) and Corollary \ref{coro:sencoh} imply that
\[\ker(\phi \otimes_\mathfrak{R} \mathfrak{R}') \cong H^0(G_K,\mathbf{C}_p \hat{\otimes}_{K,\sigma} (\mathfrak{M} \otimes_{\mathfrak{R}} \mathfrak{R}')) \textrm{ and } \coker(\phi \otimes_\mathfrak{R} \mathfrak{R}') \cong H^1(G_K,\mathbf{C}_p \hat{\otimes}_{K,\sigma} (\mathfrak{M} \otimes_{\mathfrak{R}} \mathfrak{R}')).\]
The $r=1$ case of Theorem \ref{thm:htcase}.(b) follows from Lemma \ref{lem:cokerspec}, so we are left to check the $r=0$ case.

Using Theorem \ref{thm:htcase}.(a) with Lemma \ref{lem:niceend} for $\psi=\phi_{Q_\sigma(0)}$ and $N=\mathfrak{R}'_{Q_\sigma(0)}$, we have a natural isomorphism
\begin{equation} \label{eqn:htcaseb1} \ker(\phi_{Q_\sigma(0)}) \otimes_{\mathfrak{R}_{Q_\sigma(0)}} \mathfrak{R}'_{Q_\sigma(0)}  \iso \ker(\phi_{Q_\sigma(0)} \otimes_{\mathfrak{R}_{Q_\sigma(0)}} \mathfrak{R}'_{Q_\sigma(0)}).\end{equation}
Using Lemma \ref{lem:niceend} with the flat $\mathfrak{R}$-module $N=\mathfrak{R}_{Q_\sigma(0)}$, we obtain a natural isomorphism
\begin{equation} \label{eqn:htcaseb2} \ker(\phi)_{Q_\sigma(0)} \otimes_{\mathfrak{R}_{Q_\sigma(0)}} \mathfrak{R}'_{Q_\sigma(0)} \iso \ker(\phi_{Q_\sigma(0)}) \otimes_{\mathfrak{R}_{Q_\sigma(0)}} \mathfrak{R}'_{Q_\sigma(0)}.\end{equation}
Composing (\ref{eqn:htcaseb1}) and (\ref{eqn:htcaseb2}) and reasoning as in (\ref{eqn:bclocal}), we obtain
\begin{equation} \label{eqn:htcaseb3} \ker(\phi) \otimes_\mathfrak{R} \mathfrak{R}'_{Q_\sigma(0)} \iso \ker(\phi \otimes_\mathfrak{R} \mathfrak{R}'_{Q_\sigma(0)}).\end{equation}
Using Lemma \ref{lem:niceend} again with $\psi = \phi \otimes_\mathfrak{R} \mathfrak{R}'$ and $N=\mathfrak{R}_{Q_\sigma(0)}'$ on the right-hand side of (\ref{eqn:htcaseb3}), we have a natural isomorphism
\begin{equation} \label{eqn:htcaseb4} \ker(\phi \otimes_\mathfrak{R} \mathfrak{R}')_{Q_\sigma(0)} \iso \ker(\phi \otimes_\mathfrak{R} \mathfrak{R}'_{Q_\sigma(0)}). \end{equation}
The composition of (\ref{eqn:htcaseb3}) with the inverse of (\ref{eqn:htcaseb4}) gives Theorem \ref{thm:htcase}.(b) for $r=0$ by the finite generation in Proposition \ref{prop:htgap}.
\end{proof}

\subsection{Bounded de Rham periods} \label{subsec:drcase}

After localizing the modules $(\bdrk{k} \hat{\otimes}_{K,\sigma} \mathfrak{M})^{G_K}$ and $(\mathbf{C}_p \hat{\otimes}_{K,\sigma} \mathfrak{M})^{G_K}$ to disallow all integral Hodge-Tate-Sen weights up to $k-1$, and also localizing to force multiplicity 1 at weight 0, Kisin \cite[Proposition 2.5 and Corollary 2.6]{kisin} proves an isomorphism between these modules, which reduces a rank 1 de Rham version of Theorem \ref{thm:htcase} to the Hodge-Tate case.  This is possible only in families with a unique fixed Hodge-Tate-Sen weight of multiplicity 1.  Our results will require a dense set of points with de Rham periods for the fixed Hodge-Tate weights.

Let $B_\mathrm{dR}^+$ be equipped with its canonical topology, and let $t \in \bdr^+$ denote the usual choice of uniformizer.  Fix a positive integer $n$, an element $\sigma \in \Sigma$, and a multiset $\set{w_{1,\sigma}, \dots, w_{n,\sigma}}$ of nonnegative integers.  We require that the $\sigma$-factor of the Sen polynomial of $\mathfrak{M}$ factors as $P_\sigma(T) = S_\sigma(T)Q_\sigma(T)$, where $S_\sigma(T) = \prod_{i=1}^n (T+w_{i,\sigma}) \in \mathbf{Z}[T]$ and $Q_\sigma(T)$ is arbitrary.  Define $Q_{\sigma,k} = \prod_{j=0}^{k-1}Q_\sigma(-j)$.  Notice that localization at $Q_{\sigma,k}$ fixes the multiplicity of each integral Hodge-Tate-Sen weight in the interval $[0,k-1]$.

\begin{thm} \label{thm:drcase}

Fix a positive integer $k$ and let $d_k=\#\set{w_{j,\sigma}|w_{j,\sigma} < k}$.  Suppose that there exists a collection $\set{\xi_i}_{i \in I}$ of maps $\xi_i: \mathfrak{R}\rightarrow \mathfrak{R}_i$ of $E$-Banach algebras, where $\mathfrak{R}_i$ is local Artinian of finite dimension over $E$, satisfying the following properties.
\begin{enumerate}[{\normalfont (i)}]
  \item For $i \in I$, $\dim_E H^0(G_K,\bdrk{k} \otimes_{K,\sigma} (\mathfrak{M} \otimes_{\mathfrak{R}} \mathfrak{R}_i))=d_k\dim_E \mathfrak{R}_i$.
  \item For $i \in I$, the image of $Q_{\sigma,k}$ in $\mathfrak{R}_i$ is a unit.
  \item The collection $\set{\xi_i}$ is stably dense in $\mathfrak{R}$.
\end{enumerate}
Then we have the following.
\begin{enumerate}[{\normalfont (a)}]
	\item For $r \in \set{0,1}$, the $\mathfrak{R}_{Q_{\sigma,k}}$-module $H^r(G_K,\bdrk{k} \hat{\otimes}_{K,\sigma} \mathfrak{M})_{Q_{\sigma,k}}$ is finite flat of rank $d_k$.
	\item For any map $\xi:\mathfrak{R} \rightarrow \mathfrak{R}'$ of Noetherian $E$-Banach algebras, the cokernel and kernel of the map
\begin{equation}\label{eqn:bdrspec}
H^r(G_K,\bdrk{k} \hat{\otimes}_{K,\sigma} \mathfrak{M}) \otimes_{\mathfrak{R}} \mathfrak{R}' \rightarrow H^r(G_K,\bdrk{k} \hat{\otimes}_{K,\sigma} (\mathfrak{M} \otimes_{\mathfrak{R}} \mathfrak{R}'))
\end{equation}
are killed by a power of $\xi(Q_{\sigma,k})$ for $r \in \set{0,1}$.
\end{enumerate}

\end{thm}

\begin{proof}

We prove Theorem \ref{thm:drcase}.(a) and (b) by induction on $k$; we will also need the $k$ case of (a) to prove the $k$ case of (b).  If $k=1$, the results follow from Theorem \ref{thm:htcase}.

\begin{claim}\label{claim:hkvanish}
We have $H^1(H_K,\bdrk{k}\hat{\otimes}_{K,\sigma} \mathfrak{M})=0$.
\end{claim}

\begin{proof}
This follows from the result of Sen \cite[Proposition 2]{sen} mentioned in Corollary \ref{coro:sencoh} and induction using the exact sequence
\begin{equation} \label{eqn:mainexact} 0 \rightarrow \mathbf{C}_p(k-1)\hat{\otimes}_{K,\sigma} \mathfrak{M} \rightarrow \bdrk{k}\hat{\otimes}_{K,\sigma} \mathfrak{M} \rightarrow \bdrk{k-1}\hat{\otimes}_{K,\sigma} \mathfrak{M} \rightarrow 0.\end{equation}
\end{proof}

\begin{remark} \label{remark:hkvanish}
We record for future reference that a continuous additive section to the surjection in the exact sequence (\ref{eqn:mainexact}) exists by \cite[Remark 3.2]{iz}.
\end{remark}

Fix a topological generator $\gamma \in \Gamma_K$.  The diagram
\[\xymatrix@C=10pt@R=15pt{ 0 \ar[r] & (\mathbf{C}_p(k-1)\hat{\otimes}_{K,\sigma} \mathfrak{M})^{H_K} \ar[r] \ar[d]^{\gamma-1} & (\bdrk{k} \hat{\otimes}_{K,\sigma} \mathfrak{M})^{H_K} \ar[r] \ar[d]^{\gamma-1}\ar[r]\ar[d]^{\gamma-1} & (\bdrk{k-1}\hat{\otimes}_{K,\sigma} \mathfrak{M})^{H_K} \ar[r]\ar[d]^{\gamma-1} & 0 \\ 0 \ar[r] & (\mathbf{C}_p(k-1)\hat{\otimes}_{K,\sigma} \mathfrak{M})^{H_K} \ar[r] & (\bdrk{k} \hat{\otimes}_{K,\sigma} \mathfrak{M})^{H_K} \ar[r] \ar[r] & (\bdrk{k-1}\hat{\otimes}_{K,\sigma} \mathfrak{M})^{H_K} \ar[r] & 0}\]
has exact rows by the vanishing of $H^1(H_K,\mathbf{C}_p(k-1)\hat{\otimes}_{K,\sigma} \mathfrak{M})$.  By inflation-restriction, Claim \ref{claim:hkvanish}, and Lemma \ref{lem:zpcoh}, the exact sequence associated to the above diagram using the snake lemma is
\begin{align}
	\label{eqn:exactsnake} 0 &\rightarrow (\mathbf{C}_p(k-1) \hat{\otimes}_{K,\sigma} \mathfrak{M})^{G_K} \rightarrow (\bdrk{k} \hat{\otimes}_{K,\sigma} \mathfrak{M})^{G_K}\\
	\nonumber &\rightarrow (\bdrk{k-1} \hat{\otimes}_{K,\sigma} \mathfrak{M})^{G_K} \stackrel{\psi}{\rightarrow} H^1(G_K,\mathbf{C}_p(k-1) \hat{\otimes}_{K,\sigma} \mathfrak{M})\\
	\nonumber &\rightarrow H^1(G_K,\bdrk{k} \hat{\otimes}_{K,\sigma} \mathfrak{M}) \rightarrow H^1(G_K,\bdrk{k-1} \hat{\otimes}_{K,\sigma} \mathfrak{M}) \rightarrow 0.
\end{align}

\begin{claim} \label{claim:0map}

The map $\psi_{Q_{\sigma,k}}$ is 0.

\end{claim}

\begin{proof}

Consider the 6-term exact sequence associated to
\[0 \rightarrow \mathbf{C}_p(k-1)\otimes_{K,\sigma} (\mathfrak{M} \otimes_\mathfrak{R} \mathfrak{R}_i) \rightarrow \bdrk{k}\otimes_{K,\sigma} (\mathfrak{M} \otimes_\mathfrak{R} \mathfrak{R}_i) \rightarrow \bdrk{k-1}\otimes_{K,\sigma} (\mathfrak{M} \otimes_\mathfrak{R} \mathfrak{R}_i) \rightarrow 0\]
by the method that produced (\ref{eqn:exactsnake}).  Dimension counting over $E$ using condition (i) above,
\[(\bdrk{k} \otimes_{K,\sigma} (\mathfrak{M} \otimes_\mathfrak{R} \mathfrak{R}_i))^{G_K} \rightarrow (\bdrk{k-1} \otimes_{K,\sigma} (\mathfrak{M} \otimes_\mathfrak{R} \mathfrak{R}_i))^{G_K}\]
is surjective, so
\[(\bdrk{k-1} \otimes_{K,\sigma} (\mathfrak{M} \otimes_\mathfrak{R} \mathfrak{R}_i))^{G_K} \stackrel{\psi_i}{\rightarrow} H^1(G_K,\mathbf{C}_p(k-1) \otimes_{K,\sigma} (\mathfrak{M} \otimes_\mathfrak{R} \mathfrak{R}_i))\]
is the 0 map.

The commutative diagram
\[ \xymatrix@R=15pt{ (\bdrk{k-1} \hat{\otimes}_{K,\sigma} \mathfrak{M})^{G_K}_{Q_{\sigma,k}}\ar[r]^{\psi_{Q_{\sigma,k}}}\ar[d] & H^1(G_K,\mathbf{C}_p(k-1) \hat{\otimes}_{K,\sigma} \mathfrak{M})_{Q_{\sigma,k}}\ar[d]\\
(\bdrk{k-1} \hat{\otimes}_{K,\sigma} \mathfrak{M})^{G_K}\otimes_\mathfrak{R} \mathfrak{R}_i\ar[r]\ar[d]^{\begin{sideways}$\sim$\end{sideways}} & H^1(G_K,\mathbf{C}_p(k-1) \hat{\otimes}_{K,\sigma} \mathfrak{M}) \otimes_\mathfrak{R} \mathfrak{R}_i\ar[d]^{\begin{sideways}$\sim$\end{sideways}}\\
(\bdrk{k-1} \otimes_{K,\sigma} (\mathfrak{M} \otimes_\mathfrak{R} \mathfrak{R}_i))^{G_K}\ar[r]^{\psi_i} & H^1(G_K,\mathbf{C}_p(k-1) \otimes_{K,\sigma} (\mathfrak{M} \otimes_\mathfrak{R} \mathfrak{R}_i))}\]
has the marked downward isomorphisms by the inductive hypothesis and Theorem \ref{thm:htcase}.(b).  By Theorem \ref{thm:htcase}.(a), the $\mathfrak{R}_{Q_{\sigma,k}}$-module $H^1(G_K,\mathbf{C}_p(k-1) \hat{\otimes}_{K,\sigma} \mathfrak{M})_{Q_{\sigma,k}}$ is locally free, so we may pick elements $f_j \in \mathfrak{R}_{Q_{\sigma,k}}$ such that $H^1(G_K,\mathbf{C}_p(k-1) \hat{\otimes}_{K,\sigma} \mathfrak{M})_{Q_{\sigma,k}f_j}$ is free for each $j$, and such that the $\spec \mathfrak{R}_{Q_{\sigma,k}f_j}$ cover $\spec\mathfrak{R}_{Q_{\sigma,k}}$.

Let $i$ be such that $f_j \notin \mathfrak{m}_i$, where $\mathfrak{m}_i$ denotes the maximal ideal of $\mathfrak{R}_i$.  Looking at the diagram localized at $f_j$ (which only affects the top row), the image of $\psi_{f_j Q_{\sigma,k}}$ lies in the submodule $\ker\xi_i\cdot H^1(G_K,\mathbf{C}_p(k-1) \hat{\otimes}_{K,\sigma} \mathfrak{M})_{f_jQ_{\sigma,k}}$.  By Lemma \ref{lem:stabdens}, there is an injection $\mathfrak{R}_{Q_{\sigma,k}f_j} \rightarrow \prod_{\set{i:f_j \notin \mathfrak{m}_i}} \mathfrak{R}_i$.  By this and freeness, $\bigcap_{\set{i:f_j \notin \mathfrak{m}_i}}\ker\xi_i\cdot H^1(G_K,\mathbf{C}_p(k-1) \hat{\otimes}_{K,\sigma} \mathfrak{M})_{f_jQ_{\sigma,k}} = 0,$ so $\psi_{f_j Q_{\sigma,k}}=0$ for all $j$ and thus $\psi_{Q_{\sigma,k}}=0$.
\end{proof}

From the claim, we have exact sequences
\begin{align}\label{eqn:h0exact} 0 \rightarrow (\mathbf{C}_p(k-1) \hat{\otimes}_{K,\sigma} \mathfrak{M})^{G_K}_{Q_{\sigma,k}} &\rightarrow (\bdrk{k} \hat{\otimes}_{K,\sigma} \mathfrak{M})^{G_K}_{Q_{\sigma,k}}\\
	\nonumber &\rightarrow (\bdrk{k-1} \hat{\otimes}_{K,\sigma} \mathfrak{M})^{G_K}_{Q_{\sigma,k}} \rightarrow 0\\
\label{eqn:h1exact}\textrm{and}\quad 0 \rightarrow H^1(G_K,\mathbf{C}_p(k-1) \hat{\otimes}_{K,\sigma} \mathfrak{M})_{Q_{\sigma,k}} &\rightarrow H^1(G_K,\bdrk{k} \hat{\otimes}_{K,\sigma} \mathfrak{M})_{Q_{\sigma,k}}\\
	\nonumber &\rightarrow H^1(G_K,\bdrk{k-1} \hat{\otimes}_{K,\sigma} \mathfrak{M})_{Q_{\sigma,k}} \rightarrow 0.
\end{align}
The $k$ case of Theorem \ref{thm:drcase}.(a) follows immediately from these, Theorem \ref{thm:htcase}.(a), and the inductive hypothesis.

For Theorem \ref{thm:drcase}.(b), we begin by observing that by Proposition \ref{prop:htgap}, the fact that $\mathfrak{R}$ is Noetherian, and induction using (\ref{eqn:mainexact}), both sides of (\ref{eqn:bdrspec}) are finitely generated, so it suffices to prove that (\ref{eqn:bdrspec}) becomes an isomorphism after localization at $\xi(Q_{\sigma,k})$.

We have a natural map between 6-term complexes of finitely generated $\mathfrak{R}$-modules
\begin{align}
	\label{eqn:bcexactsnake1} 0 &\rightarrow (\mathbf{C}_p(k-1) \hat{\otimes}_{K,\sigma} \mathfrak{M})^{G_K} \otimes_\mathfrak{R} \mathfrak{R}' \rightarrow (\bdrk{k} \hat{\otimes}_{K,\sigma} \mathfrak{M})^{G_K} \otimes_\mathfrak{R} \mathfrak{R}'\\
	\nonumber &\rightarrow (\bdrk{k-1} \hat{\otimes}_{K,\sigma} \mathfrak{M})^{G_K}  \otimes_\mathfrak{R} \mathfrak{R}' \stackrel{\psi_1}{\rightarrow} H^1(G_K,\mathbf{C}_p(k-1) \hat{\otimes}_{K,\sigma} \mathfrak{M}) \otimes_\mathfrak{R} \mathfrak{R}'\\
	\nonumber &\rightarrow H^1(G_K,\bdrk{k} \hat{\otimes}_{K,\sigma} \mathfrak{M}) \otimes_\mathfrak{R} \mathfrak{R}' \rightarrow H^1(G_K,\bdrk{k-1} \hat{\otimes}_{K,\sigma} \mathfrak{M}) \otimes_\mathfrak{R} \mathfrak{R}' \rightarrow 0
\end{align}
and
\begin{align}
	\label{eqn:bcexactsnake2} 0 &\rightarrow (\mathbf{C}_p(k-1) \hat{\otimes}_{K,\sigma} (\mathfrak{M} \otimes_\mathfrak{R} \mathfrak{R}'))^{G_K} \rightarrow (\bdrk{k} \hat{\otimes}_{K,\sigma} (\mathfrak{M} \otimes_\mathfrak{R} \mathfrak{R}'))^{G_K}\\
	\nonumber &\rightarrow (\bdrk{k-1} \hat{\otimes}_{K,\sigma} (\mathfrak{M} \otimes_\mathfrak{R} \mathfrak{R}'))^{G_K} \stackrel{\psi_2}{\rightarrow} H^1(G_K,\mathbf{C}_p(k-1) \hat{\otimes}_{K,\sigma} (\mathfrak{M} \otimes_\mathfrak{R} \mathfrak{R}'))\\
	\nonumber &\rightarrow H^1(G_K,\bdrk{k} \hat{\otimes}_{K,\sigma} (\mathfrak{M} \otimes_\mathfrak{R} \mathfrak{R}')) \rightarrow H^1(G_K,\bdrk{k-1} \hat{\otimes}_{K,\sigma} (\mathfrak{M} \otimes_\mathfrak{R} \mathfrak{R}')) \rightarrow 0,
\end{align}
where the second complex is exact. Localizing at $Q_{\sigma,k}$ and applying part (a) and (\ref{eqn:bclocal}), the first complex becomes exact as well, and by Claim \ref{claim:0map}, the map $\psi_{1,Q_{\sigma,k}}$ is 0.  By the inductive hypothesis, the first, third, fourth, and sixth downward maps are isomorphisms.  It follows that $\psi_{2,Q_{\sigma,k}}=0$, and by the 5-lemma, we obtain (\ref{eqn:bdrspec}).
\end{proof}

\subsection{Geometric specializations} \label{subsec:special}

We can improve Theorem \ref{thm:drcase}.(b).  The key step is to show that the formation of the cohomology group $H^1(G_K,\bdrk{k} \hat{\otimes}_{K,\sigma} \mathfrak{M})$ is compatible with finite base change.  We need the following lemma.

\begin{lem} \label{lem:dimh0h1}

Maintain the notation of Section \ref{subsec:senkisin}. If $\mathfrak{R}$ is an Artinian $E$-Banach algebra of finite dimension over $E$, then for any $k$,
\begin{equation} \label{eqn:bdrk} \dim_E H^0(G_K,\bdrk{k} \otimes_{K,\sigma} \mathfrak{M}) = \dim_E H^1(G_K,\bdrk{k} \otimes_{K,\sigma} \mathfrak{M}).\end{equation}

\end{lem}

\begin{proof}
If $k=1$, using Lemma \ref{lem:changeek}, we may replace $E,K$ with finite Galois extensions $E',K'$ so that Corollary \ref{coro:sencoh} holds and deduce the original result.  After doing so, we have $H^0(G_K,\mathbf{C}_p \otimes_{K,\sigma} \mathfrak{M}) \cong \ker\phi \textrm{ and } H^1(G_K,\mathbf{C}_p \otimes_{K,\sigma} \mathfrak{M}) \cong \coker\phi$ for a map $\phi:\mathfrak{E} \rightarrow \mathfrak{E}$ of finite dimensional $E$-vector spaces. Lemma \ref{lem:niceend} implies that $\dim_E \ker\phi = \dim_E \coker\phi$ as needed.  We obtain (\ref{eqn:bdrk}) by induction on $k$ using (\ref{eqn:exactsnake}) and dimension counting over $E$.
\end{proof}

In the remainder of this section, we will always assume that $\mathfrak{R}$ is an affinoid $E$-algebra.

\begin{lem} \label{lem:ctp}

Maintain the notation of Section \ref{subsec:senkisin}, and assume that $\mathfrak{R}$ is $E$-affinoid.  For any finite $\mathfrak{R}$-module $\mathfrak{N}$, the tensor product $(\mathbf{C}_p \hat{\otimes}_{K,\sigma} \mathfrak{M})^{H_K} \otimes_\mathfrak{R} \mathfrak{N}$ is already complete.

\end{lem}

\begin{proof}
By Proposition \ref{prop:wstarprop}, if $E$ and $K$ are sufficiently large, $(\mathbf{C}_p \hat{\otimes}_{K,\sigma} \mathfrak{M})^{H_K}$ is a finite $\widehat{K}_\infty \hat{\otimes}_{K,\sigma} \mathfrak{R}$-module.  We claim that the same holds for the original $E$ and $K$.  Using (\ref{eqn:edescalc}) and Lemma \ref{lem:galdes}.(a), one sees that finiteness holds over the original $E$ if it holds over an extension $E'$.

Assume that $K'$ is a Galois extension of $K$ and $\sigma': K' \rightarrow E$ extends $\sigma$.  Observe that $\mathbf{C}_p \hat{\otimes}_{K,\sigma} \mathfrak{M} \cong \mathbf{C}_p^{\set{\tau|\tau\in\gal(K'/K)}} \hat{\otimes}_{K_\mathrm{triv}',\sigma'} \mathfrak{M}$ by (\ref{eqn:switchfield}), where the action of $H_{K'}$ on the right is trivial on the $\tau$, so
\begin{align}
	\label{eqn:cpflatcalc} (\mathbf{C}_p \hat{\otimes}_{K,\sigma} \mathfrak{M})^{H_K} &\cong((\mathbf{C}_p \hat{\otimes}_{K,\sigma} \mathfrak{M})^{H_{K'}})^{H_K/H_{K'}}\cong (({\mathbf{C}_p}^{\set{\tau|\tau\in\gal(K'/K)}} \hat{\otimes}_{K_\mathrm{triv}',\sigma'} \mathfrak{M})^{H_{K'}})^{H_K/H_{K'}}\\
	\nonumber &\cong (((\mathbf{C}_p \hat{\otimes}_{K_\mathrm{triv}',\sigma'} \mathfrak{M})^{H_{K'}})^{\set{\tau|\tau\in\gal(K'/K)}})^{H_K/H_{K'}}.
\end{align}
One may similarly identify $\widehat{K}_\infty \hat{\otimes}_{K,\sigma} \mathfrak{R}$ with $((\widehat{K}_\infty' \hat{\otimes}_{K_\mathrm{triv}',\sigma'} \mathfrak{R})^{\set{\tau|\tau\in\gal(K'/K)}})^{H_K/H_{K'}}$.  It now follows from the finiteness of $(\mathbf{C}_p \hat{\otimes}_{K_\mathrm{triv}',\sigma'} \mathfrak{M})^{H_{K'}}$ over $\widehat{K}_\infty' \hat{\otimes}_{K_\mathrm{triv}',\sigma'} \mathfrak{R}$ and thus finiteness (by Lemma \ref{lem:galdes}.(a)) of $((\mathbf{C}_p \hat{\otimes}_{K_\mathrm{triv}',\sigma'} \mathfrak{M})^{H_{K'}})^{\set{\tau|\tau\in\gal(K'/K)}}$ over $\widehat{K}_\infty \hat{\otimes}_{K,\sigma} \mathfrak{R}$ that $(\mathbf{C}_p \hat{\otimes}_{K,\sigma} \mathfrak{M})^{H_K}$ is a finite $\widehat{K}_\infty \hat{\otimes}_{K,\sigma} \mathfrak{R}$-module.

We have a presentation
\begin{equation}\label{eqn:presentationn}\mathfrak{R}^n \stackrel{\psi_\mathfrak{N}}{\rightarrow} \mathfrak{R}^m \rightarrow \mathfrak{N} \rightarrow 0.\end{equation}
Writing $\psi = (\mathbf{C}_p \hat{\otimes}_{K,\sigma} \mathfrak{M})^{H_K} \otimes_\mathfrak{R} \psi_\mathfrak{N}$, the morphism $((\mathbf{C}_p \hat{\otimes}_{K,\sigma} \mathfrak{M})^{H_K})^n \stackrel{\psi}{\rightarrow} ((\mathbf{C}_p \hat{\otimes}_{K,\sigma} \mathfrak{M})^{H_K})^m$ is strict, since it is a $\widehat{K}_\infty \hat{\otimes}_{K,\sigma} \mathfrak{R}$-linear homomorphism of finite $\widehat{K}_\infty \hat{\otimes}_{K,\sigma} \mathfrak{R}$-modules, where the ring $\widehat{K}_\infty \hat{\otimes}_{K,\sigma} \mathfrak{R}$ is Noetherian since $\mathfrak{R}$ is affinoid \cite[Proposition 3.7.3/5 and Corollary 6.1.1/9]{bgr}.

By applying the completion functor of \cite[Proposition 1.1.7/6]{bgr} to the tensor product of (\ref{eqn:presentationn}) with $(\mathbf{C}_p \hat{\otimes}_{K,\sigma} \mathfrak{M})^{H_K}$, we obtain the commutative diagram
\[\xymatrix@C=9pt@R=15pt{  ((\mathbf{C}_p \hat{\otimes}_{K,\sigma} \mathfrak{M})^{H_K})^n \ar[r]^\psi \ar[d]^{\begin{sideways}=\end{sideways}} &  ((\mathbf{C}_p \hat{\otimes}_{K,\sigma} \mathfrak{M})^{H_K})^m \ar[d]^{\begin{sideways}=\end{sideways}} \ar[r] &  (\mathbf{C}_p \hat{\otimes}_{K,\sigma} \mathfrak{M})^{H_K} \otimes_\mathfrak{R} \mathfrak{N} \ar[d]\ar[r] & 0 \\((\mathbf{C}_p \hat{\otimes}_{K,\sigma} \mathfrak{M})^{H_K})^n \ar[r]^{\hat{\psi}} &  ((\mathbf{C}_p \hat{\otimes}_{K,\sigma} \mathfrak{M})^{H_K})^m \ar[r] &  (\mathbf{C}_p \hat{\otimes}_{K,\sigma} \mathfrak{M})^{H_K} \hat{\otimes}_\mathfrak{R} \mathfrak{N} \ar[r] & 0}\]
with downwards maps given by mapping each term its completion.  Note that $\hat{\psi} = \psi$ since $\psi$ is already a map of complete spaces.  By right-exactness of the tensor product, the top row of the diagram is exact.  We just checked that the first map of the top row is strict, and the second is strict by \cite[Proposition 2.1.8/6]{bgr}.  The bottom row is exact by \cite[Corollary 1.1.9/6]{bgr}, giving the needed isomorphism by the 5-lemma.
\end{proof}

We use this to prove a base change result for $H_K$-invariants.

\begin{lem}\label{lem:drbchk}
Maintain the notation of Section \ref{subsec:senkisin}, and assume that $\mathfrak{R}$ is $E$-affinoid.  Let $\mathfrak{R}'$ be a finite $\mathfrak{R}$-algebra.  We have a natural $\Gamma_K$-equivariant isomorphism of topological $\mathfrak{R}'$-modules
\begin{equation}\label{eqn:drbchk}(\bdrk{k} \hat{\otimes}_{K,\sigma} \mathfrak{M})^{H_K} \otimes_{\mathfrak{R}} \mathfrak{R}' \rightarrow (\bdrk{k} \hat{\otimes}_{K,\sigma} (\mathfrak{M}\otimes_{\mathfrak{R}} \mathfrak{R}'))^{H_K}.\end{equation}

\end{lem}

\begin{proof}
By using the analysis of (\ref{eqn:redbexact}) in the proof of Theorem \ref{thm:htcase} with $G_K$ replaced by $H_K$ and $r=0$, we can replace $E$ by any finite Galois extension and deduce the original result.

We may therefore replace $K$ by a sufficiently large finite Galois extension $K'$, equipped with an extension of $\sigma$ to $\sigma':K' \rightarrow E$, so that Proposition \ref{prop:wstarprop} holds.  By Proposition \ref{prop:wstarprop}.(d), we have an isomorphism
\[(\mathbf{C}_p \hat{\otimes}_{K_\mathrm{triv}',\sigma'} \mathfrak{M})^{H_{K'}} \hat{\otimes}_{\mathfrak{R}} \mathfrak{R}' \iso (\mathbf{C}_p \hat{\otimes}_{K_\mathrm{triv}',\sigma'} (\mathfrak{M}\otimes_{\mathfrak{R}} \mathfrak{R}'))^{H_{K'}}.\]
Using Lemma \ref{lem:ctp}, we can replace the left hand side by $(\mathbf{C}_p \hat{\otimes}_{K_\mathrm{triv}',\sigma'} \mathfrak{M})^{H_{K'}} \otimes_{\mathfrak{R}} \mathfrak{R}'$.  We take a direct sum over $\tau \in \gal(K'/K)$ to obtain a $G_K$-equivariant isomorphism
\begin{equation} \label{eqn:bcdes1} ((\mathbf{C}_p \hat{\otimes}_{K_\mathrm{triv}',\sigma'} \mathfrak{M})^{H_{K'}} \otimes_{\mathfrak{R}} \mathfrak{R}')^{\set{\tau \in \gal(K'/K)}} \iso ((\mathbf{C}_p \hat{\otimes}_{K_\mathrm{triv}',\sigma'} (\mathfrak{M}\otimes_{\mathfrak{R}} \mathfrak{R}'))^{H_{K'}})^{\set{\tau \in \gal(K'/K)}}.\end{equation}
We define a $G_K$-equivariant isomorphism
\begin{equation} \label{eqn:bcdes2} ((\mathbf{C}_p \hat{\otimes}_{K_\mathrm{triv}',\sigma'} \mathfrak{M})^{H_{K'}})^{\set{\tau \in \gal(K'/K)}} \otimes_{\mathfrak{R}} \mathfrak{R}' \iso ((\mathbf{C}_p \hat{\otimes}_{K_\mathrm{triv}',\sigma'} \mathfrak{M})^{H_{K'}} \otimes_{\mathfrak{R}} \mathfrak{R}')^{\set{\tau \in \gal(K'/K)}}\end{equation}
as in (\ref{eqn:kdes1}).  By (\ref{eqn:cpflatcalc}), $(\mathbf{C}_p \hat{\otimes}_{K,\sigma} \mathfrak{M})^{H_K} \cong (((\mathbf{C}_p \hat{\otimes}_{K_\mathrm{triv}',\sigma'} \mathfrak{M})^{H_{K'}})^{\set{\tau|\tau\in\gal(K'/K)}})^{H_K/H_{K'}}$.  We finish the descent to $K$ using Lemma \ref{lem:galdes}.(b) on the $H_K/H_{K'}$ invariants of the composition of (\ref{eqn:bcdes1}) and (\ref{eqn:bcdes2}).

Let $H_K'$ denote the intersection of $H_K$ with the kernel of the cyclotomic character and write $L_\mathrm{dR}^+ = (\bdr^+)^{H_K'}$.  For general $k$, there is a construction (after replacing $K$ with a sufficiently large finite Galois extension, which in the notation of \cite[Lemma 4.3.1]{bc} corresponds to taking $L$ and $n$ sufficiently large) of a finite free module $D_\mathrm{dif}^+(\mathfrak{M})$ over $K[[t]] \hat{\otimes}_{K,\sigma} \mathfrak{R}$ such that
\begin{align*}
	(D_\mathrm{dif}^+(\mathfrak{M}) \otimes_{K[[t]] \hat{\otimes}_{K,\sigma} \mathfrak{R}} B_\mathrm{dR}^+ \hat{\otimes}_{K,\sigma} \mathfrak{R})^{H_K'} &= D_\mathrm{dif}^+(\mathfrak{M}) \otimes_{K[[t]] \hat{\otimes}_{K,\sigma} \mathfrak{R}} (L_\mathrm{dR}^+ \hat{\otimes}_{K,\sigma} \mathfrak{R})\\
	&= (\bdr^+ \hat{\otimes}_{K,\sigma} \mathfrak{M})^{H_K'}.
\end{align*}
Here the first equality uses the fact that $D_\mathrm{dif}^+(\mathfrak{M})$ is free and has no $H_K'$ action to move the $H_K'$ invariants to $B_\mathrm{dR}^+ \hat{\otimes}_{K,\sigma} \mathfrak{R}$ and a standard Schauder basis argument (c.f.\ \cite[Lemma 2.6]{kliu}) to move the $H_K'$ invariants onto $\bdr^+$.  For simplicity, we have dropped the $n$ from the notation of \cite[Lemma 4.3.1]{bc} and substituted $K$ for $L_n$.  In particular, $(\bdrk{k}\hat{\otimes}_{K,\sigma} \mathfrak{M})^{H_K'}$ is identified with
\[D_\mathrm{dif}^+(\mathfrak{M})/t^k D_\mathrm{dif}^+(\mathfrak{M}) \otimes_{K_\infty[t]/(t^k) \hat{\otimes}_{K,\sigma} \mathfrak{R}} (L_\mathrm{dR}^+/t^kL_\mathrm{dR}^+ \hat{\otimes}_{K,\sigma} \mathfrak{R}).\]
The construction of $D_\mathrm{dif}^+(\mathfrak{M})$ is compatible with base change along any morphism $\mathfrak{R} \rightarrow \mathfrak{R}'$ of affinoid algebras with respect to taking completed tensor products (as follows, for instance, from \cite[Theorem 3.11.(4)]{kliu}).  We obtain an isomorphism
\begin{equation} \label{eqn:drbchk1}(\bdrk{k} \hat{\otimes}_{K,\sigma} \mathfrak{M})^{H_K'} \hat{\otimes}_{\mathfrak{R}} \mathfrak{R}' \rightarrow (\bdrk{k} \hat{\otimes}_{K,\sigma} (\mathfrak{M}\otimes_{\mathfrak{R}} \mathfrak{R}'))^{H_K'},\end{equation}
which by descent arguments as before holds for the original $K$ and $H_K$.

We now show that (\ref{eqn:drbchk}) is an isomorphism by induction on $k$.  Consider the natural morphism of complexes between
\begin{align*}
	0 \rightarrow (\mathbf{C}_p(k-1) \hat{\otimes}_{K,\sigma} \mathfrak{M})^{H_K} \otimes_{\mathfrak{R}} \mathfrak{R}' &\rightarrow (\bdrk{k} \hat{\otimes}_{K,\sigma} \mathfrak{M})^{H_K}\otimes_{\mathfrak{R}} \mathfrak{R}'\\&
	\rightarrow (\bdrk{k-1} \hat{\otimes}_{K,\sigma} \mathfrak{M})^{H_K}\otimes_{\mathfrak{R}} \mathfrak{R}' \rightarrow 0\\
	\textrm{and}\quad 0 \rightarrow (\mathbf{C}_p(k-1) \hat{\otimes}_{K,\sigma} (\mathfrak{M}\otimes_{\mathfrak{R}} \mathfrak{R}'))^{H_K} &\rightarrow (\bdrk{k} \hat{\otimes}_{K,\sigma} (\mathfrak{M}\otimes_{\mathfrak{R}} \mathfrak{R}'))^{H_K}\\
	&\rightarrow (\bdrk{k-1} \hat{\otimes}_{K,\sigma} (\mathfrak{M}\otimes_{\mathfrak{R}} \mathfrak{R}'))^{H_K} \rightarrow 0.
\end{align*}
The first row is right exact and the second is left exact by Claim \ref{claim:hkvanish} and right-exactness of the tensor product, and the outer maps are isomorphisms by the inductive hypothesis and $k=1$ case above.  The middle map is an algebraic isomorphism by the 5-lemma.  By the universal property of completion, the middle map factors through the isomorphism (\ref{eqn:drbchk1}).  Since the map from a space to its completion is homeomorphic onto its image, the map (\ref{eqn:drbchk}) is also a topological isomorphism.
\end{proof}

We get the de Rham case of Lemma \ref{lem:ctp} as a corollary.  We use these to obtain a base change result for cohomology.

\begin{prop} \label{prop:dr1cocycle}

Maintain the notation of Section \ref{subsec:senkisin}, and assume that $\mathfrak{R}$ is $E$-affinoid.  Let $\mathfrak{R}'$ be a finite $\mathfrak{R}$-algebra.  The natural map
\begin{equation}\label{eqn:dr1cocycle} H^1(G_K,\bdrk{k} \hat{\otimes}_{K,\sigma} \mathfrak{M})\otimes_{\mathfrak{R}} \mathfrak{R}' \rightarrow H^1(G_K,\bdrk{k} \hat{\otimes}_{K,\sigma} (\mathfrak{M}\otimes_{\mathfrak{R}} \mathfrak{R}'))\end{equation}
is an isomorphism of (non-topological) $\mathfrak{R}$-modules.

\end{prop}

\begin{proof}
Write $\psi = \gamma-1: (\bdrk{k} \hat{\otimes}_{K,\sigma} \mathfrak{M})^{H_K} \rightarrow (\bdrk{k} \hat{\otimes}_{K,\sigma} \mathfrak{M})^{H_K}.$  The $\mathfrak{R}'$-module isomorphism
\[(\bdrk{k} \hat{\otimes}_{K,\sigma} \mathfrak{M})^{H_K} \otimes_\mathfrak{R} \mathfrak{R}' \iso (\bdrk{k} \hat{\otimes}_{K,\sigma} (\mathfrak{M}\otimes_\mathfrak{R} \mathfrak{R}'))^{H_K}\]
of Lemma \ref{lem:drbchk} identifies $\psi \otimes_\mathfrak{R} \mathfrak{R}'$ with the natural action of $\gamma-1$ on the right hand side.  Observe that by this and Lemma \ref{lem:zpcoh}, we have identifications
\begin{align*} \coker(\psi) &\cong H^1(\Gamma_K,(\bdrk{k} \hat{\otimes}_{K,\sigma} \mathfrak{M})^{H_K})\\ \textrm{and}\quad \coker(\psi \otimes_\mathfrak{R} \mathfrak{R}') &\cong H^1(\Gamma_K,(\bdrk{k} \hat{\otimes}_{K,\sigma} (\mathfrak{M}\otimes_{\mathfrak{R}} \mathfrak{R}'))^{H_K}).\end{align*}
From the inflation-restriction exact sequence with respect to the subgroup $H_K$ of $G_K$ and Claim \ref{claim:hkvanish}, we obtain the isomorphisms
\begin{align*} \coker(\psi) &\cong H^1(G_K,\bdrk{k} \hat{\otimes}_{K,\sigma} \mathfrak{M})\\ \textrm{and}\quad \coker(\psi \otimes_\mathfrak{R} \mathfrak{R}') &\cong H^1(G_K,\bdrk{k} \hat{\otimes}_{K,\sigma} (\mathfrak{M}\otimes_{\mathfrak{R}} \mathfrak{R}')).\end{align*}
We now obtain (\ref{eqn:dr1cocycle}) from the natural isomorphism $\coker(\psi) \otimes_\mathfrak{R} \mathfrak{R}' \iso \coker(\psi \otimes_\mathfrak{R} \mathfrak{R}')$ of Lemma \ref{lem:cokerspec}.
\end{proof}

We obtain de Rham periods at all geometric specializations.

\begin{thm} \label{thm:dr1cocycle}

Maintain the notation and hypotheses of Theorem \ref{thm:drcase} and assume that either $k=1$ or $\mathfrak{R}$ is $E$-affinoid.  If $\xi:\mathfrak{R} \rightarrow E'$ is a specialization to a finite field extension of $E$, then we have $\dim_{E'} H^r(G_K,\bdrk{k} \hat{\otimes}_{K,\sigma} (\mathfrak{M}\otimes_\mathfrak{R} E')) \ge d_k$ for $r \in \set{0,1}$.

\end{thm}

\begin{proof}

Let $\mathfrak{m} = \ker\xi$ and $\mathfrak{p}\subseteq \mathfrak{m}$ be a minimal prime.  By Proposition \ref{prop:dr1cocycle} (or Theorem \ref{thm:htcase}.(b) if $k=1$), we have identifications
\begin{align} \label{eqn:minprime} H^1(G_K,\bdrk{k} \hat{\otimes}_{K,\sigma} \mathfrak{M})\otimes_{\mathfrak{R}} \mathfrak{R}/\mathfrak{p} &\iso H^1(G_K,\bdrk{k} \hat{\otimes}_{K,\sigma} \mathfrak{M}/\mathfrak{p}\mathfrak{M})\\
\label{eqn:minprimespecial} \textrm{and}\quad H^1(G_K,\bdrk{k} \hat{\otimes}_{K,\sigma} \mathfrak{M}/\mathfrak{p}\mathfrak{M})\otimes_{\mathfrak{R}/\mathfrak{p}} E' &\iso H^1(G_K,\bdrk{k} \hat{\otimes}_{K,\sigma} (\mathfrak{M}\otimes_\mathfrak{R} E')).\end{align}

Picking $f$ in every minimal prime of $\mathfrak{R}$ except $\mathfrak{p}$, we obtain injections $\mathfrak{R}/\mathfrak{p} \inj \mathfrak{R}_f \inj \prod_{\set{i: f \notin \mathfrak{m}_i}} \mathfrak{R}_i$ by condition (iii) of Theorem \ref{thm:drcase}.  By condition (ii) of Theorem \ref{thm:drcase}, the image of $Q_{\sigma,k}$ in $\mathfrak{R}/\mathfrak{p}$ is nonzero, so $H^1(G_K,\bdrk{k} \hat{\otimes}_{K,\sigma} \mathfrak{M}/\mathfrak{p}\mathfrak{M})_{Q_{\sigma,k}}$ is finite flat of rank $d_k$ by Theorem \ref{thm:drcase}.(a), (\ref{eqn:bclocal}), and (\ref{eqn:minprime}).  Let $L =\Frac(\mathfrak{R}/\mathfrak{p}) = \Frac((\mathfrak{R}/\mathfrak{p})_{\mathfrak{m}/\mathfrak{p}})$.  It follows that
\[\dim_L H^1(G_K,\bdrk{k} \hat{\otimes}_{K,\sigma} \mathfrak{M}/\mathfrak{p}\mathfrak{M})_{\mathfrak{m}/\mathfrak{p}} \otimes_{(\mathfrak{R}/\mathfrak{p})_{\mathfrak{m}/\mathfrak{p}}} L = d_k.\]
By Nakayama's lemma, $H^1(G_K,\bdrk{k} \hat{\otimes}_{K,\sigma} \mathfrak{M}/\mathfrak{p}\mathfrak{M})_{\mathfrak{m}/\mathfrak{p}}$ can be generated by
\[\dim_{E'} H^1(G_K,\bdrk{k} \hat{\otimes}_{K,\sigma} \mathfrak{M}/\mathfrak{p}\mathfrak{M})\otimes_{\mathfrak{R}/\mathfrak{p}} E'\]
elements, so this dimension must be at least $d_k$, giving the desired bound for $r=1$ by (\ref{eqn:minprimespecial}).  Lemma \ref{lem:dimh0h1} implies the $r=0$ case.
\end{proof}

\subsection{Unbounded periods} \label{subsec:unbounded}

To pass from bounded to unbounded de Rham periods, we will use exact sequences exchanging projective limits with cohomology.  Proposition \ref{prop:profcoh} has the appearance of \cite[Theorem 2.7.5]{nsw}, but their result makes a profiniteness assumption on the module.  We instead consider a module $M$ that is the projective limit of topological modules $M/M_i$ such that the transition maps $M/M_i \rightarrow M/M_{i-1}$ have continuous topological sections.  The category of topological abelian groups is not abelian, so we cannot use the spectral sequence for a composition of functors as they do.  We instead work with the groups of cochains, cocycles, and coboundaries ``by hand'' to obtain a similar result.

\begin{prop} \label{prop:profcoh}

Let $G$ be a profinite group.  Let $M$ be a continuous $G$-module, and suppose that there exist $G$-submodules $M_i \subseteq M$ such that $M \cong \varprojlim_i M/M_i$ as topological $G$-modules, and such that there exist continuous sections $M/M_{i-1} \rightarrow M/M_i$.  Then for $n \ge 1$ there exists an exact sequence
\begin{equation}\label{eqn:h2gammaexact}\ses{{\varprojlim_i}^1 H^{n-1}(G,M/M_i)}{H^n(G,M)}{\varprojlim_i H^n(G,M/M_i)}.\end{equation}
\end{prop}

\begin{remark}
Note that $\bdr^+ = \varprojlim_k \bdrk{k}$, where the left-hand side has the inverse limit topology \cite[Remark 1.1]{iz2}.  As mentioned in Remark \ref{remark:hkvanish}, sections to the transition maps exist.  We will primarily use Proposition \ref{prop:profcoh} to pass from modules over $\bdrk{k}$ to modules over $\bdr^+$.
\end{remark}

\begin{proof}[Proof of Proposition \ref{prop:profcoh}]

We first note that there exists an exact sequence
\begin{equation} \label{eqn:zcbexact} 0 \rightarrow (Z^{n-1}(G,M/M_i))_i \rightarrow (C^{n-1}(G,M/M_i))_i \stackrel{d}{\rightarrow} (B^n(G,M/M_i))_i \rightarrow 0\end{equation}
of projective systems of abelian groups.  Here $Z^{n-1},C^{n-1},$ and $B^n$ denote the usual groups of continuous cocycles, cochains, and coboundaries.  By the hypothesis that continuous sections exist, the transition maps of the middle term and therefore (by commutativity) the last term are surjective, so they satisfy the Mittag-Leffler condition.  Moreover, it follows from the definition of the cochains and cocycles that
\begin{equation} \label{eqn:limchaincycle}\varprojlim_i C^n(G,M/M_i) = C^n(G,M) \textrm{ and }\varprojlim_i Z^n(G,M/M_i) = Z^n(G,M).\end{equation}
Passing to the limit in (\ref{eqn:zcbexact}) and using (\ref{eqn:limchaincycle}) to simplify the expressions, we obtain
\[0 \rightarrow Z^{n-1}(G,M) \rightarrow C^{n-1}(G,M) \stackrel{d}{\rightarrow} \varprojlim_i B^n(G,M/M_i) \rightarrow {\varprojlim_i}^1 Z^{n-1}(G,M/M_i)\rightarrow 0,\]
where the 0 on the right is by the Mittag-Leffler property of $(C^{n-1}(G,M/M_i))_i$.  If we define $(C/Z)^{n-1}(G,M) = C^{n-1}(G,M)/Z^{n-1}(G,M)$, we deduce an isomorphism
\begin{equation} \label{eqn:profcohiso} (\varprojlim_i B^n(G,M/M_i))/d((C/Z)^{n-1}(G,M)) \cong {\varprojlim_i}^1 Z^{n-1}(G,M/M_i).\end{equation}

By (\ref{eqn:limchaincycle}), the limit of the exact sequence
\begin{equation} \label{eqn:bzhexact} 0 \rightarrow (B^n(G,M/M_i))_i \rightarrow (Z^n(G,M/M_i))_i \rightarrow (H^n(G,M/M_i))_i \rightarrow 0\end{equation}
of projective systems of abelian groups is
\[0 \rightarrow \varprojlim_i B^n(G,M/M_i) \rightarrow Z^n(G,M) \rightarrow  \varprojlim_i H^n(G,M/M_i) \rightarrow 0,\]
where we have used the surjectivity of the transition maps of $(B^n(G,M/M_i))_i$ to see that the last term vanishes.  We obtain an exact sequence
\begin{align*}
  0 \rightarrow (\varprojlim_i B^n(G,M/M_i))/d((C/Z)^{n-1}(G,M)) &\rightarrow Z^n(G,M)/d((C/Z)^{n-1}(G,M))\\
  &\rightarrow  \varprojlim_i H^n(G,M/M_i) \rightarrow 0,
\end{align*}
which can be rewritten by (\ref{eqn:profcohiso}) and the definition of cohomology as
\[0 \rightarrow {\varprojlim_i}^1 Z^{n-1}(G,M/M_i) \rightarrow H^n(G,M) \rightarrow  \varprojlim_i H^n(G,M/M_i) \rightarrow 0.\]

By  \cite[Proposition 2.7.4]{nsw}, the (not very) long exact sequence of (\ref{eqn:bzhexact}) for $n-1$ in place of $n$ gives an isomorphism $\varprojlim_i^1 Z^{n-1}(G,M/M_i) \cong \varprojlim_i^1 H^{n-1}(G,M/M_i),$ which yields (\ref{eqn:h2gammaexact}) above.
\end{proof}

\begin{remark} \label{remark:topology}
The definition of the ring $\bdr^+$ and its topology can be found in \cite[\S 1.5.3]{fontaine}.  For a very explicit definition more adapted to the functional analytic discussion of this section, we refer the reader to \cite[Remark 3.2]{iz} and the definition of $w_n$ preceding that remark, which precisely exhibits $\bdr^+$ as a Fr\'echet space by giving a family of norms whose kernels are the $t^n\bdr^+$.

In this paper, we equip $\bdr$ with the locally convex final topology as the strict inductive limit of its $K$-locally convex subspaces $t^k\bdr^+$ \cite[\S I.5.E2]{schneider}.  By definition, the topology on $\bdr$ is induced by the seminorms $\bdr \rightarrow \mathbf{R}$ with continuous restriction to each $t^k\bdr^+$.  (Also see the equivalent definitions in \cite[Theorem 11.1.1]{gs}.)  With respect to this topology, $\bdr$ is $K$-locally convex, complete (with respect to Cauchy nets), bornological, barrelled, and Hausdorff \cite[Proposition 5.5.(ii), \S 6, and Lemma 7.9]{schneider}.  Moreover, $t^k\bdr^+$ is closed in $\bdr$ and inherits its original topology from that of $\bdr$ \cite[Proposition 5.5.(i) and (iii)]{schneider}.  It follows from the definition that the $G_K$-action is continuous.  Although the $t^k\bdr^+$ are Fr\'echet spaces, this topology on $\bdr$ is likely not metrizable; see the discussion in \cite[\S 11.2]{gs}.

We always use $\mathfrak{A} \otimes_K \mathfrak{B}$ to denote the \emph{inductive} tensor product of two locally convex $K$-vector spaces, which is induced by those seminorms $p:\mathfrak{A} \otimes_K \mathfrak{B} \rightarrow\mathbf{R}$ such that the induced map $\mathfrak{A} \times \mathfrak{B} \rightarrow \mathbf{R}$ is separately continuous, i.e. continuous when restricted to $\set{a} \times \mathfrak{B}$ or $\mathfrak{A} \times \set{b}$ for any $a \in \mathfrak{A}$ or $b \in \mathfrak{B}$.  It follows from the definition of this topology and the universal property of completion that if $\mathfrak{A}$ and $\mathfrak{B}$ are equipped with continuous actions of $G_K$, then $\mathfrak{A} \otimes_K \mathfrak{B}$ and $\mathfrak{A} \hat{\otimes}_K \mathfrak{B}$ are as well \cite[Proposition 7.5]{schneider}.

Note that since $t^k\bdr^+$ is a Fr\'echet space, by \cite[Proposition 17.6]{schneider} the inductive and projective topologies coincide on $t^k\bdr^+ \otimes_K \mathfrak{M}$ for any $K$-Banach (or even Fr\'echet) space $\mathfrak{M}$.

\end{remark}

Completed tensor products are not needed in the Artinian case; the following result is clear.

\begin{lem} \label{lem:artcomplete}

Maintain the notation of Section \ref{subsec:senkisin}, and suppose that $\mathfrak{R}$ is Artinian of finite dimension over $E$.  For $B \in \set{\bdr,t^k\bdr^+,t^k\bdrk{\ell}}$, where $k \le \ell \in \mathbf{Z}$, the natural map $B \otimes_{K,\sigma} \mathfrak{M}\rightarrow B \hat{\otimes}_{K,\sigma} \mathfrak{M}$ is an isomorphism.

\end{lem}

We will also need the following fact.

\begin{lem} \label{lem:tensorprodlim}

Let $K$ be a finite extension of $\mathbf{Q}_p$. Suppose that $\mathfrak{A}$ is a strict inductive limit of locally convex $K$-vector spaces $\mathfrak{A}_i$ for $i \in \mathbf{Z}_{\ge 0}$ and $\mathfrak{B}$ is any locally convex $K$-vector space.  Then the natural map $\varinjlim_i (\mathfrak{A}_i \otimes_K \mathfrak{B}) \rightarrow \mathfrak{A} \otimes_K \mathfrak{B}$, where the left-hand side has the locally convex final topology, is an isomorphism of locally convex $K$-vector spaces.

\end{lem}

\begin{proof}

Note that the map is an isomorphism of $K$-vector spaces, so we need only check that it is a homeomorphism.

The topology on $\varinjlim_i (\mathfrak{A}_i \otimes_K \mathfrak{B})$ is defined by all seminorms $p$ such that the composite map $\mathfrak{A}_i \otimes_K \mathfrak{B} \rightarrow \mathfrak{A} \otimes_K \mathfrak{B} \rightarrow \mathbf{R}$ is continuous for all $i$.  By definition of the inductive topology, this holds if and only if for each $i \in \mathbf{Z}_{\ge 0}$, $a \in \mathfrak{A}_i$, and $b \in \mathfrak{B}$, the composite maps $\set{a} \times \mathfrak{B} \rightarrow \mathfrak{A}_i \otimes_K \mathfrak{B} \rightarrow \mathbf{R}$ and $\mathfrak{A}_i \times \set{b} \rightarrow \mathfrak{A}_i \otimes_K \mathfrak{B} \rightarrow \mathbf{R}$ are continuous.

The topology of $\mathfrak{A} \otimes_K \mathfrak{B}$ is defined by those seminorms $p$ such that for any $a \in \mathfrak{A}$ and $b \in \mathfrak{B}$, the induced maps $\set{a} \times \mathfrak{B} \rightarrow \mathbf{R}$ and $\mathfrak{A} \times \set{b} \rightarrow \mathbf{R}$ are continuous.  By definition, the latter is equivalent to saying that for each $b \in \mathfrak{B}$ and $i \in \mathbf{Z}_{\ge 0}$, the induced map $\mathfrak{A}_i \times \set{b} \rightarrow \mathbf{R}$ is continuous, which matches the second condition on $p$ for $\varinjlim_i (\mathfrak{A}_i \otimes_K \mathfrak{B})$.  For every $i$ and $a' \in \mathfrak{A}_i$ mapping to $a$, there are maps $\set{a'} \times \mathfrak{B} \rightarrow \set{a} \times \mathfrak{B} \stackrel{p}{\rightarrow} \mathbf{R}$, where the first map is a homeomorphism.  In particular, the composite map is continuous if and only if the second map is continuous.

Since $p$ needs to satisfy the same conditions for continuity in either case, we deduce that the topologies are the same.
\end{proof}

As a consequence, we can identify the topology on $\bdr \hat{\otimes}_{K,\sigma}\mathfrak{M}$ with a strict inductive limit topology.

\begin{lem} \label{eqn:bdrtensortop}

Maintain the notation of Section \ref{subsec:senkisin}.  There is a natural isomorphism of locally convex $K$-vector spaces $\varinjlim_k (t^k\bdr^+ \hat{\otimes}_{K,\sigma} \mathfrak{M}) \iso \bdr \hat{\otimes}_{K,\sigma} \mathfrak{M}$.  Moreover, the limit is strict with closed subspaces $t^k\bdr^+ \hat{\otimes}_{K,\sigma} \mathfrak{M}$.

\end{lem}

\begin{proof}

The inclusion $t^k\bdr^+ \otimes_{K,\sigma} \mathfrak{M} \rightarrow t^{k-1}\bdr^+ \otimes_{K,\sigma} \mathfrak{M}$ is strict by \cite[Corollary 17.5.(ii) and Proposition 17.6]{schneider}.  We obtain strict injections $t^{k-1}\bdr^+ \otimes_{K,\sigma} \mathfrak{M} \rightarrow t^{k-1}\bdr^+ \hat{\otimes}_{K,\sigma} \mathfrak{M}$ and $t^{k-1}\bdr^+ \hat{\otimes}_{K,\sigma} \mathfrak{M} \rightarrow t^k\bdr^+ \hat{\otimes}_{K,\sigma} \mathfrak{M}$ from two applications of \cite[Proposition 7.5]{schneider}.  The image of the latter map is closed \cite[Remark 7.1.(v)]{schneider}.  The second claim now follows from \cite[Proposition 5.5.(iii)]{schneider}.

By Lemma \ref{lem:tensorprodlim}, there is a natural isomorphism $(\varinjlim_k (t^k\bdr^+ \otimes_{K,\sigma} \mathfrak{M}))^\wedge \iso \bdr \hat{\otimes}_{K,\sigma}  \mathfrak{M}$ of locally convex $K$-vector spaces, where $\wedge$ denotes completion.  By the universal property, we have a continuous map $\varinjlim_k (t^k\bdr^+ \otimes_{K,\sigma} \mathfrak{M}) \rightarrow \varinjlim_k (t^k\bdr^+ \hat{\otimes}_{K,\sigma} \mathfrak{M})$.  Since the range is complete and Hausdorff by \cite[Proposition 5.5.(ii) and Lemma 7.9]{schneider}, this uniquely extends to a continuous map $\psi: (\varinjlim_k (t^k\bdr^+ \otimes_{K,\sigma} \mathfrak{M}))^\wedge \rightarrow \varinjlim_k (t^k\bdr^+ \hat{\otimes}_{K,\sigma} \mathfrak{M})$ by \cite[Lemma 7.3]{schneider}.  For each $k$, we have a map $t^k\bdr^+ \otimes_{K,\sigma} \mathfrak{M} \rightarrow (\varinjlim_k (t^k\bdr^+ \otimes_{K,\sigma} \mathfrak{M}))^\wedge$, where the image is Hausdorff and complete, so this extends to a map $t^k\bdr^+ \hat{\otimes}_{K,\sigma} \mathfrak{M} \rightarrow (\varinjlim_k (t^k\bdr^+ \otimes_{K,\sigma} \mathfrak{M}))^\wedge$.  Using the universal property again, we obtain a continuous inverse to $\psi$.  The needed isomorphism follows.
\end{proof}

To calculate higher cohomology of $\bdr$, we will use the following observation.

\begin{lem} \label{lem:indlim}

Maintain the notation of Section \ref{subsec:senkisin}.  For $r \ge 0$, the natural maps $t^k\bdr^+ \hat{\otimes}_{K,\sigma} \mathfrak{M} \rightarrow \bdr \hat{\otimes}_{K,\sigma} \mathfrak{M}$ induce an isomorphism
\[\varinjlim_k H^r(G_K,t^k\bdr^+ \hat{\otimes}_{K,\sigma} \mathfrak{M}) \rightarrow H^r(G_K,\bdr \hat{\otimes}_{K,\sigma} \mathfrak{M}).\]

\end{lem}

\begin{proof}

We claim that every continuous map $\psi: G_K^r \rightarrow \bdr \hat{\otimes}_{K,\sigma} \mathfrak{M}$ has image contained in $t^k\bdr^+$ for some $k$.  Both injectivity and surjectivity will follow from this.

A bounded subset of a locally convex $K$-vector space $V$ is one such that for any open lattice $L\subseteq V$, there is $a\in K$ such that $B \subseteq aL$.  Any finite set is bounded \cite[\S I.4]{schneider}.  It follows that for any open lattice $L$, each point of $\psi(G_K^r)$ is contained in some $aL$.  Since $\psi(G_K^r)$ is compact, it is contained in a single $aL$, so $\psi(G_K^r)$ is bounded.

By Lemma \ref{eqn:bdrtensortop}, the completed tensor product $\bdr \hat{\otimes}_{K,\sigma} \mathfrak{M}$ is the strict inductive limit of its closed subspaces $t^k\bdr^+ \hat{\otimes}_{K,\sigma} \mathfrak{M}$.  The result now follows from \cite[Proposition 5.6]{schneider}, which shows that any bounded subset of $\bdr \hat{\otimes}_{K,\sigma} \mathfrak{M}$ is contained in some $t^k\bdr^+ \hat{\otimes}_{K,\sigma} \mathfrak{M}$.
\end{proof}

We can deduce the following result.

\begin{prop}\label{prop:bdrdimh0h1}

Maintain the notation of Section \ref{subsec:senkisin}. If $\mathfrak{R}$ is an Artinian $E$-Banach algebra of finite dimension over $E$, then for $\ell < k \in \mathbf{Z}$ and $B \in \set{\bdr,t^k\bdr^+,t^k\bdrk{\ell}}$ we have
\begin{equation}\label{eqn:beq} \dim_E H^0(G_K,B \otimes_{K,\sigma} \mathfrak{M}) = \dim_E H^1(G_K,B \otimes_{K,\sigma} \mathfrak{M}).\end{equation}
Writing $d(B)$ for this common dimension, we have inequalities
\begin{equation}\label{eqn:dimeq} d(\bdr) \ge d(t^k\bdr^+) \ge d(t^k\bdrk{\ell}).\end{equation}
The first is an equality if $k$ is sufficiently small, as is the second if $\ell$ is sufficiently large.

\end{prop}

\begin{proof}

Using Lemma \ref{lem:artcomplete}, we may work with ordinary tensor products.  Let $\kappa(\mathfrak{R})$ denote the residue field.  The image of $P_\sigma(T)$ in $\kappa(\mathfrak{R})[T]$ has finitely many roots.  For $j$ sufficiently large or small, we deduce from Proposition \ref{prop:htgap} that for $r \in \set{0,1}$, we have
\begin{equation}\label{eqn:cpkvanish} \dim_E H^r(G_K,\mathbf{C}_p(j) \otimes_{K,\sigma} \mathfrak{M}) = 0.\end{equation}

For $k < \ell$, we have exact sequences
\begin{gather}\label{eqn:bdrexact1} \ses{\mathbf{C}_p(\ell)}{t^k\bdrk{\ell+1}}{t^k\bdrk{\ell}},\\
\label{eqn:bdrexact3} 0\rightarrow t^k \bdrk{\ell} \rightarrow t^{k-1} \bdrk{\ell} \rightarrow \mathbf{C}_p(k-1) \rightarrow 0\\
\label{eqn:bdrexact2} \textrm{and} \quad \ses{t^k\bdr^+}{t^{k-1}\bdr^+}{\mathbf{C}_p(k-1)}.\end{gather}
For $k$ sufficiently small in (\ref{eqn:lsurj1}) and (\ref{eqn:lsurj0}) and $\ell$ sufficiently large in (\ref{eqn:ksurj0}), the natural maps
\begin{align}\label{eqn:ksurj0} H^r(G_K,t^k\bdrk{\ell+1} \otimes_{K,\sigma} \mathfrak{M}) &\rightarrow H^r(G_K,t^k\bdrk{\ell} \otimes_{K,\sigma} \mathfrak{M}),\\
\label{eqn:lsurj1} H^r(G_K,t^k\bdrk{\ell} \otimes_{K,\sigma} \mathfrak{M}) &\rightarrow H^r(G_K,t^{k-1}\bdrk{\ell} \otimes_{K,\sigma} \mathfrak{M}),\\
\label{eqn:lsurj0} \textrm{and}\quad H^r(G_K,t^k\bdr^+ \otimes_{K,\sigma} \mathfrak{M}) &\rightarrow H^r(G_K,t^{k-1}\bdr^+ \otimes_{K,\sigma} \mathfrak{M})
\end{align}
are isomorphisms for $r \in \set{0,1}$ by (\ref{eqn:cpkvanish}) and, for (\ref{eqn:ksurj0}) and $r=1$, the existence of a continuous section from Remark \ref{remark:hkvanish}.

Now suppose $k < \ell \in \mathbf{Z}$ are given.  By Lemma \ref{lem:dimh0h1},
\[\dim_E H^0(G_K,\bdrk{\ell} \otimes_{K,\sigma} \mathfrak{M})=\dim_E H^1(G_K,\bdrk{\ell} \otimes_{K,\sigma} \mathfrak{M}).\]
We induct using the 6-term exact sequence similar to (\ref{eqn:exactsnake}) obtained from (\ref{eqn:bdrexact3}) to find
\[\dim_E H^0(G_K,t^k\bdrk{\ell} \otimes_{K,\sigma} \mathfrak{M})=\dim_E H^1(G_K,t^k\bdrk{\ell} \otimes_{K,\sigma} \mathfrak{M}).\]
This is (\ref{eqn:beq}) for $B = t^k\bdrk{\ell}$.

Fix $\ell$ large enough so that (\ref{eqn:cpkvanish}) holds for all $j \ge \ell$.  Also note that by right-exactness of the tensor product, $t^k\bdrk{i} \otimes_{K,\sigma} \mathfrak{M} \cong (t^k\bdr^+ \otimes_{K,\sigma} \mathfrak{M})/(t^i\bdr^+ \otimes_{K,\sigma} \mathfrak{M})$.  Using the isomorphisms (\ref{eqn:ksurj0}) with both $r=0$ and $r=1$, we see that in the exact sequence (\ref{eqn:h2gammaexact}) with $n=1$, $M = t^k\bdr^+ \otimes_{K,\sigma} \mathfrak{M}$, and $M/M_i = t^k\bdrk{i} \otimes_{K,\sigma} \mathfrak{M}$, the left-hand term vanishes and the right-hand term has dimension equal to $\dim_E H^1(G_K,t^k\bdrk{\ell} \otimes_{K,\sigma} \mathfrak{M})$.  Using the isomorphisms (\ref{eqn:ksurj0}) with $r=0$ and the $B = t^k\bdrk{\ell}$ case of (\ref{eqn:beq}), we deduce the equality (\ref{eqn:beq}) for $B = t^k\bdr^+$.

Choose $k$ so that (\ref{eqn:cpkvanish}) holds for all $j \le k$.  Beginning with the $B=t^k\bdr^+$ case of (\ref{eqn:beq}), we form the inductive limit over $k$, and use that (\ref{eqn:lsurj0}) is an isomorphism for $r \in \set{0,1}$ and $j \le k$ together with Lemmas \ref{lem:artcomplete} and \ref{lem:indlim} to deduce (\ref{eqn:beq}) for $B=\bdr$.

We obtain the first inequality in (\ref{eqn:dimeq}) for 0-cocycles by applying left-exactness of $G_K$-invariants to the injection $t^k\bdr^+ \otimes_{K,\sigma} \mathfrak{M} \inj \bdr \otimes_{K,\sigma} \mathfrak{M}$.  For the second, we observe that
\[H^1(G_K,t^k\bdrk{\ell+1} \otimes_{K,\sigma} \mathfrak{M}) \rightarrow H^1(G_K,t^k\bdrk{\ell} \otimes_{K,\sigma} \mathfrak{M})\]
is always surjective by the sequence similar to (\ref{eqn:exactsnake}) obtained from (\ref{eqn:bdrexact1}), so we have
\[d(t^k\bdrk{\ell+1}) \ge d(t^k\bdrk{\ell}).\]
If we apply this inequality $\ell'-\ell$ times, where $\ell'$ is large enough for (\ref{eqn:cpkvanish}) to always hold, we obtain $d(t^k\bdrk{\ell'}) \ge d(t^k\bdrk{\ell})$.  It now follows from the proof of (\ref{eqn:beq}) for $B=t^k\bdr^+$ that
\[d(t^k\bdr^+)= \dim_E H^1(G_K,t^k\bdr^+ \otimes_{K,\sigma} \mathfrak{M}) = d(t^k\bdrk{\ell'}) \ge d(t^k\bdrk{\ell}).\]
\end{proof}

We use Proposition \ref{prop:profcoh} with Theorem \ref{thm:drcase} to deduce the existence of de Rham periods for certain base changes of a family.  We remark that in the Berger-Colmez case (i.e. when $Q_\sigma(T)$ is a unit), we find that $H^r(G_K,\bdr \hat{\otimes}_{K,\sigma} \mathfrak{M})$ is finite flat of rank $n$ for $r \in \set{0,1}$.

\begin{prop} \label{prop:nicespecial}

Maintain the hypotheses of Theorem \ref{thm:drcase}.  Suppose that $\mathfrak{R}'$ is a Noetherian Banach $\mathfrak{R}$-algebra such that the image of $Q_{\sigma,k}$ is a unit in $\mathfrak{R}'$ for all $k\ge 0$.  Then for $r \in \set{0,1}$, the $\mathfrak{R}'$-module $H^r(G_K,\bdr^+ \hat{\otimes}_{K,\sigma} (\mathfrak{M} \otimes_{\mathfrak{R}} \mathfrak{R}'))$ is finite flat of rank $n$ and $H^0(G_K,\bdr \hat{\otimes}_{K,\sigma} (\mathfrak{M} \otimes_{\mathfrak{R}} \mathfrak{R}'))$ has a finite flat $\mathfrak{R}'$-submodule of rank $n$.

If, in addition, $Q_\sigma(j)$ is a unit in $\mathfrak{R}'$ for all $j > 0$, then $H^r(G_K,\bdr \hat{\otimes}_{K,\sigma} (\mathfrak{M} \otimes_{\mathfrak{R}} \mathfrak{R}'))$ are finite flat of rank $n$ for $r \in\set{0,1}$.

\end{prop}

\begin{proof}

By \cite[Proposition 1.1.29]{emerton}, we have a natural isomorphism
\begin{equation} \label{eqn:switchprojlim} \bdr^+\hat{\otimes}_{K,\sigma} (\mathfrak{M} \otimes_{\mathfrak{R}} \mathfrak{R}') \cong \varprojlim_k (\bdrk{k} \hat{\otimes}_{K,\sigma} (\mathfrak{M} \otimes_{\mathfrak{R}} \mathfrak{R}')).\end{equation}
It follows that $H^0(G_K,\bdr^+\hat{\otimes}_{K,\sigma} (\mathfrak{M} \otimes_{\mathfrak{R}} \mathfrak{R}')) = \varprojlim_k H^0(G_K,\bdrk{k} \hat{\otimes}_{K,\sigma} (\mathfrak{M} \otimes_{\mathfrak{R}} \mathfrak{R}'))$. The discussion of the exact sequences (\ref{eqn:bcexactsnake1}) and (\ref{eqn:bcexactsnake2}) in the proof of Theorem \ref{thm:drcase} implies that the transition maps on the right-hand side are surjective maps of finite flat modules.  For $k$ sufficiently large, combining Theorem \ref{thm:drcase}.(a) and (b),
\[\rank_{\mathfrak{R}'} H^0(G_K,\bdrk{k} \hat{\otimes}_{K,\sigma} (\mathfrak{M} \otimes_{\mathfrak{R}} \mathfrak{R}')) = d_k = n\]
is independent of $k$, so the transition maps are isomorphisms and the limit is finite flat of rank $n$ as well.

From Proposition \ref{prop:profcoh} and (\ref{eqn:switchprojlim}), we have an exact sequence
\begin{align*}
	0 &\rightarrow {\varprojlim_i}^1 H^0(G_K,\bdrk{k} \hat{\otimes}_{K,\sigma} (\mathfrak{M} \otimes_{\mathfrak{R}} \mathfrak{R}')) \rightarrow H^1(G_K,\bdr^+\hat{\otimes}_{K,\sigma} (\mathfrak{M} \otimes_{\mathfrak{R}} \mathfrak{R}'))\\
	&\rightarrow \varprojlim_i H^1(G_K,\bdrk{k} \hat{\otimes}_{K,\sigma} (\mathfrak{M} \otimes_{\mathfrak{R}} \mathfrak{R}')) \rightarrow 0.
\end{align*}
By the aforementioned discussion of (\ref{eqn:bcexactsnake1}) and (\ref{eqn:bcexactsnake2}), the transition maps of the first and third terms are surjective maps of finite flat $\mathfrak{R}'$-modules, so the first term vanishes.  As before, we have
\[\rank_{\mathfrak{R}'} H^1(G_K,\bdrk{k} \hat{\otimes}_{K,\sigma} (\mathfrak{M} \otimes_{\mathfrak{R}} \mathfrak{R}')) = d_k = n\]
for $k$ sufficiently large by Theorem \ref{thm:drcase}.(a) and (b), so the transition maps are eventually isomorphisms and the limit is finite flat of rank $n$ over $\mathfrak{R}'$.

We take $G_K$-invariants of the natural injection $\bdr^+\hat{\otimes}_{K,\sigma} (\mathfrak{M} \otimes_{\mathfrak{R}} \mathfrak{R}') \hookrightarrow \bdr\hat{\otimes}_{K,\sigma} (\mathfrak{M} \otimes_{\mathfrak{R}} \mathfrak{R}')$ to deduce that $H^0(G_K,\bdr \hat{\otimes}_{K,\sigma} (\mathfrak{M} \otimes_{\mathfrak{R}} \mathfrak{R}'))$ has a finite flat rank $n$ submodule.

If $Q_\sigma(j)$ is a unit in $\mathfrak{R}'$ for all $j>0$, then the isomorphisms (\ref{eqn:lsurj0}) hold for all $k \le 0$.  We deduce the final claim from Lemma \ref{lem:indlim}.
\end{proof}

If we no longer ask for flatness, we can relax the strong hypothesis on $Q_{\sigma,k}$.  Observe that the bound $d_k\dim_E \mathfrak{R}'$ below is compatible with the expectation of $d_k$ interpolated periods.

\begin{prop} \label{prop:anyspecial}

Maintain the hypotheses of Theorem \ref{thm:drcase}.  Suppose that we are given a local Artinian Banach $\mathfrak{R}$-algebra $\mathfrak{R}'$ of finite $E$-dimension with structure morphism $\xi: \mathfrak{R} \rightarrow \mathfrak{R}'$.  Fix $k \ge 0$, and suppose that $\xi(Q_{\sigma,k})$ is a unit.  Then for $r \in \set{0,1}$ we have
\[ \dim_E H^r(G_K,\bdr \otimes_{K,\sigma} (\mathfrak{M} \otimes_{\mathfrak{R}} \mathfrak{R}')) \ge \dim_E H^r(G_K,\bdr^+ \otimes_{K,\sigma} (\mathfrak{M} \otimes_{\mathfrak{R}} \mathfrak{R}')) \ge d_k\dim_E \mathfrak{R}'.\]

\end{prop}

Using Theorem \ref{thm:dr1cocycle}, we can also prove a result without the hypothesis that $Q_{\sigma,k}$ vanishes, but we need to assume that $\mathfrak{R}$ is $E$-affinoid.

\begin{prop} \label{prop:drspecial}

Maintain the hypotheses of Theorem \ref{thm:drcase}, assume that $\mathfrak{R}$ is $E$-affinoid, and let $\mathfrak{R} \rightarrow E'$ be a map to a finite extension $E'$ of $E$.  Then for $r \in \set{0,1}$ we have
\[\dim_{E'} H^r(G_K,\bdr \otimes_{K,\sigma} (\mathfrak{M} \otimes_{\mathfrak{R}} E')) \ge \dim_{E'} H^r(G_K,\bdr^+ \otimes_{K,\sigma} (\mathfrak{M} \otimes_{\mathfrak{R}} E')) \ge d_k.\]

\end{prop}

\begin{proof}[Proof of Propositions \ref{prop:anyspecial} and \ref{prop:drspecial}]
Proposition \ref{prop:anyspecial} follows from Theorem \ref{thm:drcase} and Proposition \ref{prop:bdrdimh0h1}.  For Proposition \ref{prop:drspecial}, we apply Theorem \ref{thm:dr1cocycle} to the specialization $\mathfrak{R} \rightarrow E'$ and use Proposition \ref{prop:bdrdimh0h1}.
\end{proof}

\subsection{Essentially self-dual and decomposable specializations} \label{subsec:esd}

We study specializations $V$ that are essentially self-dual in the sense that $V \cong V^\vee(s)$ for some $s$. Examples of such representations include polarizable representations with trivial complex conjugation.  In the Hodge-Tate case, we will be able to recover $2m$ periods of weight $w$ of $V$ from just $m$ fixed periods of weight $w$ in a family, except for when the weight $w$ satisfies $w = s-w$.

\begin{thm} \label{thm:vcp}
Maintain the notation and hypotheses of Theorem \ref{thm:htcase}.  Assume that a morphism of $E$-Banach algebras $\xi:\mathfrak{R} \rightarrow E'$ is given, where $E'$ is a finite extension of $E$.  Let $V=\mathfrak{M} \otimes_\mathfrak{R} E'$, and suppose that there is an isomorphism $V \cong V^\vee(s)$ for some $s \in \mathbf{Z}$.  Then for $r\in\set{0,1}$,
\begin{equation} \label{eqn:vcp} \dim_{E'}H^r(G_K,\mathbf{C}_p\otimes_{K,\sigma} V)+\dim_{E'}H^r(G_K,\mathbf{C}_p(-s)\otimes_{K,\sigma} V) \ge 2m.\end{equation}
\end{thm}

\begin{proof}
The module $\mathfrak{M} \oplus \mathfrak{M}^\vee$ satisfies the same hypotheses as $\mathfrak{M}$.  We apply Theorem \ref{thm:dr1cocycle} (with $k=1$) to this module and the map $\xi$ to derive the conclusion.
\end{proof}

To study the de Rham case, we will impose stronger hypotheses.  We will be interested in specializations that decompose as an extension of a representation $V$ by its twisted dual $V^\vee(s)$ or vice-versa, for $s \ge 1$.  We will also need to know the determinant of $V$.

\begin{thm} \label{thm:vhtdr}

Maintain the hypotheses of Theorem \ref{thm:drcase}.  Assume that $\rank_\mathfrak{R} \mathfrak{M} = 2d_k$.  Let $\xi:\mathfrak{R} \rightarrow E'$ be a map of $E$-Banach algebras, with $E'$ is a finite extension of $E$.  Suppose that there is a short exact sequence
\begin{align}\label{eqn:vexact1} &\ses{V}{\mathfrak{M} \otimes_\mathfrak{R} E'}{V^\vee(s)}\\
\label{eqn:vexact2}\textrm{or}\quad &\ses{V^\vee(s)}{\mathfrak{M} \otimes_\mathfrak{R} E'}{V},\end{align}
where $s\in \mathbf{Z}_{\ge 1}$ and $V$ is an $E'$-vector space equipped with a continuous action of $G_K$.  Assume in addition that $\det(\mathbf{C}_p \otimes_{K,\sigma} V) \cong \mathbf{C}_p\(-\sum_i w_{i,\sigma}\) \otimes_{K,\sigma} E',$ where $\mathbf{C}_p \otimes_{K,\sigma} V$ is viewed as a module over $\mathbf{C}_p \otimes_{K,\sigma} E'$.  For $r \in \set{0,1}$, we have
\begin{equation} \label{eqn:vdr}  \dim_{E'}H^r(G_K,\bdr^+\otimes_{K,\sigma} V)=d_k,\quad  \dim_{E'}H^r(G_K,\bdr\otimes_{K,\sigma} V)=d_k,\end{equation}
and, in the case of (\ref{eqn:vexact2}),
\begin{equation} \label{eqn:mdr} \dim_{E'}H^r(G_K,\bdr\otimes_{K,\sigma} (\mathfrak{M}\otimes_\mathfrak{R} E'))=2d_k.\end{equation}

\end{thm}

\begin{proof}

With respect to the fixed $\sigma: K \rightarrow E$, let $P_{\sigma,V}(T) \in E'[T]$ denote the $\sigma$-factor of the Sen polynomial of $V$.  We will use the subscript $E'$ to denote the specialization of an element of $\mathfrak{R}[T]$ to $E'$.  Then from (\ref{eqn:vexact1}) or (\ref{eqn:vexact2}), the $\sigma$-factor of the Sen polynomial of $\mathfrak{M} \otimes_\mathfrak{R} E'$ satisfies $P_{\sigma,E'}(T) = S_{\sigma,E'}(T)Q_{\sigma,E'}(T) = P_{\sigma,V}(T)P_{\sigma,V}(-s-T).$  Since the $w_{i,\sigma}$ are nonnegative, $w_{i,\sigma} \ne -s-w_{i',\sigma}$ for any $i,i'$.  The zeros of $P_{\sigma,V}(T)P_{\sigma,V}(-s-T)$ must be the multiset $\coprod_j\set{w_{i,\sigma},-s-w_{i,\sigma}},$ so we must have $Q_{\sigma,E'}(T) = S_{\sigma,E'}(-s-T)$.

The determinant condition ensures that $P_{\sigma,V}(T)=S_{\sigma,E'}(T)$ and $P_{\sigma,V}(-s-T)=Q_{\sigma,E'}(T)$.  For any vector space $V'$ over $E'$ equipped with a linear $G_K$-action,
\[\dim_{E'} H^0(G_K, \bdr^+ \otimes_{K,\sigma} V') \le \dim_{E'} H^0(G_K, \bht^+ \otimes_{K,\sigma} V').\]
From this and our identification of $P_{\sigma,V}(-s-T)$, $\dim_{E'} H^0(G_K, \bdr^+ \otimes_{K,\sigma} V^\vee(s)) = 0$.  By Proposition \ref{prop:nicespecial}, $\dim_{E'} H^0(G_K, \bdr^+ \otimes_{K,\sigma} (\mathfrak{M}\otimes_\mathfrak{R} E'))\ge d_k.$  Tensoring (\ref{eqn:vexact1}) or (\ref{eqn:vexact2}) with $\bdr^+$ and using the exact sequence in $G_K$-cohomology, if $V$ is a submodule, we are done.  If it is a quotient, we have an exact sequence
\[0 \rightarrow  H^0(G_K, \bdr^+ \otimes_{K,\sigma} (\mathfrak{M}\otimes_\mathfrak{R} E')) \rightarrow H^0(G_K, \bdr^+ \otimes_{K,\sigma} V) \rightarrow H^1(G_K, \bdr^+ \otimes_{K,\sigma} V^\vee(s)).\]
By Proposition \ref{prop:bdrdimh0h1}, the last entry in the exact sequence vanishes, giving $\dim_{E'}H^0(\bdr^+\otimes_{K,\sigma} V)= d_k$ and thus (\ref{eqn:vdr}).  For the case (\ref{eqn:vexact2}), we obtain (\ref{eqn:mdr}) from \cite[Example 6.3.6]{bcon} and Proposition \ref{prop:bdrdimh0h1}.
\end{proof}

\section{Stratification by de Rham data} \label{sec:stratfam}

Given a reduced rigid space $\mathfrak{X}$ with a $G_K$-representation on a coherent locally free sheaf $\mathfrak{M}$ of $\mathfrak{O}_\mathfrak{X}$-modules, we are now interested in constructing closed subvarieties $\mathfrak{S}_\mathbf{D} \subseteq \mathfrak{X}$ indexed by \emph{de Rham data} $\mathbf{D}$.  Such a datum consists of lower bounds for the dimensions of bounded de Rham periods in all intervals of weights at geometric specializations of $\mathfrak{S}_\mathbf{D}$.  Moreover, we want to decompose the entirety of $\mathfrak{X}$ into locally closed strata over which sheaves of bounded de Rham periods in a fixed range are locally free.  In the terminology of this section, Section \ref{sec:nonreducedbase} is a study of the interpolation and specialization properties of a \emph{full} de Rham datum for possibly non-reduced bases.

We will make use of the methodology of the previous sections.  Our approach is as follows.  We define the notion of a de Rham datum in Section \ref{subsec:drdatumdef} and study their existence, specialization, and interpolation on integral affinoid algebras $\mathfrak{R}$ in Section \ref{subsec:drdatumatt}.  In Section \ref{subsec:global}, we do the same for a reduced rigid space $\mathfrak{X}$.  In Section \ref{subsec:stratfam}, we state and prove the main theorems using the previous results.

\subsection{Definition of a de Rham datum} \label{subsec:drdatumdef}

We describe all of the terminology we will use to describe a de Rham datum.

We define a \emph{de Rham datum} $\mathbf{D} = (\Omega,\Delta)$ as a tuple consisting of functions $\Omega: \mathbf{Z} \rightarrow \mathbf{Z}_{\ge 0}$ and $\Delta: \mathbf{Z} \times \mathbf{Z} \rightarrow \mathbf{Z}_{\ge 0}$ such that
\begin{enumerate}[{\normalfont (i)}]
	\item the function $\Omega$ is finitely supported,
	\item we have $\Delta(i,j)= 0$ if $i \ge j$,
	\item we have $\min(\Omega(i),1) \le \Delta(i,i+1) \le \Omega(i)$ for $i \in \mathbf{Z}$, and
	\item we have the inequality $\max(\Delta(i,j),\Delta(j,k)) \le \Delta(i,k) \le \Delta(i,j) + \Delta(j,k)$ for all $i \le j \le k \in \mathbf{Z}$.
\end{enumerate}

\begin{remark}

One should think of $\Omega(i)$ as keeping track of the number of Hodge-Tate-Sen weights equal to $i$ and $\Delta(i,j)$ as keeping track of the dimension of de Rham periods in the weight range $[i,j)$ of a representation.  The conditions (i)-(iv) encode the relationships among these values implied by properties of the Sen operator and the structure of $\bdr$.

\end{remark}

We say that $\mathbf{D}$ is \emph{full} if $\Delta(i,i+1) = \Omega(i)$ and $\Delta(i,k) = \Delta(i,j) + \Delta(j,k)$ for all $i \le j\le k \in \mathbf{Z}$.

We say that $\mathbf{D}$ is a \emph{Hodge-Tate datum} if $\max(\Delta(i,j),\Delta(j,k)) = \Delta(i,k)$ for all $i\le j\le k \in \mathbf{Z}$.  We say that $\mathbf{D}$ is a \emph{Sen datum} if $\Delta(i,i+1) = \min(\Omega(i),1)$ and it is Hodge-Tate.  For a de Rham datum $\mathbf{D}$, we define its \emph{associated Hodge-Tate datum} $\mathrm{HT}(\mathbf{D}) = (\Omega_\mathrm{HT},\Delta_\mathrm{HT})$ by $\Omega_\mathrm{HT} = \Omega$ and $\Delta_\mathrm{HT}(i,j) = \max_{i \le k < j} \Delta(k,k+1)$ and its \emph{associated Sen datum} $\mathrm{Sen}(\mathbf{D}) = (\Omega_\mathrm{Sen},\Delta_\mathrm{Sen})$ by $\Omega_\mathrm{Sen} = \Omega$ and $\Delta_\mathrm{Sen}(i,j) = \max_{i \le k < j} \min(\Omega(k),1)$.

\begin{remark}

A Sen datum is determined by $\Omega$, which as mentioned above keeps track of the Hodge-Tate-Sen weights, and a Hodge-Tate datum is determined by $\Delta(i,i+1)$ for $i \in \mathbf{Z}$, which keeps track of the Hodge-Tate weights.

\end{remark}

For a de Rham datum $\mathbf{D}$, we define
\begin{itemize}
	\item the \emph{Sen dimension} by $\sd(\mathbf{D}) = \sum_{i \in \mathbf{Z}} \Omega(i)$,
	\item the \emph{Hodge-Tate dimension} by $\htd(\mathbf{D}) = \sum_{i \in \mathbf{Z}} \Delta(i,i+1)$ and its bounded variant $\htd^{(k,\ell)}(\mathbf{D}) = \sum_{i=k}^{\ell-1} \Delta(i,i+1)$, and
	\item the \emph{de Rham dimension} by $\drd(\mathbf{D}) = \max_{i,j \in \mathbf{Z}}\Delta(i,j)$.
\end{itemize}

\begin{remark}

The sums defining the first two are finite by conditions (i) and (iv), and satisfy $\htd(\mathbf{D}) \le \sd(\mathbf{D})$.  Observe that for any $i,j$, $\Delta(i,j) \le \sum_{\nu=i}^{j-1} \Delta(i,i+1)$ by (iv), so $\drd(\mathbf{D}) \le \htd(\mathbf{D})$.

\end{remark}

We define the $n^\textrm{th}$ twist of $\mathbf{D} = (\Omega,\Delta)$ by $\mathbf{D}(n)=(\Omega',\Delta')$, where $\Omega'(i) = \Omega(i+n)$ and $\Delta'(i,j) = \Delta(i+n,j+n)$.

We define the \emph{lower and upper bounds} of $\mathbf{D}$ by
\[L(\mathbf{D}) = \min\set{i:\Omega(i)>0}\textrm{  and  }U(\mathbf{D}) = \max\set{i:\Omega(i)>0},\]
respectively.  By conditions (ii), (iii), and (iv), $\mathbf{D}$ is determined by $\Omega|_{[L(\mathbf{D}),U(\mathbf{D})]}$ and $\Delta(i,j)$ for $L(\mathbf{D}) \le i < j \le  U(\mathbf{D})+1$.  We write $\supp(\mathbf{D}) = [L(\mathbf{D}),U(\mathbf{D})] \subseteq \mathbf{Z}$ for the \emph{support} of $\mathbf{D}$.

We define a various notions of ordering on de Rham data $\mathbf{D} = (\Omega,\Delta)$ and $\mathbf{D}' = (\Omega',\Delta')$ as follows.  The definition of $\le$ is the most important, but $<_\mathrm{min}^{[i,j]}$ will be used in Theorem \ref{thm:stratfam2} to study flatness of period sheaves.
\begin{itemize}
	\item We write $\mathbf{D}\le \mathbf{D}'$ if $\Omega(i) \le \Omega'(i)$ and $\Delta(i,j) \le \Delta'(i,j)$ for $i,j \in \mathbf{Z}$.
	\item We write $\mathbf{D} < \mathbf{D}'$ if $\mathbf{D} \le \mathbf{D}'$ but not $\mathbf{D}' \le \mathbf{D}$.
	\item We write $\mathbf{D} <^{[i,j]} \mathbf{D}'$ if $\mathbf{D} < \mathbf{D}'$ with $\supp(\mathbf{D}') \subseteq [i,j]$, and write $\mathbf{D} <_\mathrm{min}^{[i,j]} \mathbf{D}'$ if there is additionally no de Rham datum $\mathbf{D}''$ with $\mathbf{D} <^{[i,j]} \mathbf{D}'' <^{[i,j]} \mathbf{D}'$.
\end{itemize}
We also define the $[i,j]$-\emph{truncation} $\mathbf{D}^{[i,j]} = (\Omega^{[i,j]},\Delta^{[i,j]})$ of a de Rham datum $\mathbf{D}$ by $\Omega^{[i,j]} = \chi_{[i,j]}\Omega$ and $\Delta^{[i,j]}(k,\ell)=\Delta(\max(k,i),\min(\ell,j+1))$, where $\chi_{[i,j]}$ is the characteristic function of $[i,j]$.  Note that $\mathbf{D}^{[i,j]}$ is a de Rham datum and $\mathbf{D}^{[i,j]} \le \mathbf{D}$.

\subsection{The de Rham datum of a family} \label{subsec:drdatumatt}

We attach a de Rham datum to any $G_K$-representation on a free module over an integral affinoid $E$-algebra, and study its specialization and interpolation properties.

\begin{remark}

The properties of affinoid $E$-algebras $\mathfrak{R}$ we will need are that they are Jacobson, Noetherian, compatible with analytic field extension, and that their residue fields at maximal points are finite extensions of $E$.

\end{remark}

We will need the following.

\begin{lem}\label{lem:bdrkgeneric}

Maintain the notation of Section \ref{subsec:senkisin}. If $\mathfrak{R}$ is an integral Noetherian $E$-Banach algebra with fraction field $L$, then for any $k < \ell \in \mathbf{Z}$,
\begin{equation} \label{eqn:bdrkgeneric} \dim_L H^0(G_K,t^k\bdrk{\ell} \hat{\otimes}_{K,\sigma} \mathfrak{M}) \otimes_\mathfrak{R} L = \dim_L H^1(G_K,t^k\bdrk{\ell} \hat{\otimes}_{K,\sigma} \mathfrak{M}) \otimes_\mathfrak{R} L.\end{equation}

\end{lem}

\begin{proof}

First assume $\ell-k=1$.  By symmetry we can just assume $k=0$ and $\ell=1$.  As before, using Lemma \ref{lem:changeek}, we may replace $E,K$ with finite Galois extensions $E',K'$ so that Corollary \ref{coro:sencoh} holds.  Extending $E$ might cause $\mathfrak{R}$ to no longer be integral, but if we prove (\ref{eqn:bdrkgeneric}) for $L$ replaced by the localization at each minimal prime of $\mathfrak{R} \otimes_E E'$ the conclusion for $\mathfrak{R}$ follows from (\ref{eqn:edescalc}) and (\ref{eqn:eswitch}).

So assume Corollary \ref{coro:sencoh} holds, but that $\mathfrak{R}$ is only reduced, and let $\mathfrak{p}$ be a minimal prime of $\mathfrak{R}$.  We have $\phi: \mathfrak{E} \rightarrow \mathfrak{E}$ with $\mathfrak{E}$ finitely generated over $\mathfrak{R}$,
\[H^0(G_K,\mathbf{C}_p \hat{\otimes}_{K,\sigma} \mathfrak{M}) \cong \ker\phi, \textrm{ and } H^1(G_K,\mathbf{C}_p \hat{\otimes}_{K,\sigma} \mathfrak{M}) \cong \coker\phi.\]
Localizing at $\mathfrak{p}$, we obtain (\ref{eqn:bdrkgeneric}) with $L = \mathfrak{R}_\mathfrak{p}$ for $\ell-k=1$ by dimension counting.  The equation (\ref{eqn:bdrkgeneric}) for general $k$ follows by induction from the sequence similar to (\ref{eqn:exactsnake}) but instead constructed from (\ref{eqn:bdrexact1}) and localized at $\mathfrak{p}$.
\end{proof}

\begin{prop} \label{prop:drdatumaff}
Suppose $\mathfrak{R}$ is an integral affinoid $E$-algebra with $L = \Frac(\mathfrak{R})$ and $\mathfrak{M}$ is a free $\mathfrak{R}$-module equipped with a continuous $G_K$-action.  Let $P_\sigma(T) \in \mathfrak{R}[T]$ be the $\sigma$-factor of the associated Sen polynomial.  We define $\Omega_\mathfrak{M}(i)=\max\set{n:(T+i)^n|P_\sigma(T)}$ and
\[\Delta(i,j) = \dim_L H^0(G_K,t^i\bdrk{j} \hat{\otimes}_{K,\sigma} \mathfrak{M}) \otimes_\mathfrak{R} L\]
for $i, j \in\mathbf{Z}$ with $i < j$.  If  $i \ge j$, we set $\Delta_\mathfrak{M}(i,j)=0$.  Then $\mathbf{D}_\mathfrak{M} = (\Omega_\mathfrak{M},\Delta_\mathfrak{M})$ is a de Rham datum.

We have $\mathbf{D}_{\mathfrak{M}(n)} = \mathbf{D}_\mathfrak{M}(n)$ and $\Delta_\mathfrak{M}(i,j) = \dim_L H^1(G_K,t^i\bdrk{j} \hat{\otimes}_{K,\sigma} \mathfrak{M}) \otimes_\mathfrak{R} L$.

\end{prop}

\begin{proof}

The $(T+i)^n$ are coprime for distinct $i$, so condition (i) is clear and (ii) is true by definition.  The modules $H^r(G_K,t^k\bdrk{\ell} \hat{\otimes}_{K,\sigma} \mathfrak{M})$ for $r \in \set{0,1}$ are finitely generated by induction using Proposition \ref{prop:htgap} and the 6-term exact sequences constructed as in (\ref{eqn:exactsnake}) from the sequences (\ref{eqn:bdrexact1}) and (\ref{eqn:bdrexact3}).  Condition (iv) follows from tensoring $L$ with the long exact sequence similar to (\ref{eqn:exactsnake}), but constructed from
\[\ses{t^j\bdrk{k}}{t^i\bdrk{k}}{t^i\bdrk{j}}\]
instead of (\ref{eqn:mainexact}).  The equality $\mathbf{D}_{\mathfrak{M}(n)} = \mathbf{D}_\mathfrak{M}(n)$ is clear.  The last claim follows from Lemma \ref{lem:bdrkgeneric}.

We check condition (iii).  We may let $i=0$ by symmetry.  We write $P_\sigma(T) = S_\sigma(T)Q_\sigma(T)$, where $S_\sigma(T) = \prod_{i \in \mathbf{Z}} (T+i)^{\Omega_\mathfrak{M}(i)}$.  By Proposition \ref{prop:dr1cocycle}, for any maximal ideal $\mathfrak{m}\subseteq \mathfrak{R}$ with residue field $\kappa(\mathfrak{m})$, we have an isomorphism
\[H^1(G_K,\mathbf{C}_p \hat{\otimes}_{K,\sigma} \mathfrak{M}) \otimes_\mathfrak{R} \kappa(\mathfrak{m}) \cong H^1(G_K,\mathbf{C}_p \hat{\otimes}_{K,\sigma} (\mathfrak{M} \otimes_\mathfrak{R} \kappa(\mathfrak{m}))).\]
Since $H^1(G_K,\mathbf{C}_p \hat{\otimes}_{K,\sigma} \mathfrak{M})$ is finitely generated, there exists a nonzero element $f \in \mathfrak{R}$ so that $H^1(G_K,\mathbf{C}_p \hat{\otimes}_{K,\sigma} \mathfrak{M})_f$ is free.  By definition, $Q_\sigma(0) \ne 0$, so by the Jacobson property, $fQ_\sigma(0)$ is nonvanishing in $\kappa(\mathfrak{m})$ for some $\mathfrak{m}$.  By (\ref{eqn:bclocal}) we have an identification
\[H^1(G_K,\mathbf{C}_p \hat{\otimes}_{K,\sigma} \mathfrak{M}) \otimes_\mathfrak{R} \kappa(\mathfrak{m}) \cong H^1(G_K,\mathbf{C}_p \hat{\otimes}_{K,\sigma} \mathfrak{M})_{fQ_\sigma(0)} \otimes_{\mathfrak{R}_{fQ_\sigma(0)}} \kappa(\mathfrak{m}).\]
In particular, $\dim_{\kappa(\mathfrak{m})} H^1(G_K,\mathbf{C}_p \hat{\otimes}_{K,\sigma} (\mathfrak{M} \otimes_\mathfrak{R} \kappa(\mathfrak{m}))) = \Delta_\mathfrak{M}(0,1) \le \Omega_\mathfrak{M}(0)$ by Lemma \ref{lem:dimh0h1} and compatibility of the Sen polynomial with specialization.  If the bound $\Omega_\mathfrak{M}(0) \ge 1$ holds, then $\dim_{\kappa(\mathfrak{m})} H^1(G_K,\mathbf{C}_p \hat{\otimes}_{K,\sigma} (\mathfrak{M} \otimes_\mathfrak{R} \kappa(\mathfrak{m}))) \ge 1$ since $T$ divides the image of $P_\sigma(T)$ in $\kappa(\mathfrak{m})[T]$, so $\Delta_\mathfrak{M}(0,1) \ge 1$.
\end{proof}

We study the behavior of a de Rham datum under specialization.

\begin{prop} \label{prop:drdatumspecialaff}

Suppose that $\mathfrak{R}$, $\mathfrak{M}$, and $\mathbf{D}_\mathfrak{M}$ are as in Proposition \ref{prop:drdatumaff}.  Then for any prime ideal $\mathfrak{p}\subseteq \mathfrak{R}$ with quotient $\mathfrak{R}'=\mathfrak{R}/\mathfrak{p}$, $\mathbf{D}_{\mathfrak{M} \otimes_\mathfrak{R} \mathfrak{R}'} \ge \mathbf{D}_\mathfrak{M}$.

\end{prop}

\begin{proof}

The inequality $\Omega_{\mathfrak{M} \otimes_\mathfrak{R} \mathfrak{R}'} \ge \Omega_{\mathfrak{M}}$ is clear.  For any $k,\ell \in \mathbf{Z}$ we have an isomorphism
\begin{equation} \label{eqn:bciso} H^1(G_K,t^k\bdrk{\ell} \hat{\otimes}_{K,\sigma} \mathfrak{M}) \otimes_\mathfrak{R} \mathfrak{R}' \cong H^1(G_K,t^k\bdrk{\ell} \hat{\otimes}_{K,\sigma} (\mathfrak{M} \otimes_\mathfrak{R} \mathfrak{R}'))\end{equation}
by Proposition \ref{prop:dr1cocycle} (and twisting).  Let $L = \Frac(\mathfrak{R})$ and $\kappa = \Frac(\mathfrak{R}')$.  By (\ref{eqn:bciso}) and (\ref{eqn:bclocal}) we have
\begin{align*} \dim_\kappa H^1(G_K,t^k\bdrk{\ell} \hat{\otimes}_{K,\sigma} \mathfrak{M})_{\mathfrak{p}} \otimes_{\mathfrak{R}_\mathfrak{p}} \kappa &= \Delta_{\mathfrak{M} \otimes_\mathfrak{R} \mathfrak{R}'}(k,\ell)\\ \textrm{and}\quad\dim_L H^1(G_K,t^k\bdrk{\ell} \hat{\otimes}_{K,\sigma} \mathfrak{M})_{\mathfrak{p}} \otimes_{\mathfrak{R}_\mathfrak{p}} L &= \Delta_\mathfrak{M}(k,\ell).\end{align*}
We deduce the inequality $\Delta_{\mathfrak{M} \otimes_\mathfrak{R} \mathfrak{R}'}(k,\ell) \ge \Delta_\mathfrak{M}(k,\ell)$ from Nakayama's lemma.
\end{proof}

We prove a converse to Proposition \ref{prop:drdatumspecialaff}.  Taken together, they form an interpolation result for de Rham data.  If $x \in \spm \mathfrak{R}$ has residue field $\kappa(x)$, we write $\mathbf{D}_{\mathfrak{M},x}$ for $\mathbf{D}_{\mathfrak{M} \otimes_\mathfrak{R} \kappa(x)}$.

\begin{prop} \label{prop:drdatuminterpaff}

Maintain the notation and hypotheses of Proposition \ref{prop:drdatumaff}.  Let $\mathfrak{X} = \spm\mathfrak{R}$, let $\mathbf{D} = (\Omega,\Delta)$ be a de Rham datum, assume that $\Xi \subseteq \mathfrak{X}$ is Zariski dense, and assume that $\mathbf{D}_{\mathfrak{M},x} \ge \mathbf{D}$ for all $x \in \Xi$.  Then $\mathbf{D}_\mathfrak{M} \ge \mathbf{D}$.

\end{prop}

\begin{proof}
To see that $\Omega_\mathfrak{M} \ge \Omega$, write the $\sigma$-factor $P_\sigma(T)$ of the Sen polynomial of $\mathfrak{M}$ as a function of $T+i$ and observe that the images of the lowest $\Omega_\mathfrak{M}(i)$ coefficients vanish in $(\prod_{x \in \Xi} \kappa(x))[T]$.

By generic freeness (and the finite generation from the proof of Proposition \ref{prop:drdatumaff}), there is a nonzero element $f \in \mathfrak{R}$ such that $H^1(G_K,t^k\bdrk{\ell} \hat{\otimes}_{K,\sigma} \mathfrak{M})_f$ is free of rank $\Delta_{\mathfrak{M}}(k,\ell)$.  

By Proposition \ref{prop:dr1cocycle} and (\ref{eqn:bclocal}), for any $x \in \spm\mathfrak{R} \setminus V((f))$, we have a natural isomorphism
\[H^1(G_K,t^k\bdrk{\ell} \hat{\otimes}_{K,\sigma} \mathfrak{M})_f \otimes_{\mathfrak{R}_f} \kappa(x) \iso H^1(G_K,t^k\bdrk{\ell} \hat{\otimes}_{K,\sigma} (\mathfrak{M} \otimes_\mathfrak{R} \kappa(x))\]
By picking $x \in \Xi \cap (\spm\mathfrak{R} \setminus V((f)))$, it follows that $\Delta_{\mathfrak{M}}(k,\ell) = \Delta_{\mathfrak{M},x}(k,\ell) \ge \Delta(k,\ell)$.
\end{proof}

We can deduce the following.

\begin{coro} \label{coro:drdatuminterp}

Maintain the hypotheses of Proposition \ref{prop:drdatumaff}.  Suppose $\xi:\mathfrak{R} \rightarrow \mathfrak{R}'$ is a finite map of integral $E$-affinoid algebras, and let $\mathfrak{M}' = \mathfrak{M}\otimes_\mathfrak{R} \mathfrak{R}'$.  Then $\mathbf{D}_{\mathfrak{M}'} \ge \mathbf{D}_{\mathfrak{M}}$.  If $\Xi \subseteq \spm\mathfrak{R}$ is Zariski-dense in the image of $\xi^*$ and $\mathbf{D}$ is a de Rham datum with $\mathbf{D}_{\mathfrak{M},x} \ge \mathbf{D}$ for $x \in \Xi$, then $\mathbf{D}_{\mathfrak{M}'} \ge \mathbf{D}$.

\end{coro}

\begin{proof}
We may factor our map as $\mathfrak{R} \rightarrow \mathfrak{R}'' \rightarrow \mathfrak{R}'$, where the first map is surjective and the second is injective, so in particular $\mathfrak{R}''$ is also integral.  We apply Proposition \ref{prop:drdatumspecialaff} to $\mathfrak{R} \rightarrow \mathfrak{R}''$ to obtain $\mathbf{D}_{\mathfrak{M} \otimes_\mathfrak{R} \mathfrak{R}''} \ge \mathbf{D}_{\mathfrak{M}}$ and use the definition of the de Rham datum to find $\mathbf{D}_{\mathfrak{M} \otimes_\mathfrak{R} \mathfrak{R}''} = \mathbf{D}_{\mathfrak{M} \otimes_\mathfrak{R} \mathfrak{R}'}$, since $\Frac(\mathfrak{R}')$ is just a field extension of $\Frac(\mathfrak{R}'')$ and the equality $\Omega_{\mathfrak{M} \otimes_\mathfrak{R} \mathfrak{R}''} = \Omega_{\mathfrak{M} \otimes_\mathfrak{R} \mathfrak{R}'}$ is clear.  For the second claim, we apply Proposition \ref{prop:drdatuminterpaff} to $\mathfrak{R}''$ and the first claim to the map $\mathfrak{R}'' \rightarrow \mathfrak{R}'$.
\end{proof}

\subsection{Globalization} \label{subsec:global}

Specialization of a de Rham datum is a local phenomenon, but one may ask for a global version of the definition and interpolation of a de Rham datum.

For a rigid analytic space $\mathfrak{X}$, we use $\mathfrak{O}_\mathfrak{X}$ to denote the structure sheaf, $\mathfrak{O}_{\mathfrak{X},x}$ for the stalk at $x \in \mathfrak{X}$, and $\kappa(x)$ for the residue field.  For a coherent sheaf $\mathfrak{M}$ on $\mathfrak{X}$, we write $\mathfrak{M}_x$ for the stalk at $x$ and define $\ol{\mathfrak{M}}_x = \mathfrak{M}_x \otimes_{\mathfrak{O}_{\mathfrak{X},x}} \kappa(x)$.  We define $\mathbf{D}_{\mathfrak{M},x} = \mathbf{D}_{\ol{\mathfrak{M}}_x}$.

If $\mathfrak{X}$ is a rigid analytic space and $\mathfrak{M}$ is a coherent locally free sheaf on $\mathfrak{X}$, we say that a continuous $G_K$-representation on $\mathfrak{M}$ is a map $G_K \rightarrow \End_{\mathfrak{O}_\mathfrak{X}}(\mathfrak{M})$ such that the restriction to any affinoid is continuous with respect to the canonical topology.

Since the $\sigma$-factor $P_\sigma(T)$ of the Sen polynomial attached to a finite free module $\mathfrak{M}$ over an $E$-Banach algebra $\mathfrak{R}$ and $\sigma \in \Sigma$ as before is canonical, if $\mathfrak{M}$ is instead a coherent locally free sheaf over a rigid space $\mathfrak{X}$, one can patch together locally defined Sen polynomials on a sufficiently fine admissible affinoid cover to obtain $P_\sigma(T) \in \mathfrak{O}_\mathfrak{X}(\mathfrak{X})$.

We will use the notion of an irreducible component of a rigid analytic space $\mathfrak{X}$ \cite{con}.  Following Conrad, we say that a closed analytic subvariety $\mathfrak{Z} \subseteq \mathfrak{X}$ is nowhere dense in $\mathfrak{X}$ if it contains no non-empty admissible open of $\mathfrak{X}$.  Equivalently, $\dim \mathfrak{O}_{\mathfrak{X},z} > \dim \mathfrak{O}_{\mathfrak{Z},z}$ for all $z\in\mathfrak{Z}$ \cite[\S2]{con}.  This local definition implies that the finite union of nowhere dense closed analytic subvarieties is again nowhere dense.  If $\mathfrak{X}$ is irreducible, then every proper closed analytic subvariety is nowhere dense \cite[Lemma 2.2.3]{con}.

We attach a de Rham datum to any reduced, irreducible rigid analytic space $\mathfrak{X}$ and a $G_K$-representation on a coherent locally free sheaf $\mathfrak{M}$.  We begin with the case where $\mathfrak{X}$ is normal.

\begin{prop} \label{prop:drdatumglobalnorm}

Let $\mathfrak{X}$ be a reduced, normal, irreducible rigid analytic space and let $\mathfrak{M}$ be a coherent locally free sheaf equipped with a continuous $G_K$-action.  Then there is a naturally associated de Rham datum $\mathbf{D}_\mathfrak{M}$ such that for any connected (thus integral) affinoid subdomain $\mathfrak{U} \subseteq \mathfrak{X}$ with $\mathfrak{M}(\mathfrak{U})$ free, we have $\mathbf{D}_{\mathfrak{M}(\mathfrak{U})} = \mathbf{D}_\mathfrak{M}$.

\end{prop}

\begin{proof}
It suffices to check that if connected affinoid subdomains $\mathfrak{U},\mathfrak{V} \subseteq \mathfrak{X}$ have $\mathfrak{M}(\mathfrak{U})$ and $\mathfrak{M}(\mathfrak{V})$ free, then $\mathbf{D}_{\mathfrak{M}(\mathfrak{U})} = \mathbf{D}_{\mathfrak{M}(\mathfrak{V})}$.  Fix an admissible covering $\set{\mathfrak{U}_j}_{j \in J}$ of $\mathfrak{X}$ by connected open affinoid subdomains over which $\mathfrak{M}$ is free.  We claim that there exists a sequence $\mathfrak{U}_{j_0}, \mathfrak{U}_{j_1},\dots,\mathfrak{U}_{j_n}$ of affinoid subdomains from the covering such that $\mathfrak{U} \cap \mathfrak{U}_{j_0} \ne \emptyset$, $\mathfrak{U}_{j_n} \cap \mathfrak{V} \ne \emptyset$, and $\mathfrak{U}_{j_\nu} \cap \mathfrak{U}_{j_{\nu+1}}\ne \emptyset$ for each $\nu \in [0,n-1]$.  Indeed, observe that the union $\mathfrak{X}_\mathfrak{U}$ of all $\mathfrak{U}_j$ for $j \in J$ such that $\mathfrak{U}_j$ can be reached in finitely many steps by such a sequence starting from $\mathfrak{U}$ is closed, since it either contains or does not intersect each $\mathfrak{U}_j$.  Since $\mathfrak{X}_\mathfrak{U}$ contains an affinoid and $\mathfrak{X}$ is irreducible, we must have $\mathfrak{X}_\mathfrak{U} = \mathfrak{X}$.

By the existence of such a sequence, we are reduced to the case where $\mathfrak{U} \cap \mathfrak{V} \ne 0$.  Write $\mathfrak{U} = \spm \mathfrak{R}_1$ and $\mathfrak{V} = \spm \mathfrak{R}_2$.  Fix any connected affinoid subdomain $\spm \mathfrak{R} = \mathfrak{W} \subseteq \mathfrak{U} \cap \mathfrak{V}$.  As in the proof of Proposition \ref{prop:drdatuminterpaff}, there exist nonzero $f \in \mathfrak{R}_1$ and $g \in \mathfrak{R}_2$ so that the localizations $H^1(G_K,t^k\bdrk{\ell} \hat{\otimes}_{K,\sigma} \mathfrak{M}(\mathfrak{U}))_f$ and $H^1(G_K,t^k\bdrk{\ell} \hat{\otimes}_{K,\sigma} \mathfrak{M}(\mathfrak{V}))_g$ are free.  The proof of Proposition \ref{prop:drdatuminterpaff} also shows that for any $x \in \mathfrak{U} \setminus V((f))$, we have $\Delta_{\mathfrak{M},x}(k,\ell) = \Delta_{\mathfrak{M}(\mathfrak{U})}(k,\ell)$, and similarly for $x \in \mathfrak{V} \setminus V((g))$.  Both $f$ and $g$ define nonzero elements of $\mathfrak{R}$.  Since $\mathfrak{R}$ is integral, we may choose $x \in \mathfrak{W} \setminus V((fg))$.  This gives the equality $\Delta_{\mathfrak{M}(\mathfrak{U})} = \Delta_{\mathfrak{M}(\mathfrak{V})}$.

By irreducibility of $\mathfrak{U}$ and $\mathfrak{V}$ and by examining the vanishing locus of the lowest order coefficients of $P_\sigma(T)|_{\mathfrak{U}}$ and $P_\sigma(T)|_{\mathfrak{V}}$ on $\mathfrak{W}$, we find $\Omega_{\mathfrak{M}(\mathfrak{U})}(0) = \Omega_{\mathfrak{M}(\mathfrak{V})}(0)$; $\Omega_{\mathfrak{M}(\mathfrak{U})} = \Omega_{\mathfrak{M}(\mathfrak{V})}$ follows similarly.
\end{proof}

We obtain the general case by applying Corollary \ref{coro:drdatuminterp} to the normalization of $\mathfrak{X}$.

\begin{prop} \label{prop:drdatumglobal}

Let $\mathfrak{X}$ be a reduced, irreducible rigid analytic space and let $\mathfrak{M}$ be a coherent locally free sheaf equipped with a continuous $G_K$-action.  Then there is a naturally associated de Rham datum $\mathbf{D}_\mathfrak{M}$ such that for any irreducible component (with the reduced structure) $\mathfrak{V}$ of any affinoid subdomain $\mathfrak{U} \subseteq \mathfrak{X}$ with $\mathfrak{M}(\mathfrak{U})$ free, if we write $\mathfrak{M}'$ for the pullback to $\mathfrak{V}$, we have $\mathbf{D}_{\mathfrak{M}'} = \mathbf{D}_\mathfrak{M}$.

\end{prop}

\begin{proof}
Let $\pi:\widetilde{\mathfrak{X}} \rightarrow \mathfrak{X}$ be the normalization of $\mathfrak{X}$.  Let $\mathfrak{U}$ and $\mathfrak{V}$ be as in the statement of the proposition.  The restriction $\pi|_{\pi^{-1}(\mathfrak{U})}: \widetilde{\mathfrak{U}} \rightarrow \mathfrak{U}$ is a finite map of reduced affinoids.  Let $\mathfrak{R} \rightarrow \mathfrak{R}'$ be the corresponding map of reduced affinoid algebras.  The map factors as $\mathfrak{R} \rightarrow \prod_\mathfrak{p} \mathfrak{R}/\mathfrak{p} \rightarrow \prod_\mathfrak{p} \mathfrak{R}_{(\mathfrak{p})}'$, where $\mathfrak{p}$ ranges over the minimal primes of $\mathfrak{R}$ and $\mathfrak{R}_{(\mathfrak{p})}'$ denotes the normalization of $\mathfrak{R}/\mathfrak{p}$.  Choose $\mathfrak{p}$ corresponding to $\mathfrak{V}$.  The proof of Corollary \ref{coro:drdatuminterp} shows that $\mathbf{D}_{\mathfrak{M}(\mathfrak{U}) \otimes_\mathfrak{R} \mathfrak{R}/\mathfrak{p}} = \mathbf{D}_{\mathfrak{M}(\mathfrak{U}) \otimes_\mathfrak{R} \mathfrak{R}_{(\mathfrak{p})}'}$.  By Proposition \ref{prop:drdatumglobalnorm}, $\mathbf{D}_{\mathfrak{M}(\mathfrak{U}) \otimes_\mathfrak{R} \mathfrak{R}_{(\mathfrak{p})}'} = \mathbf{D}_{\pi^*\mathfrak{M}}$, which we take to be $\mathbf{D}_\mathfrak{M}$.
\end{proof}

We now globalize the interpolation of a de Rham datum.  We will assume Zariski density for our interpolation result, which has the following definition for a general rigid analytic space.

\begin{defin} \label{defin:zdense}

We say that a subset $S \subseteq \mathfrak{X}$ of a rigid analytic space is \emph{Zariski dense} if every closed analytic subvariety $\mathfrak{Z} \subseteq \mathfrak{X}$ containing $S$ is equal to $\mathfrak{X}$.

\end{defin}

\begin{prop} \label{prop:drdatuminterp}

Suppose that $\mathfrak{X}$, $\mathfrak{M}$, and $\mathbf{D}_\mathfrak{M}$ are as in Proposition \ref{prop:drdatumglobal}.  Let $\mathbf{D} = (\Omega,\Delta)$ be a de Rham datum, assume that $\Xi \subseteq \mathfrak{X}$ is Zariski dense, and assume that $\mathbf{D}_{\mathfrak{M},x} \ge \mathbf{D}$ for all $x \in \Xi$.  Then $\mathbf{D}_\mathfrak{M} \ge \mathbf{D}$.

\end{prop}

\begin{proof}
Write $P_\sigma(T)$ for the $\sigma$-factor of the Sen polynomial of $\mathfrak{M}$.  To see that $\Omega_\mathfrak{M} \ge \Omega$, observe that the vanishing locus of any of the first $\Omega(i)$ coefficients of $P_\sigma(T)$ expanded around $T+i$ includes $\Xi$ and is closed.

We claim that the locus $\mathfrak{Y}$ of points $x \in \mathfrak{X}$ such that $\mathbf{D}_{\mathfrak{M},x} \ge \mathbf{D}$ is closed.  To see this, pick an admissible cover of $\mathfrak{X}$ by connected affinoid subsets $\mathfrak{U}$ small enough that $\mathfrak{M}(\mathfrak{U})$ is free, and let $\ol{\mathfrak{Y}}_\mathfrak{U}$ denote the Zariski closure of $\mathfrak{Y} \cap \mathfrak{U}$.  We apply Corollary \ref{coro:drdatuminterp} to each irreducible component $\mathfrak{Y}' \rightarrow \ol{\mathfrak{Y}}_\mathfrak{U}$ to deduce that the pullback $\mathfrak{M}'$ of $\mathfrak{M}(\mathfrak{U})$ to $\mathfrak{Y}'$ satisfies $\mathbf{D}_{\mathfrak{M}'}\ge \mathbf{D}$.  By Proposition \ref{prop:drdatumspecialaff}, we deduce that $\mathfrak{Y} \cap \mathfrak{U} = \ol{\mathfrak{Y}}_\mathfrak{U}$.

Thus $\mathfrak{Y}$ is closed and contains the Zariski-dense subset $\Xi$, so $\mathfrak{Y} = \mathfrak{X}$.  The result now follows from applying Proposition \ref{prop:drdatuminterpaff} to any affinoid of $\mathfrak{X}$ small enough that $\mathfrak{M}(\mathfrak{U})$ is free and the dense set consisting of all of its points.
\end{proof}

We deduce the following globalization of Corollary \ref{coro:drdatuminterp}.

\begin{coro} \label{coro:drdatuminterpglobal}

Suppose $\pi:\mathfrak{Y} \rightarrow \mathfrak{X}$ is a finite map of reduced, irreducible rigid spaces and $\mathfrak{M}$ is a coherent locally free sheaf on $\mathfrak{X}$ equipped with a continuous $G_K$-action as before.  Let $\mathfrak{M}' = \pi^*\mathfrak{M}$.  Then $\mathbf{D}_{\mathfrak{M}'} \ge \mathbf{D}_{\mathfrak{M}}$.

Moreover, if $\Xi \subseteq \mathfrak{X}$ is Zariski-dense in the image of $\mathfrak{Y}$ and $\mathbf{D}$ is a de Rham datum with $\mathbf{D}_{\mathfrak{M},x} \ge \mathbf{D}$ for $x \in \Xi$, then $\mathbf{D}_{\mathfrak{M}'} \ge \mathbf{D}$.

\end{coro}

\begin{proof}
For the first claim, pick an affinoid open $\mathfrak{U} \subseteq \mathfrak{X}$ so that $\mathfrak{M}(\mathfrak{U})$ is free, and let $\mathfrak{V} \subseteq \mathfrak{Y}$ be its preimage.  By \cite[Corollary 9.4.4/2]{bgr}, $\mathfrak{V}$ is an affinoid open.  Pick an irreducible component $\mathfrak{Z} \subseteq \mathfrak{U}$, pick an irreducible component $\mathfrak{Z}' \subseteq \pi^{-1}(\mathfrak{Z})$, and write $\xi: \mathfrak{Z} \rightarrow \mathfrak{X}$ and $\xi': \mathfrak{Z}' \rightarrow\mathfrak{Y}$ for the inclusions.  Apply Corollary \ref{coro:drdatuminterp} and Proposition \ref{prop:drdatumglobal} to the map $\mathfrak{Z}' \rightarrow \mathfrak{Z}$ to find that $\mathbf{D}_{\mathfrak{M}'} = \mathbf{D}_{\xi^{\prime *}\mathfrak{M}'} \ge \mathbf{D}_{\xi^*\mathfrak{M}} = \mathbf{D}_{\mathfrak{M}}$.

For the second claim, the image of $\mathfrak{Y}$ is an irreducible closed analytic subvariety \cite[Propositions 9.6.2/5 and 9.6.3/3]{bgr}, so we may factor our map as $\mathfrak{Y} \rightarrow \mathfrak{Z} \rightarrow \mathfrak{X}$, where the second map is a closed immersion.  We then apply Proposition \ref{prop:drdatuminterp} to $\mathfrak{Z}$ and the first claim to the map $\mathfrak{Y} \rightarrow \mathfrak{Z}$.
\end{proof}

\subsection{Stratification of families of Galois representations} \label{subsec:stratfam}

We collect the results of this section into a simple statement.  The motivating case is where $\mathfrak{X}$ is the global eigenvariety on a group and $\mathfrak{M}$ is the associated Galois representation over $\mathfrak{X}$.  A discussion of the importance of understanding the geometry of $\mathfrak{X}$ can be found in, for instance, Bella\"iche-Chenevier \cite{bch}.

We will use the following flatness criterion.

\begin{lem} \label{lem:flatcrit}

Suppose that $R$ is a reduced Noetherian ring with total ring of fractions $K$ and $M$ is a finitely generated $R$-module.  Suppose that $M \otimes_R K$ is free of rank $r$ and that $\dim_{\kappa(\mathfrak{m})} M \otimes_R \kappa(\mathfrak{m}) = r$ for every maximal ideal $\mathfrak{m} \subseteq R$, denoting the residue field of $\mathfrak{m}$ by $\kappa(\mathfrak{m})$.  Then $M$ is locally free of rank $r$.

\end{lem}

\begin{proof}
It suffices to check that $M_\mathfrak{m}$ is free of rank $r$ for all maximal $\mathfrak{m} \subseteq R$ by \cite[Theorem II.5.2.1]{bourbakica}.  Also note that by hypothesis $M_\mathfrak{p}$ has dimension $r$ for any minimal prime $\mathfrak{p}$ of $R$.

Let $L = \prod_{\mathfrak{p}} R_\mathfrak{p}$ be the total ring of fractions of $R_\mathfrak{m}$, where $\mathfrak{p}$ ranges over the minimal primes of $R_\mathfrak{m}$.  Using Nakayama's lemma, let $\psi:R_\mathfrak{m}^r \surj M_\mathfrak{m}$ be a surjection, and suppose that $a$ is a nonzero element of its kernel.  Since $R_\mathfrak{m}^r \rightarrow L^r$ is injective, there is some $\mathfrak{p}$ so that $a$ is nonzero in $R_\mathfrak{p}^r$.  Localizing $\psi$ at $\mathfrak{p}$, we obtain $R_\mathfrak{p}^r \surj M_\mathfrak{p}$.  By dimension count, this map is an isomorphism, so $a$ cannot have been nonzero.  In particular, $\psi$ is an isomorphism as needed.
\end{proof}

We will use the following to reduce checking flatness and base change to the study of 1-cocycles.

\begin{lem} \label{lem:flat1cocycle}

Maintain the notation of Section \ref{subsec:senkisin}.  Let $[i,j]\subseteq \mathbf{Z}$ be a fixed interval, let $f\in \mathfrak{R}$, and suppose that $H^1(G_K,t^k\bdrk{\ell} \hat{\otimes}_{K,\sigma} \mathfrak{M})_f$ is flat for $i \le k < \ell \le j+1$.  Then the following hold.
\begin{enumerate}[{\normalfont (a)}]
	\item The $\mathfrak{R}_f$-module $H^0(G_K,t^k\bdrk{\ell} \hat{\otimes}_{K,\sigma} \mathfrak{M})_f$ is flat for $i \le k < \ell \le j+1$.
	\item If $\xi:\mathfrak{R}\rightarrow \mathfrak{R}'$ is a map of Noetherian $E$-Banach algebras, the natural maps
	\begin{equation}\label{eqn:flat6bc} H^r(G_K,t^k\bdrk{\ell} \hat{\otimes}_{K,\sigma} \mathfrak{M}) \otimes_\mathfrak{R} \mathfrak{R}' \rightarrow H^r(G_K,t^k\bdrk{\ell} \hat{\otimes}_{K,\sigma} (\mathfrak{M} \otimes_\mathfrak{R} \mathfrak{R}'))\end{equation}
	have kernel and cokernel annihilated by a power of $\xi(f)$ for $r \in \set{0,1}$.
\end{enumerate}

\end{lem}

\begin{proof}
For a fixed interval $[i,j]$, we prove both results by induction on $\ell-k$.  If $\ell-k=1$, after increasing $E$ and $K$, both parts follow from Corollary \ref{coro:sencoh} and Lemma \ref{lem:niceend}.  The reduction to the original $E$ and $K$ at the beginning of the proof of Theorem \ref{thm:htcase}.(a) applies here for both parts, giving the base case.

For part (a), observe that if $R$ is a ring, $0 \rightarrow A \rightarrow B \rightarrow C \rightarrow D \rightarrow E \rightarrow F \rightarrow 0$ is a 6-term exact sequence of $R$-modules, and $A,C,D,E,$ and $F$ are flat over $R$, then $B$ is flat as well.  For the inductive step, we localize the 6-term exact sequence similar to (\ref{eqn:exactsnake}) but starting instead with (\ref{eqn:bdrexact1}) at $f$ to reduce to a case with smaller $\ell-k$.

For part (b), first note that as in the proof of Theorem \ref{thm:drcase}.(b), it suffices to check that the map is an isomorphism after localization at $\xi(f)$.  We consider the map between two rows of 6-term complexes as in (\ref{eqn:bcexactsnake1}) and (\ref{eqn:bcexactsnake2}), except beginning with the exact sequence (\ref{eqn:bdrexact1}).  In particular, the downward maps have the form of (\ref{eqn:flat6bc}).  The lower row is exact, while the top row becomes exact after localizing at $\xi(f)$ by (\ref{eqn:bclocal}) and the flatness of $H^r(G_K,t^k\bdrk{\ell} \hat{\otimes}_{K,\sigma} \mathfrak{M})_f$ from part (a).  By the inductive hypothesis, after localizing at $\xi(f)$, the first, third, fourth, and sixth downward maps are isomorphisms, which gives the result by the 5-lemma.
\end{proof}

We also check a finiteness property of de Rham data.  In fact, what is useful here is not the finiteness but the somewhat explicit description of the $\mathbf{D}'$ with $\mathbf{D} <_\mathrm{min}^{[i,j]} \mathbf{D}'$.

\begin{lem} \label{lem:minfin}

Let $\mathbf{D}=(\Omega,\Delta)$ be a de Rham datum.  For a fixed interval $[i,j] \supseteq \supp(\mathbf{D})$, there are finitely many de Rham data $\mathbf{D}'$ with $\mathbf{D} <_\mathrm{min}^{[i,j]} \mathbf{D}'$.

\end{lem}

\begin{proof}
We claim that it suffices to define de Rham data $\mathbf{D}_w =(\Omega_w,\Delta_w)$ for each $w \in [i,j]$ that have the properties
\begin{itemize}
	\item $\mathbf{D} <_\mathrm{min}^{[i,j]} \mathbf{D}_w$ and
	\item if $\mathbf{D}'=(\Omega',\Delta')$ satisfies $\mathbf{D}' \ge \mathbf{D}$ and $\Omega'(w) > \Omega(w)$, then $\mathbf{D} \ge \mathbf{D}_w$.
\end{itemize}
To see this, suppose such $\mathbf{D}_w$'s are given.  For any $\mathbf{D}'$ with $\mathbf{D} <_\mathrm{min}^{[i,j]} \mathbf{D}'$, we must either have (a) $\Omega'=\Omega$ or (b) $\mathbf{D}'=\mathbf{D}_w$ for some $w$ (due to minimality).  There are finitely many $\mathbf{D}_w$'s, so it suffices to check that there are only finitely many possibilities for $\mathbf{D}'$ in case (a).  But it is easy to see that if $\Omega= \Omega'$, there can only be finitely many choices $\Delta'$ meeting the defining conditions.

We define $\Omega_w(i) = \Omega(i)$ for $i\ne w$ and $\Omega_w(w)=\Omega(w)+1$, and for $k < \ell$, we set $\Delta_w(k,\ell) = \max(\Delta(k,\ell),1)$ if $k \le w < \ell$ and $\Delta_w(k,\ell)=\Delta(k,\ell)$ otherwise.  It is clear that $\mathbf{D}_w$ is a de Rham datum and has the needed properties.
\end{proof}

We denote the set of such $\mathbf{D}'$ by $\mathrm{Min}(i,j)$.  If $\mathbf{D}$ is a full de Rham datum, then the proof shows that $\set{\mathbf{D}_w}_{w \in [i,j]} = \mathrm{Min}(i,j)$.

We will use the following gluing result to globalize the construction of sheaves of periods and 1-cocycles.

\begin{lem} \label{lem:constructsheaves}

Let $\mathfrak{X}$ be a rigid analytic space over $E$, let $\mathfrak{M}$ be a coherent locally free sheaf of $\mathfrak{O}_\mathfrak{X}$-modules equipped with a continuous $G_K$-action.

Let $\mathfrak{V} \subseteq \mathfrak{U}$ be a pair of affinoid subdomains of $\mathfrak{X}$ such that $\mathfrak{M}(\mathfrak{U})$ is free, and let $\mathfrak{U} = \spm \mathfrak{R}_\mathfrak{U}$ and $\mathfrak{V} = \spm\mathfrak{R}_\mathfrak{V}$.  Assume that for any such pair $(\mathfrak{U},\mathfrak{V})$, the natural map
\begin{equation}\label{eqn:hrhyp} H^r(G_K,t^k\bdrk{\ell} \hat{\otimes}_{K,\sigma} \mathfrak{M}(\mathfrak{U})) \otimes_{\mathfrak{R}_\mathfrak{U}} \mathfrak{R}_\mathfrak{V}\rightarrow H^r(G_K,t^k\bdrk{\ell} \hat{\otimes}_{K,\sigma} \mathfrak{M}(\mathfrak{V}))\end{equation}
is an isomorphism.

Then there are coherent sheaves $\mathfrak{H}^r_{(k,\ell)}(\mathfrak{M})$ on $\mathfrak{X}$ for $r \in \set{0,1}$ such that for any affinoid subdomain $\mathfrak{W}$ with $\mathfrak{M}(\mathfrak{W})$ free,
\begin{equation}\label{eqn:hrdef} \mathfrak{H}_{(k,\ell)}^r(\mathfrak{W}) \cong H^r(G_K, t^k\bdrk{\ell} \hat{\otimes}_{K,\sigma} \mathfrak{M}(\mathfrak{W})).\end{equation}

\end{lem}

\begin{proof}
Let $\set{\mathfrak{U}_j}$ be an admissible affinoid cover of $\mathfrak{X}$ such that $\mathfrak{M}(\mathfrak{U}_j)$ is free for each $j$.  We claim that if we define $\mathfrak{H}_r^{(k,\ell)}(\mathfrak{M})|_{\mathfrak{U}_j}$ to be the sheaf on $\mathfrak{U}_j$ associated to the module in (\ref{eqn:hrdef}) for each $j$, then these naturally patch together on the intersections to form a coherent sheaf.  So let $\mathfrak{V} \subseteq \mathfrak{U}_{j_1} \cap \mathfrak{U}_{j_2}$, where all of these are affinoid, and write $\mathfrak{V} = \spm \mathfrak{R}$ and $\mathfrak{U}_{j_\nu}= \spm \mathfrak{R}_\nu$ for $\nu \in \set{1,2}$.  Our hypothesis (\ref{eqn:hrhyp}) implies that the restrictions of each $\mathfrak{H}_r^{(k,\ell)}(\mathfrak{M})|_{\mathfrak{U}_{j_\nu}}$ to $\mathfrak{V}$ are naturally identified with the sheaf associated to $H^r(G_K,t^k\bdrk{\ell} \hat{\otimes}_{K,\sigma} \mathfrak{M}(\mathfrak{V}))$.  We pick an admissible affinoid cover of $\mathfrak{U}_{j_1} \cap \mathfrak{U}_{j_2}$ to define the gluing via this identification, and it is clear from the hypothesis (\ref{eqn:hrhyp}) that these are compatible on the triple intersections (by again restricting to affinoids).

If $\mathfrak{W} = \spm \mathfrak{R}$ is any affinoid subdomain of $\mathfrak{X}$ with $\mathfrak{M}(\mathfrak{W})$ free, then we claim that $\mathfrak{H}_r^{(k,\ell)}(\mathfrak{M})|_{\mathfrak{W}}$ is naturally isomorphic to the sheaf associated to the module in (\ref{eqn:hrdef}).  This can be checked locally, so we pick an admissible affinoid cover that refines the cover $\set{\mathfrak{W} \cap \mathfrak{U}_j}$.  For any member $\mathfrak{V}$ of this cover, we again use the hypothesis (\ref{eqn:hrhyp}) to obtain the needed compatibility.
\end{proof}

We break the main theorem up into three parts, which respectively discuss the closed subvarieties of the base associated to de Rham data, behavior of periods on the corresponding stratification, and specializations.  We maintain the same notation and hypotheses throughout.

\begin{thm} \label{thm:stratfam1}

Suppose that $\mathfrak{X}$ is a reduced rigid space over $E$ and $\mathfrak{M}$ is a coherent locally free sheaf equipped with a continuous homomorphism $G_K \rightarrow \End_{\mathfrak{O}_\mathfrak{X}} \mathfrak{M}$.

Then for any de Rham datum $\mathbf{D}$, the points $x \in \mathfrak{X}$ such that $\mathbf{D}_{\mathfrak{M},x} \ge \mathbf{D}$ form a closed analytic subvariety $\mathfrak{S}_\mathbf{D}$, which we give its reduced structure.  If $\mathbf{D} \ge \mathbf{D}'$, then $\mathfrak{S}_\mathbf{D} \subseteq \mathfrak{S}_{\mathbf{D}'}$.  We write $\pi_\mathbf{D}: \mathfrak{S}_\mathbf{D} \rightarrow \mathfrak{X}$ for the inclusion and define $\mathfrak{M}_\mathbf{D} = \pi_\mathbf{D}^*\mathfrak{M}$.

If we range instead over the various Sen data or Hodge-Tate data, we obtain coarser stratifications.  For any Sen datum $\mathbf{D}$, the various Hodge-Tate or de Rham data $\mathbf{D}'$ with $\mathrm{Sen}(\mathbf{D}') = \mathbf{D}$ give stratifications of $\mathfrak{S}_\mathbf{D}$, and the same holds for a Hodge-Tate datum $\mathbf{D}$ and the de Rham data $\mathbf{D}'$ with $\mathrm{HT}(\mathbf{D}')=\mathbf{D}$.

\end{thm}

\begin{proof}
Let $\mathfrak{S}_\mathbf{D}$ be the Zariski closure of the set of $x\in\mathfrak{X}$ with $\mathbf{D}_{\mathfrak{M},x} \ge \mathbf{D}$, let $\mathfrak{Z}$ be any irreducible component of $\mathfrak{S}_\mathbf{D}$, and let $\mathfrak{Z}'$ be any irreducible component of $\mathfrak{X}$ containing $\mathfrak{Z}$.  We apply the second part of Corollary \ref{coro:drdatuminterpglobal} to the closed immersion of $\mathfrak{Z}$ into $\mathfrak{Z}'$ to deduce that $\mathbf{D}_{\mathfrak{M}_\mathfrak{Z}} \ge \mathbf{D}$, where $\mathfrak{M}_\mathfrak{Z}$ is the pullback of $\mathfrak{M}$ to $\mathfrak{Z}$.  Using Proposition \ref{prop:drdatumspecialaff}, we deduce that for any $x \in \mathfrak{Z}$, $\mathbf{D}_{\mathfrak{M},x} \ge \mathbf{D}.$  The containment $\mathfrak{S}_\mathbf{D} \subseteq \mathfrak{S}_{\mathbf{D}'}$ is clear, as is the second claim.
\end{proof}

We next study flatness and base change.  Note that the proof of Lemma \ref{lem:minfin} gives a somewhat explicit description of the possible $\mathfrak{L}_\mathbf{D}^{[i,j]}$ below.  Moreover, $\sd(\mathbf{D})$ is constrained by the rank of $\mathfrak{M}$.

\begin{thm} \label{thm:stratfam2}

For any interval $[i,j] \subseteq \mathbf{Z}$, we define Zariski open (in $\mathfrak{S}_\mathbf{D}$) subvarieties $\mathfrak{L}^{[i,j]}_\mathbf{D} = \mathfrak{S}_\mathbf{D} \setminus \cup_{\mathbf{D}' \in \mathrm{Min}(i,j)} \mathfrak{S}_{\mathbf{D}'}$ for $\mathbf{D}$ with $\supp(\mathbf{D})\subseteq [i,j]$.  These are locally closed in $\mathfrak{X}$, and $\coprod_{\mathbf{D}} \mathfrak{L}^{[i,j]}_\mathbf{D}=\mathfrak{X}$.  We write $\pi_\mathbf{D}^{[i,j]}:\mathfrak{L}_{\mathbf{D}}^{[i,j]} \rightarrow \mathfrak{X}$ for the inclusion.  Let $\mathfrak{M}_{\mathbf{D}}^{[i,j]} = \pi_{\mathbf{D}}^{[i,j]*}\mathfrak{M}$.  Then we have the following.

\begin{enumerate}[{\normalfont (a)}]
	\item If $i \le k < \ell \le j+1$, $r \in \set{0,1}$, and $\mathbf{D} = (\Omega,\Delta)$ has $\supp(\mathbf{D}) \subseteq [i,j]$, then there exists a coherent locally free sheaf $\mathfrak{H}_{(k,\ell)}^r(\mathfrak{M}_\mathbf{D}^{[i,j]})$ of rank $\Delta(k,\ell)$ such that for any affinoid subdomain $\mathfrak{U}\subseteq \mathfrak{L}_\mathbf{D}^{[i,j]}$, there is a canonical isomorphism $\mathfrak{H}_{(k,\ell)}^r(\mathfrak{M}_{\mathbf{D}}^{[i,j]})(\mathfrak{U}) \cong H^r(G_K,t^k\bdrk{\ell} \hat{\otimes}_{K,\sigma} \mathfrak{M}_\mathbf{D}^{[i,j]} (\mathfrak{U})).$
		\item Suppose that $i \le k < \ell \le j+1$, $r \in \set{0,1}$, and $\pi:\mathfrak{Y} \rightarrow \mathfrak{L}_\mathbf{D}^{[i,j]}$ is a map of reduced rigid spaces.  Then for any $y \in \mathfrak{Y}$, $\mathbf{D}_{\pi^*\mathfrak{M}_{\mathbf{D}}^{[i,j]},y}^{[i,j]} = \mathbf{D}$ and the natural map $\pi^*\mathfrak{H}_{(k,\ell)}^r(\mathfrak{M}_{\mathbf{D}}^{[i,j]}) \rightarrow \mathfrak{H}_{(k,\ell)}^r(\pi^*\mathfrak{M}_{\mathbf{D}}^{[i,j]})$ is an isomorphism.
\end{enumerate}

\end{thm}

\begin{proof}
If $x \in \mathfrak{L}_\mathbf{D}^{[i,j]} \cap \mathfrak{L}_{\mathbf{D}'}^{[i,j]}$, then $\mathbf{D}_x^{[i,j]} \ge \mathbf{D}$ and $\mathbf{D}_x^{[i,j]} \ge \mathbf{D}'$, so by definition of the $\mathfrak{L}_\mathbf{D}^{[i,j]}$, we must have $\mathbf{D}_x^{[i,j]} = \mathbf{D} = \mathbf{D}'$.  Moreover, for any $x \in \mathfrak{X}$, $x \in \mathfrak{L}_{\mathbf{D}_x^{[i,j]}}^{[i,j]}$, so $\coprod_\mathbf{D} \mathfrak{L}_\mathbf{D}^{[i,j]} = \mathfrak{X}$.

To prove (a), first let $\mathfrak{U} =\spm\mathfrak{R}$ be any affinoid subdomain of $\mathfrak{L}_\mathbf{D}^{[i,j]}$ such that $\mathfrak{M}_\mathbf{D}^{[i,j]}(\mathfrak{U})$ is free.  We define
\[M^r = H^r(G_K,t^k\bdrk{\ell} \hat{\otimes}_{K,\sigma} \mathfrak{M}_\mathbf{D}^{[i,j]} (\mathfrak{U}))\textrm{ and }M^r(x) = H^r(G_K,t^k\bdrk{\ell} \hat{\otimes}_{K,\sigma} \ol{\mathfrak{M}}_{\mathbf{D},x}^{[i,j]}),\]
for $x \in \mathfrak{U}$ and $r \in \set{0,1}$.  We claim that for $i \le k < \ell \le j+1$, $M^1$ is finite flat of rank $\Delta(k,\ell)$.  It will then follow from Lemma \ref{lem:flat1cocycle}.(a) that the same is true for $M^0$.

For any $x \in \mathfrak{U}$, $M^1(x) \cong M^1\otimes_\mathfrak{R} \kappa(x)$ by Proposition \ref{prop:dr1cocycle}.  By the bound $i \le k < \ell \le j+1$ and the definition of $\mathfrak{L}_\mathbf{D}^{[i,j]}$, we have $\mathbf{D}_x^{[i,j]} = \mathbf{D}$ for each $x$. By Proposition \ref{prop:drdatumglobal}, for every irreducible component $\mathfrak{Z}$ of $\mathfrak{U}$, if $\mathfrak{Z}$ is equipped with the pullback of $\mathfrak{M}^{[i,j]}_\mathbf{D}$, it has de Rham datum $\mathbf{D}$.  Thus the module $M^1$ satisfies the hypotheses of Lemma \ref{lem:flatcrit}, so it is finite flat of rank $\Delta(k,\ell)$.  It now follows from Lemma \ref{lem:flat1cocycle}.(b) that the hypothesis (\ref{eqn:hrhyp}) of Lemma \ref{lem:constructsheaves} holds for $\mathfrak{L}_\mathbf{D}^{[i,j]}$, $\mathfrak{M}_\mathbf{D}^{[i,j]}$, and the pair $(k,\ell)$, which gives (a).  (That $\mathfrak{H}_{(k,\ell)}^r(\mathfrak{M}_\mathbf{D}^{[i,j]})$ is locally free follows from the finite flatness of $M^1$ above and Lemma \ref{lem:flatcrit}.(a).)

For (b), we pick an admissible affinoid cover $\set{\mathfrak{U}_j}$ of $\mathfrak{L}_\mathbf{D}^{[i,j]}$ with $\mathfrak{M}_\mathbf{D}^{[i,j]}(\mathfrak{U}_j)$ free.  For each $y \in \mathfrak{Y}$, pick $j$ so that $y$ maps to $\mathfrak{U}_j$.  We obtain $\mathbf{D}_{\pi^*\mathfrak{M}_{\mathbf{D}}^{[i,j]},y}^{[i,j]} = \mathbf{D}$ from the restriction of $\pi$ to $y \rightarrow \mathfrak{U}_j$, part (a), and Lemma \ref{lem:flat1cocycle}.(b).  We use this and the construction of part (a) to define $\mathfrak{H}_{(k,\ell)}^r(\pi^*\mathfrak{M}_{\mathbf{D}}^{[i,j]})$.

For each $j$, we pick an admissible affinoid cover $\set{\mathfrak{V}_{ij}}$ of $\pi^{-1}(\mathfrak{U}_j)$ for each $j$.  These form an admissible cover of $\mathfrak{Y}$ by $(G_2)$ of \cite[\S9.1.2]{bgr}. We now apply Lemma \ref{lem:flat1cocycle}.(b) to each $\mathfrak{U}_j$ and $\mathfrak{V}_{ij}$, using the flatness in part (a), to deduce that the restriction of the natural map $\pi^*\mathfrak{H}_{(k,\ell)}^r(\mathfrak{M}_{\mathbf{D}}^{[i,j]}) \rightarrow \mathfrak{H}_{(k,\ell)}^r(\pi^*\mathfrak{M}_{\mathbf{D}}^{[i,j]})$ to each $\mathfrak{V}_{ij}$ is an isomorphism, as needed.
\end{proof}

\begin{thm} \label{thm:stratfam3}

Suppose that we have an interval $[i,j] \subseteq \mathbf{Z}$ and a map $\xi:\mathfrak{x}\rightarrow \mathfrak{L}_{\mathbf{D}}^{[i,j]}$, where $\mathfrak{x}=\spm\mathfrak{R}$ for a local Artinian $E$-algebra $\mathfrak{R}$ of finite dimension over $E$ and $\supp(\mathbf{D}) \subseteq [i,j]$.  Then for $i \le k < \ell \le j+1$ we have $\rank_\mathfrak{R} H^r(G_K,t^k\bht^+/t^\ell\bht^+ \otimes_{K,\sigma} \xi^*\mathfrak{M}) = \htd^{(k,\ell)}(\mathbf{D})$ and $\rank_\mathfrak{R}H^r(G_K,t^k\bdrk{\ell} \otimes_{K,\sigma} \xi^*\mathfrak{M})=\Delta(k,\ell)$ for $r \in \set{0,1}$, where these modules are finite flat.  Moreover, for $r,r^+ \in \set{0,1}$,
\begin{equation}\label{eqn:stratbdrbound} \dim_E H^r(G_K,\bdr \otimes_{K,\sigma} \xi^*\mathfrak{M}) \ge \dim_E H^{r^+}(G_K,t^i\bdr^+ \otimes_{K,\sigma} \xi^*\mathfrak{M}) \ge \drd(\mathbf{D})\dim_E \mathfrak{R}.\end{equation}

\end{thm}

\begin{proof}
The first two claims follow from Theorem \ref{thm:stratfam2}.(b).  We deduce (\ref{eqn:stratbdrbound}) from Proposition \ref{prop:bdrdimh0h1}.
\end{proof}

\section{Higher cohomology} \label{sec:highercoh}

The strategy for proving vanishing of higher cohomology of $\bdr \hat{\otimes}_{K,\sigma} \mathfrak{M}$ and $\bht \hat{\otimes}_{K,\sigma} \mathfrak{M}$ (in the notation above) is to first prove vanishing of $H_K$-cohomology and $\Gamma_K$-cohomology in the bounded setting, and then use a continuous cohomology version of the inflation-restriction exact sequence to obtain vanishing of $G_K$-cohomology.  We apply Proposition \ref{prop:profcoh} and Lemma \ref{lem:indlim} to pass to $t^k\bdr^+$ and then $\bdr$.

We remark that Pottharst \cite{pottharst} has proven strong results about higher cohomology of families of $(\varphi,\Gamma)$-modules using algebraic methods.  Due to the functional analytic difficulties that arise when studying modules obtained from completed tensor products of families of Galois representations with rings like $\bdr$, we need to use an estimation approach here instead.

\subsection{Higher $H_K$-cohomology}

We begin by proving a higher dimensional variant of a vanishing result of Sen \cite[Proposition 2]{sen}.  Sen proves the $n=1$ case of the following result.

\begin{thm} \label{thm:hkvanish}

Let $\mathfrak{M}$ be a Banach space over $\mathbf{C}_p$ on which $H_K$ acts continuously by semilinear automorphisms.  Then $H^n(H_K,\mathfrak{M})=0$ for $n>0$.

\end{thm}

\begin{proof}
For a continuous function $\psi: H_K^n \rightarrow \mathfrak{M}$, define $|\psi| = \max_{h_i\in H_K} \|\psi(h_1,\dots,h_n)\|.$  Fix $\delta > 1$.  Given $\epsilon, \epsilon' > 0$, and a normalized cocycle $\psi:H_K^n \rightarrow \mathfrak{M}$ such that $|\psi| \le \epsilon$, we will produce a normalized cocycle $\psi' = \psi - d\beta$ where $|\psi'| \le \epsilon'$ and $d\beta$ is a normalized coboundary such that $|d\beta| \le \delta \epsilon$.  (Recall that a normalized cocycle has the property that $\psi(h_1,\dots,h_n)=0$ if any $h_i=1$.  Every cocycle is equivalent to a normalized cocycle \cite[Lemma 6.1]{em}, and the coboundary of a normalized cochain is normalized.)

Assume that such a construction is possible.  By compactness of $H_K$, any cocycle $\psi=\psi_1$ (which we may assume is normalized) has a bound $\epsilon_1$.  We choose a decreasing sequence $\epsilon_i \to 0$, and in iterating the construction, we set $\epsilon=\epsilon_i,$ $\epsilon'=\epsilon_{i+1}$ at each stage.  We obtain $\psi_i, d\beta_{i-1}$ with $\psi_i = \psi_{i-1} - d\beta_{i-1}$ with $|\psi_i| \le \epsilon_i$ and $|d\beta_{i-1}| \le \delta\epsilon_{i-1}$.  Then the coboundary $\sum_i d\beta_i$ is defined and $\psi_1 = -\sum_i d\beta_i$, as needed.

Now suppose $\psi:H_K^n \rightarrow \mathfrak{M}$, $\epsilon>0$, and $\epsilon'>0$ as above are given.  We claim that there exists an open normal subgroup $H' \subseteq H_K$ such that $\|\psi(h_1,\dots,h_n)\|\le \frac{\epsilon'}{\delta}$ for $h_1,\dots,h_{n-1} \in H_K$ and $h_n \in H'$.  Using \cite[Proposition 1.1.3]{nsw}, let $\set{H_i}_{i \in I}$ be a set of open normal subgroups of $H_K$ such that $\cap_i H_i=1$.  The images $\mathfrak{M}_i=\psi(H_K,\dots,H_K,H_i)$ satisfy $\cap_i \mathfrak{M}_i=\set{0}$ (because $\psi$ is normalized and continuous and $\mathfrak{M}$ is Hausdorff) and are compact.  Let $Y_i = (\mathfrak{M} \setminus B_{\frac{\epsilon'}{\delta}}) \cap \mathfrak{M}_i$, where $B_{\frac{\epsilon'}{\delta}}$ denotes the open ball of radius $\frac{\epsilon'}{\delta}$.  Since $\mathfrak{M} \setminus B_{\frac{\epsilon'}{\delta}}$ is closed, $Y_i$ is compact.  We have $\cap_i Y_i = \emptyset$.  By compactness, some $\cap_{i \in I'} Y_i$ is empty with $I' \subseteq I$ finite, so we may set $H'= \cap_{i \in I'} H_i$.

Let $S$ be a set of coset representatives of $H_K/H'$, and as in \cite[Proposition 2]{sen}, we define an element $z \in \mathbf{C}_p^{H'}$ such that $|z| \le \delta$ and $\tr z = \sum_{s \in S} s(z) = 1$.

We define a normalized cochain $\beta: H_K^{n-1}\rightarrow \mathfrak{M}$ by
\[ \beta(h_1,\dots,h_{n-1}) = (-1)^n\sum_{s \in S} h_1\dots h_{n-1} s(z) \psi(h_1,\dots,h_{n-1},s).\]
Note that we have the bound $|d\beta| \le \delta\epsilon$.  In the following calculation, we write $O(r)$ to denote any element $m \in \mathfrak{M}$ such that $\|m\| \le r$.  Since $\psi$ is a cocycle, we have
\begin{equation} \label{eqn:cocyclerel}
	h_1\psi(h_2,\dots,h_n,s) = \sum_{j=1}^n (-1)^{j+1}\psi(h_1,\dots,h_jh_{j+1},\dots,h_n,s) + (-1)^n\psi(h_1,\dots,h_n)
\end{equation}
We have
\begin{align*}
	(\psi - d\beta)(h_1,\dots,h_n) &= \psi(h_1,\dots,h_n) - h_1\((-1)^n\sum_{s \in S} h_2 \dots h_n s(z) \psi(h_2,\dots,h_n,s)\)\\
	&- \sum_{j=1}^{n-1} (-1)^j\((-1)^n\sum_{s \in S} h_1 \dots h_n s(z) \psi(h_1,\dots, h_jh_{j+1},\dots,h_n,s)\)\\
	&+ (-1)^n\((-1)^n\sum_{s \in S} h_1 \dots h_{n-1} s(z) \psi(h_1,\dots,h_{n-1},s)\).
\end{align*}
By (\ref{eqn:cocyclerel}) and semilinearity of the $H_K$-action, we have
\begin{align*}
	&(\psi - d\beta)(h_1,\dots,h_n) = \psi(h_1,\dots,h_n)\\
	&+ (-1)^{n+1}\sum_{s \in S} h_1h_2 \dots h_n s(z)\\
	&\quad\quad\cdot \(\sum_{j=1}^n (-1)^{j+1}\psi(h_1,\dots,h_jh_{j+1},\dots,h_n,s)+(-1)^n\psi(h_1,\dots,h_n)\)\\
	&+ \sum_{j=1}^{n-1} (-1)^{j+1}\((-1)^n\sum_{s \in S} h_1 \dots h_n s(z) \psi(h_1,\dots, h_jh_{j+1},\dots,h_n,s)\)\\
	&- \sum_{s \in S} h_1 \dots h_{n-1} s(z) \psi(h_1,\dots,h_{n-1},s)\\
	&= \psi(h_1,\dots,h_n) - \sum_{s \in S} h_1h_2 \dots h_n s(z) \psi(h_1,\dots,h_n)\\
	&+ (-1)^{n+1}\sum_{s \in S} h_1h_2 \dots h_n s(z)(-1)^{n+1}\psi(h_1,\dots,h_{n-1},h_ns)\\
	&- \sum_{s \in S} h_1 \dots h_{n-1} s(z) \psi(h_1,\dots,h_{n-1},s).
\end{align*}
Note that $\sum_{s \in S} h_1h_2 \dots h_n s(z) =\tr z=1$.  We are left with
\[\sum_{s \in S} h_1 \dots h_n s(z)\psi(h_1,\dots,h_{n-1},h_ns) - \sum_{s \in S} h_1 \dots h_{n-1} s(z) \psi(h_1,\dots,h_{n-1},s).\]

Let $s'$ and $h'$ be the unique elements of $S$ and $H'$, respectively, such that $h_ns=s'h'$.  Then by $H'$-invariance of $z$, this expression becomes
\[\sum_{s \in S} h_1 \dots h_{n-1} s'(z)\psi(h_1,\dots,h_{n-1},s'h') - \sum_{s \in S} h_1 \dots h_{n-1} s(z) \psi(h_1,\dots,h_{n-1},s),\]
which we may rewrite as
\[\sum_{s' \in S} h_1 \dots h_{n-1} s'(z)\(\psi(h_1,\dots,h_{n-1},s'h') - \psi(h_1,\dots,h_{n-1},s')\)\]
since the $s'$ are a permutation of the $s$.  We use the cocycle relation
\begin{align*}
	&\psi(h_1,\dots,h_{n-1},s'h') = (-1)^{n+1}\big[h_1\psi(h_2,\dots,h_{n-1},s',h')\\
	&+ \sum_{j=1}^{n-1} (-1)^j \psi(h_1,\dots,h_jh_{j+1},\dots,h_{n-1},s',h')+(-1)^{n+1}\psi(h_1,\dots,h_{n-1},s')\big]\\
	&=\psi(h_1,\dots,h_{n-1},s') + O\(\frac{\epsilon'}{\delta}\).
\end{align*}

Therefore
\begin{align*}
	\sum_{s' \in S} & h_1 \dots h_{n-1} s'(z)\(\psi(h_1,\dots,h_{n-1},s'h') - \psi(h_1,\dots,h_{n-1},s')\)\\
	&= \sum_{s' \in S} h_1 \dots h_{n-1} s'(z)O\(\frac{\epsilon'}{\delta}\) = O(\epsilon')
\end{align*}
as needed, using $|z| \le \delta$.
\end{proof}

\subsection{$\Gamma_K$-cohomology and inflation-restriction}

We now use Proposition \ref{prop:profcoh} to find sufficient conditions for vanishing of higher continuous $\Gamma_K$-cohomology.

\begin{prop}\label{prop:gammakcohvanish}

Let $M$ be a continuous $\mathbf{Z}_p[\Gamma_K]$-module, and suppose that there exist $\mathbf{Z}_p[\Gamma_K]$-submodules $M_i \subseteq M$ for $i \in \mathbf{Z}_{\ge 0}$ such that $M_i \subseteq M_{i-1}$, each $M/M_i$ is a discrete $p^\infty$-torsion module, and $M \cong \varprojlim_i M/M_i$ as topological $\mathbf{Z}_p[\Gamma_K]$-modules.  Then $H^n(\Gamma_K,M)=0$ for $n\ge 2$.

\end{prop}

\begin{proof}
Observe that $H^n(\Gamma_K,M/M_i)=0$ for $n\ge 2$ and all $i$ by \cite[Proposition 1.7.7]{nsw}.  For $n\ge 3$, the claim follows from (\ref{eqn:h2gammaexact}).  For $n=2$, we find instead that $\varprojlim_i^1 H^1(\Gamma_K,M/M_i) \cong H^2(\Gamma_K,M),$ so we need only show that $(H^1(\Gamma_K,M/M_i))_i$ has surjective transition maps.  But we have $(H^1(\Gamma_K,M/M_i))_i = (M/(M_i+(\gamma-1)M))_i$ by Lemma \ref{lem:zpcoh}, so the result follows.
\end{proof}

We prove the inflation-restriction exact sequence for continuous cohomology of profinite groups acting on a special class of modules.

\begin{prop}\label{prop:infres}

Suppose that $H \subseteq G$ are profinite groups, with $H$ closed and normal in $G$.  Assume that $M$ is a continuous $G$-module with $H^k(H,M)=0$ for $1 \le k \le n-1$.  Moreover, assume that there exist open $G$-submodules $M_i \subseteq M$ for $i \in \mathbf{Z}_{\ge 0}$ such that $M_i \subseteq M_{i-1}$ for each $i$ and $M \cong \varprojlim_i M/M_i$ as topological $G$-modules.  Then the inflation and restriction maps induce an exact sequence
\begin{equation}\label{eqn:infres} 0 \rightarrow H^n(G/H, M^H) \rightarrow H^n(G,M) \rightarrow H^n(H,M).\end{equation}

\end{prop}

We use dimension shifting, so we need to construct an induced module for continuous cohomology.

\begin{lem} \label{lem:indcoh}

Let $G$ be profinite, and let $M$ and $M_i$ be as in Proposition \ref{prop:infres}.  Define $\ind^G M$ to be the set of continuous maps $\xi: G \rightarrow M$ equipped with the $G$-action $(g\xi)(g') = \xi(g^{-1}g')$ for $g,g' \in G$ and the projective limit topology coming from the identification $\ind^G M \cong \varprojlim_i \ind^G (M/M_i)$, where $\ind^G (M/M_i)$ has the discrete topology.  Then $H^n(G,\ind^G M) = 0$ for $n\ge 1$.  If $H \subseteq G$ is a closed subgroup, then $H^n(H,\ind^G M) = 0$ for $n \ge 1$.  If $H$ is also normal, $H^n(G/H,(\ind^G M)^H)=0$ for $n\ge 1$ as well.

\end{lem}

\begin{proof}
By Proposition \ref{prop:profcoh}, we have an exact sequence
\[0 \rightarrow {\varprojlim_i}^1 H^{n-1}(G,\ind^G (M/M_i)) \rightarrow H^n(G,\ind^G M) \rightarrow \varprojlim_i H^n(G,\ind^G (M/M_i)) \rightarrow 0.\]
The usual vanishing of cohomology of induced modules implies that $H^n(G,\ind^G M) = 0$ for $n \ge 2$.  If $n=1$, the same follows from $H^0(G,\ind^G (M/M_i))=M/M_i$.

Consider the exact sequence
\[0 \rightarrow {\varprojlim_i}^1 H^{n-1}(H,\ind^G (M/M_i)) \rightarrow H^n(H,\ind^G M) \rightarrow \varprojlim_i H^n(H,\ind^G (M/M_i)) \rightarrow 0.\]
By \cite[Proposition 1.3.6.(ii)]{nsw}, $\ind^G (M/M_i)$ is also an induced $H$-module, so we obtain vanishing of $H^n(H,\ind^G M)$ for $n \ge 2$.  (The definition of induced module in \cite{nsw} has the action $(g\xi')(g') = g\xi'(g^{-1}g')$ instead, but we can construct a $G$-equivariant isomorphism between these two definitions by mapping $\xi'$ with $(g\xi')(g') = g\xi'(g^{-1}g')$ to the function $\xi(g) = g^{-1}\xi'(g)$.)  If $n=1$, we need to show that the transition maps $H^0(H,\ind^G (M/M_{i+1})) \rightarrow H^0(H,\ind^G (M/M_i))$ are surjective.  Note that $(\ind^G (M/M_{i+1}))^H$ just consists of the functions that factor through $H\backslash G$, so since any map $H\backslash G \rightarrow M/M_i$ can be lifted (by discreteness) to a map $H\backslash G \rightarrow M/M_{i+1}$, the surjectivity is clear.

If $H$ is normal, we may identify $(\ind^G M)^H$ with $\ind^{G/H} M$, so $H^n(G/H,(\ind^G M)^H)=0$ for $n \ge 1$.
\end{proof}

\begin{proof}[Proof of Proposition \ref{prop:infres}]
For $n=1$, this is well known, so assume $n \ge 2$.

There is a closed $G$-equivariant embedding $M_i \rightarrow \ind^G M_i$ taking $m$ to $g \mapsto g^{-1}m$, and similarly for $M$ and $M/M_i$.  (To see that the embedding is closed, recall from the proof of Lemma \ref{lem:indcoh} that if we define the induced module using $(g\xi')(g') = g\xi'(g^{-1}g')$, the resulting module is isomorphic to the original.  The $G$-equivariant embedding into the new module takes $M, M_i,$ or $M/M_i$ into the constant functions, which are easily seen to form a closed subspace.)  Write $M_i' = (\ind^G M_i)/M_i$ and $M' = (\ind^G M)/M$ with the topologies induced from $\ind^G M_i$ and $\ind^G M$.  The quotient $M'/M_i'$ then carries the discrete topology.  The commutative diagram
\[\xymatrix@R=10pt{ & 0 \ar[d] & 0 \ar[d] & 0 \ar[d] & \\ 0 \ar[r] & M_i \ar[r] \ar[d] & \ind^G(M_i) \ar[r] \ar[d] & M_i' \ar[r] \ar[d] & 0 \\ 0 \ar[r] & M \ar[r] \ar[d] & \ind^G(M) \ar[r] \ar[d] & M' \ar[r] \ar[d] & 0 \\ 0 \ar[r] & M/M_i \ar[r] \ar[d] & \ind^G(M/M_i) \ar[r] \ar[d] & M'/M_i' \ar[r] \ar[d] & 0 \\ & 0 & 0 & 0 & }\]
has first and third column exact by definition.  The horizontal maps are continuous, as are the vertical maps in the left two columns.  The map $M_i' \rightarrow M'$ is continuous and open, since the composite map $\ind^G(M_i) \rightarrow \ind^G(M) \rightarrow M'$ is continuous and open and $M_i'$ has the topology induced from the surjection $\ind^G(M_i)$.  It follows that every map is continuous.  To see that the second column is also exact, note that since $M/M_i$ is discrete, there is a continuous section to $M \rightarrow M/M_i$, so any continuous map $G \rightarrow M/M_i$ can be lifted to one into $M$.  The first two rows are exact by definition of $M_i'$ and $M'$, so the third row is exact by the 9-lemma.  If we take $H$-invariants of this diagram, left-exactness shows that $M^H/M_i^H$, $(\ind^G(M))^H/(\ind^G(M_i))^H$, and $M^{\prime H}/M_i^H$ are discrete.  The system $M/M_i$ has the Mittag-Leffler property, so if we look at the morphism from the middle row of the diagram to the limit over $i$ of the third row, we find that $M' \cong \varprojlim_i M'/M_i'$.

We have an exact sequence of $G$-modules
\begin{equation}\label{eqn:exactgh} 0 \rightarrow M \rightarrow \ind^G M \rightarrow M' \rightarrow 0.\end{equation}
We construct a topological section to $\ind^G M \rightarrow M'$ by choosing a compatible family of sections $M'/M_i' \rightarrow \ind^G (M/M_i)$.  (Note that since the topology on $M'/M_i'$ is discrete, we may do this by picking an arbitrary section for $i=0$ and then choosing arbitrary liftings when incrementing $i$.)  By the assumption $H^1(H,M)=0$, we also have an exact sequence
\begin{equation}\label{eqn:exactgmodh} 0 \rightarrow M^H \rightarrow (\ind^G M)^H \rightarrow M^{\prime H} \rightarrow 0.\end{equation}
We claim that there is again a continuous topological section to the surjection.  For this, observe that by the above, $(\ind^G M_i)^H$ (resp.\ $M_i^{\prime H}$) is open in $(\ind^G M)^H$ (resp.\ $M^{\prime H}$), and $\cap_i (\ind^G M_i)^H = 0$ (resp.\ $\cap_i M_i^{\prime H}=0$).  Since $(\ind^G M)^H$ (resp.\ $M^{\prime H}$) is closed in $\ind^G M$ (resp.\ $M'$), we have $(\ind^G M)^H \cong \varprojlim_i (\ind^G M)^H/(\ind^G M_i)^H$ (resp.\ $M^{\prime H} \cong \varprojlim_i M^{\prime H}/M_i^{\prime H}$).  We pick compatible sections $M^{\prime H}/M_i^{\prime H} \rightarrow (\ind^G M)^H/(\ind^G M_i)^H$ to construct a continuous section.  

From vanishing of cohomology in Lemma \ref{lem:indcoh} and the exact sequences (\ref{eqn:exactgh}) and (\ref{eqn:exactgmodh}), we obtain isomorphisms
\begin{align*} H^n(G,M') \iso H^{n+1}(G,M), H^n(H,M') &\iso H^{n+1}(H,M),\\  \textrm{ and }H^n(G/H,M^{\prime H}) &\iso H^{n+1}(G/H,M^H)\end{align*}
for $n \ge 1$.

To obtain (\ref{eqn:infres}), we induct on $n$; the base case is the usual inflation-restriction sequence for 1-cocycles.  If we assume the result for $n$, the upper row of
\[\xymatrix@R=15pt{ 0 \ar[r] & H^n(G/H,M^{\prime H}) \ar[r] \ar[d]^{\begin{sideways}$\sim$\end{sideways}} & H^n(G,M') \ar[r]\ar[d]^{\begin{sideways}$\sim$\end{sideways}}  & H^n(H,M') \ar[d]^{\begin{sideways}$\sim$\end{sideways}} \\ 0 \ar[r] & H^{n+1}(G/H,M^H) \ar[r] & H^{n+1}(G,M) \ar[r] & H^{n+1}(H,M)}\]
is exact by the inductive hypothesis using the isomorphisms $H^k(H,M') \iso H^{k+1}(H,M) = 0$ for $k=1,\dots,n-1$.  The downward maps are isomorphisms, so the lower row is exact as well.
\end{proof}

\subsection{Vanishing of cohomology for bounded and unbounded periods}

We may now prove a vanishing theorem for continuous cohomology of Hodge-Tate and bounded de Rham periods.

\begin{thm}\label{thm:gkvanishbdd}

Suppose that $E$ and $K$ are finite extensions of $\mathbf{Q}_p$ such that $E$ contains the normal closure of $K$, $\mathfrak{R}$ is a Noetherian $E$-Banach algebra, and the finitely generated $\mathfrak{R}$-Banach module $\mathfrak{M}$ is equipped with a continuous $\mathfrak{R}$-linear action of $G_K$.  Fix $\sigma \in \Sigma$ as before.  Then for any $n \ge 2$, we have $H^n(G_K,t^k\bdrk{\ell}\hat{\otimes}_{K,\sigma} \mathfrak{M})=0$ for $k \le \ell \in \mathbf{Z}$, including the special case $H^n(G_K,\mathbf{C}_p(k)\hat{\otimes}_{K,\sigma} \mathfrak{M})=0$ for $k \in \mathbf{Z}$.

\end{thm}

\begin{proof}
Note that $\mathbf{C}_p(j)\hat{\otimes}_{K,\sigma} \mathfrak{M}$ is an $H_K$-semilinear $\mathbf{C}_p$-Banach space.  Write $N = (\mathbf{C}_p(j)\hat{\otimes}_{K,\sigma} \mathfrak{M})^{H_K}$ for any fixed $j \in \mathbf{Z}$.  Let $N_i$ be the open ball in $N$ of radius $p^{-i}$ centered at 0.  Then $N = \varprojlim_i N/N_i$ and each $N/N_i$ is a discrete $p$-torsion module.

By Theorem \ref{thm:hkvanish}, $H^n(H_K,\mathbf{C}_p(j)\hat{\otimes}_{K,\sigma} \mathfrak{M})=0$ for $n \ge 1$.  It follows from Proposition \ref{prop:infres} that $H^n(G_K,\mathbf{C}_p(j)\hat{\otimes}_{K,\sigma} \mathfrak{M}) \cong H^n(\Gamma_K,N)$.  By Proposition \ref{prop:gammakcohvanish}, $H^n(\Gamma_K,N)=H^n(G_K,\mathbf{C}_p(j)\hat{\otimes}_{K,\sigma} \mathfrak{M})=0$ for $n \ge 2$.  This handles the $\mathbf{C}_p(j)$ case, which is also the $k-\ell=1$ case for bounded de Rham periods.  For larger $k-\ell$, we obtain the result using induction via the long exact sequence in cohomology associated to
\[0 \rightarrow \mathbf{C}_p(k-1)\hat{\otimes}_{K,\sigma} \mathfrak{M} \rightarrow t^\ell\bdrk{k}\hat{\otimes}_{K,\sigma} \mathfrak{M} \rightarrow t^{\ell}\bdrk{k-1}\hat{\otimes}_{K,\sigma} \mathfrak{M} \rightarrow 0.\]
As noted in Remark \ref{remark:hkvanish}, a continuous section to the surjection exists.
\end{proof}

We pass to the limit using Proposition \ref{prop:profcoh} and Lemma \ref{lem:indlim}.

\begin{thm}\label{thm:gkvanishdr}

Maintain the hypotheses of Theorem \ref{thm:gkvanishbdd}.  We have
\[H^n(G_K,t^k\bdr^+ \hat{\otimes}_{K,\sigma} \mathfrak{M})=0\textrm{ and }H^n(G_K,\bdr \hat{\otimes}_{K,\sigma} \mathfrak{M})=0\]
for $n \ge 2$ and $k \in \mathbf{Z}$.

\end{thm}

\begin{proof}
Note that $t^k\bdr^+ \hat{\otimes}_{K,\sigma} \mathfrak{M} \cong \varprojlim_\ell (t^k\bdrk{\ell} \hat{\otimes}_{K,\sigma} \mathfrak{M})$ by (\ref{eqn:switchprojlim}).  By Proposition \ref{prop:profcoh} and the continuous sections from Remark \ref{remark:hkvanish}, we have short exact sequences
\begin{align*} 0 \rightarrow {\varprojlim_\ell}^1 H^{n-1}(G_K,t^k\bdrk{\ell}\hat{\otimes}_{K,\sigma} \mathfrak{M}) &\rightarrow H^n(G_K,t^k\bdr^+ \hat{\otimes}_{K,\sigma} \mathfrak{M})\\
	&\rightarrow \varprojlim_\ell H^n(G_K,t^k\bdrk{\ell}\hat{\otimes}_{K,\sigma} \mathfrak{M})\rightarrow 0.
\end{align*}
For $n \ge 3$, the vanishing of the middle term follows from Theorem \ref{thm:gkvanishbdd}.  For $n=2$, we use the exact sequence (\ref{eqn:exactsnake}) to see that the transition maps of the left-hand term are surjective.  The result for $\bdr$ follows from Lemma \ref{lem:indlim}.
\end{proof}

\appendix\section{Approach via invariants} \label{sec:appendix}

In this section, we briefly describe an alternative approach to the results of the preceding sections by replacing the base change result Proposition \ref{prop:dr1cocycle} with a weaker statement for invariants.  This may be valuable in situations where 1-cocycles are poorly behaved.

\begin{lem} \label{lem:regbound}

Let $R$ be a regular local ring with fraction field $K$ and residue field $\kappa(R)$, let $M$ and $N$ be projective $R$-modules, and let $\psi: M \rightarrow N$ be an $R$-linear homomorphism.  Then if $\dim_K \ker (\psi \otimes_R K)$ is finite,
\begin{equation} \label{eqn:regbound} \dim_{\kappa(R)} \ker(\psi \otimes_R \kappa(R)) \ge \dim_K \ker (\psi \otimes_R K),\end{equation}
where the left-hand side may be infinite, and if $\dim_K \ker (\psi \otimes_R K)$ is infinite, then $\dim_{\kappa(R)} \ker(\psi \otimes_R \kappa(R))$ is infinite as well.

\end{lem}

\begin{proof}

Note that $M$ and $N$ are free by Kaplansky's theorem \cite{kaplansky}.

Let $(x_1,\dots,x_n)$ be a regular system of parameters of $R$.  Define
\[R_i = R/(x_1,\dots,x_i) \textrm{ and } K_i = \Frac(R/(x_1,\dots,x_i)).\]
Then $K_0 = K$ and $K_n = \kappa(R)$.  We will prove for each $i$ that either
\[\dim_{K_{i+1}} \ker(\psi \otimes_R K_{i+1}) \ge \dim_{K_i} \ker (\psi \otimes_R K_i),\]
or $\dim_{K_{i+1}} \ker(\psi \otimes_R K_{i+1})$ is infinite.  The inequality (\ref{eqn:regbound}) will follow.

The localization $S_i = (R_i)_{(x_{i+1})}$ of the regular local ring $R_i$ is a DVR, which is hereditary, so all submodules of projective $S_i$-modules are projective.  We have $\Frac(S_i) = K_i$.  Since the $S_i$-module $K_i$ is flat, there is a natural identification
\begin{equation}\label{eqn:kiswitch} \ker(\psi \otimes_R S_i) \otimes_{S_i} K_i \iso \ker(\psi \otimes_R K_i)\end{equation}
by Lemma \ref{lem:niceend}.  Since $\ker(\psi \otimes_R S_i)$ is a submodule of a projective module, it is projective.  From (\ref{eqn:kiswitch}) and Kaplansky's theorem, it follows that $\ker(\psi \otimes_R S_i)$ is free of rank 
\begin{equation} \label{eqn:regrankdim} \rank_{S_i} \ker(\psi \otimes_R S_i)=\dim_{K_i}\ker(\psi \otimes_R K_i),\end{equation}
possibly infinite.

We have an exact sequence of $S_i$-modules
\begin{equation} \label{eqn:siseq} 0 \rightarrow \ker(\psi \otimes_R S_i) \rightarrow M \otimes_R S_i \rightarrow N \otimes_R S_i.\end{equation}
Since $S_i$ is hereditary and $N \otimes_R S_i$ is projective, $\im(\psi \otimes_R S_i)$ is projective as well.  By projectivity and Kaplansky's theorem again, we obtain a splitting
\[M \otimes_R S_i \cong \ker(\psi \otimes_R S_i) \oplus \coim(\psi \otimes_R S_i)\]
of $M\otimes_R S_i$ as the direct sum of two free $S_i$ modules, with $\psi \otimes_R S_i$ vanishing on the first factor and injective on the second.  We have $S_i/(x_{i+1}) = K_{i+1}$.  The splitting and (\ref{eqn:regrankdim}) imply that
\[\dim_{K_{i+1}} \ker(\psi \otimes_R K_{i+1}) \ge \rank_{S_i} \ker(\psi \otimes_R S_i) = \dim_{K_i}\ker(\psi \otimes_R K_i),\]
where both sides are possibly infinite, but the left hand side is infinite if the right hand side is infinite.
\end{proof}

This lemma allows one to directly work with the invariants rather than 1-cocycles when studying base change.  One first takes a resolution of singularities of the rigid analytic space so that the local rings become regular.  One reference for the existence of such a resolution is the work of Temkin \cite[Theorem 5.2.2]{tem}; the result is originally due to Bierstone-Milman \cite[\S I.(0.1).(2)]{bm}.  Then Lemma \ref{lem:regbound} and Theorem \ref{thm:dr1cocycle} give opposite inequalities for the dimension of periods (rather than 1-cocycles) before and after specialization, which can then be used in the subsequent arguments as above.  At the end, it is easy to deduce consequences for the original affinoid space as the fibers of the resolution have constant Galois representations when specializing the pullback of the module.  The statements regarding base change on strata are slighly weaker with this approach due to the need to exclude the singular locus.

\bibliography{InterpolatingPeriods}

\begin{thebibliography}{10}

\bibitem{bch}
Jo{\"e}l Bella{\"{\i}}che and Ga{\"e}tan Chenevier.
\newblock Families of {G}alois representations and {S}elmer groups.
\newblock {\em Ast\'erisque}, (324):xii+314, 2009.

\bibitem{bell}
Rebecca Bellovin.
\newblock {$p$}-adic {H}odge theory in rigid analytic families.
\newblock {\em Algebra Number Theory}, 9(2):371--433, 2015.

\bibitem{bc}
Laurent Berger and Pierre Colmez.
\newblock Familles de repr\'esentations de de {R}ham et monodromie
  {$p$}-adique.
\newblock {\em Ast\'erisque}, (319):303--337, 2008.

\bibitem{bm}
Edward Bierstone and Pierre~D. Milman.
\newblock Canonical desingularization in characteristic zero by blowing up the
  maximum strata of a local invariant.
\newblock {\em Invent. Math.}, 128(2):207--302, 1997.

\bibitem{bgr}
Siegfried Bosch, Ulrich G{\"u}ntzer, and Reinhold Remmert.
\newblock {\em Non-{A}rchimedean analysis}, volume 261 of {\em Grundlehren der
  Mathematischen Wissenschaften}.
\newblock Springer-Verlag, Berlin, 1984.

\bibitem{bourbakica}
Nicolas Bourbaki.
\newblock {\em Elements of mathematics. {C}ommutative algebra}.
\newblock Hermann, Paris, 1972.
\newblock Translated from the French.

\bibitem{bcon}
Olivier Brinon and Brian Conrad.
\newblock C{M}{I} summer school notes on $p$-adic {H}odge theory.
\newblock {\em preprint}, 2009.

\bibitem{ch}
Ga{\"e}tan Chenevier and Michael Harris.
\newblock Construction of automorphic {G}alois representations {I}{I}.
\newblock {\em Cambridge Math. Journal}, 1, 2013.

\bibitem{con}
Brian Conrad.
\newblock Irreducible components of rigid spaces.
\newblock {\em Ann. Inst. Fourier (Grenoble)}, 49(2):473--541, 1999.

\bibitem{dl}
Hansheng Diao and Ruochuan Liu.
\newblock The eigencurve is proper.
\newblock {\em Duke Math. J.}, 165(7):1381--1395, 2016.

\bibitem{ding2}
Yiwen Ding.
\newblock Companion points and locally analytic socle for $\mathrm{GL}_2({L})$.
\newblock {\em Preprint}, 2016.

\bibitem{ding3}
Yiwen Ding.
\newblock Formes modulaires $p$-adiques sur les courbes de {S}himura unitaires
  et compatibilité local-global.
\newblock {\em Preprint}, 2016.

\bibitem{ding1}
Yiwen Ding.
\newblock $\mathcal{L}$-invariants, partially de {R}ham families, and
  local-global compatibility.
\newblock {\em Annales de l'Institut Fourier}, 67(4):1457--1519, 2017.

\bibitem{em}
Samuel Eilenberg and Saunders MacLane.
\newblock Cohomology theory in abstract groups. {I}.
\newblock {\em Ann. of Math. (2)}, 48:51--78, 1947.

\bibitem{emerton}
Matthew Emerton.
\newblock Locally analytic vectors in representations of locally {$p$}-adic
  analytic groups.
\newblock {\em Mem. Amer. Math. Soc.}, 248(1175):iv+158, 2017.

\bibitem{fontaine}
Jean-Marc Fontaine.
\newblock Le corps des p\'eriodes {$p$}-adiques.
\newblock {\em Ast\'erisque}, (223):59--111, 1994.
\newblock With an appendix by Pierre Colmez, P\'eriodes $p$-adiques
  (Bures-sur-Yvette, 1988).

\bibitem{hltt}
Michael Harris, Kai-Wen Lan, Richard Taylor, and Jack Thorne.
\newblock On the rigid cohomology of certain {S}himura varieties.
\newblock {\em Res. Math. Sci.}, 3:Paper No. 37, 308, 2016.

\bibitem{hellmann}
Eugen Hellmann.
\newblock Families of {$p$}-adic {G}alois representations and
  {$(\varphi,\Gamma)$}-modules.
\newblock {\em Comment. Math. Helv.}, 91(4):721--749, 2016.

\bibitem{iz}
Adrian Iovita and Alexandru Zaharescu.
\newblock Galois theory of {$B_{\rm dR}^+$}.
\newblock {\em Compositio Math.}, 117(1):1--31, 1999.

\bibitem{iz2}
Adrian Iovita and Alexandru Zaharescu.
\newblock Generating elements for {$B_{\rm dR}^+$}.
\newblock {\em J. Math. Kyoto Univ.}, 39(2):233--248, 1999.

\bibitem{jorza}
Andrei Jorza.
\newblock {$p$}-adic families and {G}alois representations for {$\rm GSp(4)$}
  and {$\rm GL(2)$}.
\newblock {\em Math. Res. Lett.}, 19(5):987--996, 2012.

\bibitem{kaplansky}
Irving Kaplansky.
\newblock Projective modules.
\newblock {\em Ann. of Math (2)}, 68:372--377, 1958.

\bibitem{kliu}
Kiran Kedlaya and Ruochuan Liu.
\newblock On families of ({$\varphi$}, {$\Gamma$})-modules.
\newblock {\em Algebra Number Theory}, 4(7):943--967, 2010.

\bibitem{kpx}
Kiran Kedlaya, Jonathan Pottharst, and Liang Xiao.
\newblock Cohomology of arithmetic families of {$(\varphi,\Gamma)$}-modules.
\newblock {\em J. Amer. Math. Soc.}, 27(4):1043--1115, 2014.

\bibitem{kisin}
Mark Kisin.
\newblock Overconvergent modular forms and the {F}ontaine-{M}azur conjecture.
\newblock {\em Invent. Math.}, 153(2):373--454, 2003.

\bibitem{rliu}
Ruochuan Liu.
\newblock Semistable periods of finite slope families.
\newblock {\em Algebra Number Theory}, 9(2):435--458, 2015.

\bibitem{liu2}
Ruochuan Liu.
\newblock Triangulation of refined families.
\newblock {\em Comment. Math. Helv.}, 90(4):831--904, 2015.

\bibitem{luu}
Martin Luu.
\newblock Deformation theory and local-global compatibility of {L}anglands
  correspondences.
\newblock {\em Mem. Amer. Math. Soc.}, 238(1123):vii+101, 2015.

\bibitem{mw}
Barry Mazur and Andrew Wiles.
\newblock On {$p$}-adic analytic families of {G}alois representations.
\newblock {\em Compositio Math.}, 59(2):231--264, 1986.

\bibitem{nsw}
J{\"u}rgen Neukirch, Alexander Schmidt, and Kay Wingberg.
\newblock {\em Cohomology of number fields}, volume 323 of {\em Grundlehren der
  Mathematischen Wissenschaften}.
\newblock Springer-Verlag, Berlin, second edition, 2008.

\bibitem{gs}
Cristina P\'erez-Garc\'ia and Wilhelmus~H. Schikhof.
\newblock {\em Locally convex spaces over non-{A}rchimedean valued fields},
  volume 119 of {\em Cambridge Studies in Advanced Mathematics}.
\newblock Cambridge University Press, Cambridge, 2010.

\bibitem{pottharst}
Jonathan Pottharst.
\newblock Analytic families of finite-slope {S}elmer groups.
\newblock {\em Algebra Number Theory}, 7(7):1571--1612, 2013.

\bibitem{schneider}
Peter Schneider.
\newblock {\em Nonarchimedean functional analysis}.
\newblock Springer Monographs in Mathematics. Springer-Verlag, Berlin, 2002.

\bibitem{sen3}
Shankar Sen.
\newblock Continuous cohomology and {$p$}-adic {G}alois representations.
\newblock {\em Invent. Math.}, 62(1):89--116, 1980.

\bibitem{sen}
Shankar Sen.
\newblock The analytic variation of {$p$}-adic {H}odge structure.
\newblock {\em Ann. of Math. (2)}, 127(3):647--661, 1988.

\bibitem{sen2}
Shankar Sen.
\newblock An infinite-dimensional {H}odge-{T}ate theory.
\newblock {\em Bull. Soc. Math. France}, 121(1):13--34, 1993.

\bibitem{shahdiss}
Shrenik Shah.
\newblock {\em $p$-adic approaches to the Langlands program}.
\newblock PhD thesis, 2014.

\bibitem{skinner}
Christopher Skinner.
\newblock A note on the {$p$}-adic {G}alois representations attached to
  {H}ilbert modular forms.
\newblock {\em Doc. Math.}, 14:241--258, 2009.

\bibitem{su}
Christopher Skinner and Eric Urban.
\newblock The {I}wasawa main conjectures for {$\rm GL_2$}.
\newblock {\em Invent. Math.}, 195(1):1--277, 2014.

\bibitem{tate}
John Tate.
\newblock $p$-divisible groups.
\newblock In {\em Proc. {C}onf. {L}ocal {F}ields ({D}riebergen, 1966)}, pages
  158--183. Springer, Berlin, 1967.

\bibitem{tem}
Michael Temkin.
\newblock Functorial desingularization of quasi-excellent schemes in
  characteristic zero: the nonembedded case.
\newblock {\em Duke Math. J.}, 161(11):2207--2254, 2012.

\bibitem{varma}
Ila Varma.
\newblock Crystallinity of {G}alois representations associated to regular
  algebraic cuspidal automorphic representations of $\mathrm{GL}_n$.
\newblock {\em Preprint}, 2018.

\end{thebibliography}
\bibliographystyle{plain}

\end{document}